\documentclass[11pt,draft, reqno]{amsart}

\usepackage[margin=10mm]{geometry}
\usepackage{amssymb, amsmath, amsthm, amscd, soul, ulem}

\usepackage[T1]{fontenc}
\usepackage{eucal,mathrsfs}
\usepackage{color}

\renewcommand{\leq}{\leqslant}
\renewcommand{\geq}{\geqslant}
\renewcommand{\le}{\leqslant}
\renewcommand{\ge}{\geqslant}

\definecolor{mno}{rgb}{0.5,0.1,0.5}

\newcommand{\R}{\mathbb R}

\newcommand{\Pp}{\mathbb P}
\newcommand{\Ee}{\mathbb E}

\newcommand{\I}{\mathbf 1}
\newcommand{\w}{\omega}
\newcommand{\N}{\mathbb{N}}
\newcommand{\Z}{\mathbb Z}
\newcommand{\sL}{\mathcal{L}}
\newcommand{\sK}{\mathcal K}

\newcommand{\E}{\mathscr{E}}
\newcommand{\F}{\mathscr{F}}

\newcommand{\LL}{\mathcal{L}}

\newtheorem{theorem}{Theorem}[section]
\newtheorem{lemma}[theorem]{Lemma}
\newtheorem{proposition}[theorem]{Proposition}

\theoremstyle{definition}

\newtheorem{example}[theorem]{Example}
\newtheorem{remark}[theorem]{Remark}

\numberwithin{equation}{section}

\def\wt{\widetilde}

\begin{document}
\allowdisplaybreaks
\title[Time-dependent random
conductance models with stable-like jumps] {Quantitative
homogenization
on time-dependent random conductance models with stable-like jumps}

\author{Xin Chen,\quad Zhen-Qing Chen,\quad Takashi Kumagai  \quad \hbox{and}\quad Jian Wang}
\thanks{{\it X.\ Chen:}
   School of Mathematical Sciences, Shanghai Jiao Tong University, 200240 Shanghai, P.R. China. \texttt{chenxin217@sjtu.edu.cn}}
\thanks{{\it Z.-Q. Chen:} Department of Mathematics, University of Washington, Seattle,
WA 98195, USA. \texttt{zqchen@uw.edu}}
   \thanks{{\it T.\ Kumagai:}
Department of Mathematics,
Waseda University, Tokyo 169-8555, Japan.
\texttt{t-kumagai@waseda.jp}}
  \thanks{{\it J.\ Wang:}
    School of Mathematics and Statistics \& Key Laboratory of Analytical Mathematics and Applications (Ministry of Education) \& Fujian Provincial Key Laboratory
of Statistics and Artificial Intelligence, Fujian Normal University, 350007 Fuzhou, P.R. China. \texttt{jianwang@fjnu.edu.cn}}

\date{}

\maketitle

\begin{abstract} We establish quantitative
homogenization
results for time-dependent random conductance
models with stable-like
long range
jumps on $\Z^d$, where the transition probability from
$x$ to $y$ is given by $w_{t, x,y}|x-y|^{-d-\alpha}$ with $\alpha\in (0,2)$. In particular, time-dependent random coefficients $\{w_{t,x,y}: t\in \R_+, (x,y)\in E\}$ are uniformly bounded from above
(but may be degenerate), and satisfy the Kolmogorov continuous condition, where $E=\{(x, y):  x \not= y   \in \Z^d\}$ is the set of all
unordered pairs on $\Z^d$. The proofs are based on $L^2$-estimates and energy estimates for solutions to
regional
 parabolic equations and multi-scale Poincar\'e inequalities associated with time-dependent symmetric stable-like random walks with random coefficients.

\medskip


\medskip

\end{abstract}
\allowdisplaybreaks

\section{Introduction and main result}\label{section1}
The primary purpose of this paper is to investigate
convergence rates in
the $L_2$-sense for a family of parabolic $\alpha$-stable-like operators $\partial_t-\LL_t^{k,\w}$ on $k^{-1}\Z^d$
with rapidly oscillating and time-dependent random coefficients, where
\begin{equation}\label{e:1.1}
\LL_t^{k,\w} f(x):=
k^{-d}\sum_{y\in k^{-1}\Z^d:y\neq x}\left(f(y)-f(x)\right)\frac{w_{k^\alpha t,kx,ky}(\w)}{|x-y|^{d+\alpha}}
\end{equation} with $d\ge1$,
$\alpha\in (0,2)$ and $\{w_{t,x,y}(\w):t\ge0,x,y\in \Z^d,\w\in \Omega\}$ being non-negative random coefficients on some probability space $(\Omega, \Pp)$.
Here and in what follows, we use $:=$ as a way of definition.

Throughout this paper, the following assumption is imposed
in force.

\medskip

\noindent{{\bf Assumption (H)}}:
{\it
Let $\R_+:=[0,\infty)$, and
 $E=\{(x, y):  x \not= y   \in \Z^d\}$ be the collection of all unordered distinct pairs on $\Z^d$. Suppose that  $\{w_{t,x,y}: t\in \R_+, (x,y)\in E\}$ are
time-dependent non-negative random variables satisfying the following conditions{\rm:}

\begin{itemize}

\item [(i)]
$w_{t,x,y}=w_{t,y,x}$ with $\Ee[w_{t,x,y}]=\sK (t)$
for every
$t\in \R_+$ and $(x,y)\in E$. Moreover, there exist positive constants
$K_1$, $K_2$ and $\sK$ such that
\begin{equation}\label{a1-1-0}
K_1\le \sK (t)\le K_2,\quad t\in \R_+
\end{equation} and
\begin{equation}\label{a1-1-0a}
 \lim_{t \to \infty}\frac{1}{t}\int_0^t\left|\sK (s)-\sK \right|^2\,ds=0.
\end{equation}

\item [(ii)] There exists a constant $C_1\ge1$ such that
$w_{t,x,y}\le C_1$ for all $t\in \R_+$ and
$(x,y)\in E$.

\item [(iii)] There exist positive constants $p$, $\eta$ and $C_2$ such that for all $s,t\in \R_+$ with $|t-s|\le 1$ and
$(x,y)\in E$,
\begin{equation}\label{a1-1-1}
\Ee\left[|w_{t,x,y}-w_{s,x,y}|^p\right]\le C_2|t-s|^{1+\eta}.
\end{equation}

\item [(iv)] For every $n\in \mathbb{N}_+:=\{1,2,\cdots\}$, $0
\le t_1\le t_2\le \cdots \le t_n$
and
mutually distinct pairs $(x_1,y_1),\cdots,(x_n,y_n)\in E$,
random variables $\{w_{t_i,x_i,y_i}\}_{1\le i \le n}$
are independent.
 \end{itemize}}

\smallskip

We note that
under {\bf Assumption (H)},
 the time-dependent random coefficients $\{w_{t,x,y}: t\in \R_+, (x,y)\in E\}$ are not necessarily strictly positive, and are not independent for all distinct pairs $(t,(x,y))\in \R_+\times E$.
For
notational convenience, we set  $w_{t, x, x}=0$ for all  $t>0$ and $x\in \R^d$.
First,
we note that in {\bf Assumption (H)}(ii) the
random coefficients
 $\{w_{t,x,y}: t\in \R_+, (x,y)\in E\}$ are only required to be uniformly bounded from above and
may be degenerate.
Similar assumption is also taken in \cite{CCKW3} for time-independent random conductance models
with stable-like jumps. By {\bf Assumption (H)}(iii),  the  random coefficients
$\{w_{t,x,y}: t\in \R_+, (x,y)\in E\}$ fulfill the Kolmogorov continuous condition
in  the time variable $t\in \R_+$. On the other hand, by {\bf Assumption (H)}(iv),
 the following space-time mixing condition holds: there is a constant $C_3>0$ such that for all $\psi_1,\psi_2\in
C_b(\R_+)$, $s, t\in \R_+$ and  $x, y,  v, w \in \R^d$,
$${\rm Cov} \left( \psi_1 (w_{s, v, w}), \psi_2 (w_{t+s, x+v, y+w} ) \right)
\le C_3\|\psi_1\|_\infty\|\psi_2\|_\infty
\I_{\{
(x,y)=(0,0)
\}.}
$$
Unlike the paper by
 Amstrong-Bordas-Mourrat \cite{ABM}   for the diffusive case,
we do  not assume the   random coefficients
  $\{w_{t,x,y}: t\in \R_+, (x,y)\in E\}$    satisfy the finite range dependence property in time variable $t\in \R_+$.
  In other words,   $\{w_{t,x,y}: t\in \R_+, (x,y)\in E\}$ and  $\{w_{s,x,y}: t\in \R_+, (x,y)\in E\}$ can be dependent for any
  $t\not= s$.

Let $\mu$ be a counting measure on $\Z^d$. For any $\alpha\in (0,2)$, $t\in \R_+$ and $\w\in \Omega$,
we define a time-dependent stable-like operator as follows
\begin{equation}\label{e6-1}
\LL_t^\w f(x):=\sum_{y\in \Z^d:y\neq x}\left(f(y)-f(x)\right)\frac{w_{t,x,y}(\w)}{|x-y|^{d+\alpha}},\quad f\in L^2(\Z^d;\mu).
\end{equation}
It determines a
continuous-time time-inhomogeneous  strong Markov  process $\{X_t^\w\}_{t\ge 0}$ on $\Z^d$.
For each integer $k\geq 1$, the scaled process
$\{k^{-1}X_{k^\alpha t}^{\w}\}_{t\ge 0}$ on $k^{-1}\Z^d$ has infinitesimal generator $\LL_t^{k,\w} $
given by \eqref{e:1.1}.
Define
\begin{equation}\label{e:1.2}
\begin{split}
\bar \sL f(x):=&\lim_{\varepsilon \to 0}\int_{\{|y-x|\ge \varepsilon\}} \left(f(y)-f(x)\right)
\frac{\sK }{|x-y|^{d+\alpha}}\,dy,\quad
f\in C^2_c (\R^d) , \end{split}
\end{equation}
which will be served as the limit of the scaled operators $\{\LL_t^{k,\w}\}_{k\ge1}$ as $k\to \infty$ under {\bf Assumption (H)}
(partly due to  \eqref{a1-1-0a}).

Fix
$T>0$,
$g\in C_c^2 (\R^d)$
and
$\beta\in (0,\infty]$. Let
\begin{equation}\label{e:1.4}\begin{split}
\mathscr{S}_{\beta, g,T}
:=\bigg\{
 h:\,&[0,T]\times \R^d \to \R\bigg|\, \text{there  exists}\ f\in C^1([0,T];
C_b^3 (\R^d)) \hbox{ with } f(0,x)=g(x) \text{ such\ that }   \\
&h(t,x)=\frac{\partial}{\partial t}f(t,x)-\bar \LL f(t,\cdot)(x)
\text{  for\ all } (t,x)\in (0,T]\times \R^d ,
 \hbox {and there is a constant } C_0>0 \hbox{ so that } \\
 &
\sup_{t\in [0,T], x\in \R^d,0\le i\le 3}
\Big[
{\Big(|\nabla_x ^i f(t,x)|+\Big|\frac{\partial f(t,x)}{\partial t}\Big|\Big)}
(1 \vee (C_0|x|))^{d+\beta}
 \Big]
<\infty  \bigg\}.
\end{split}\end{equation}
In particular, for $\beta=\infty$,
\begin{align*}
\mathscr{S}_{\infty, g,T}
:=\bigg\{
 h:\,&[0,T]\times \R^d \to \R\bigg|\, \text{there  exists}\ f\in C^1([0,T];
 C_c^3(\R^d))     \hbox{ with } f(0,x)=g(x)   \text{ such\ that }  \nonumber \\
&h(t,x)=\frac{\partial}{\partial t}f(t,x)-\bar \LL f(t,\cdot)(x)
\text{  for\ all } (t,x)\in (0,T]\times \R^d   \bigg\}.
\end{align*}
It is clear that
$\mathscr{S}_{\beta_1 g,T} \subset \mathscr{S}_{\beta_2, g,T}$
when $\beta_1\ge \beta_2$.
Note that for any bounded function $h$ on $[0, T]\times \R^d$, the parabolic equation
\begin{equation}\label{e6-3}
\begin{cases}
  \frac{\partial}{\partial t}\bar u(t,x)=\bar \LL \bar u(t,\cdot)(x)
  + h(t,x),\quad &(t,x)\in [0,T]\times \R^d,\\
  \bar u(0,x)=g(x)
\end{cases}
\end{equation}
has a unique solution given by
$$\bar u(t,x)= \bar P_tg(x) +\int_0^t \bar P_{t-s}h(s, x)\,ds,\quad (t,x)\in [0,T]\times \R^d,$$ where $\{\bar P_t\}_{t\ge0}$ is the semigroup associated with the L\'evy operator $\bar \LL$;
e.g. see \cite[Chapter 6, Sections 4 and 5]{Fri}.
It follows
from heat kernel estimates for the $\alpha$-stable-L\'evy process (i.e., for the semigroup $\{\bar P_t\}_{t\ge0}$)
 that
  $   C([0,T]; C_c^2(\R^d))\subset \mathscr{S}_{\alpha, g,T}$ for every
     $T>0$ and $g\in C_c^2(\R^d).   $
  On the other hand, it will be proved in Subsection \ref{section5.2} of the Appendix that
  $ \mathscr{S}_{\infty, g,T}$
  is dense in $L^2([0,T]\times \R^d; dt\times dx)$ for  every
  $T>0$ and $g\in C_c^2(\R^d).$ Therefore,
  $\mathscr{S}_{\beta, g,T}$ is dense in $L^2([0,T]\times \R^d; dt\times dx)$ for every $\beta\in [\alpha,\infty]$.

For any $k\ge1$, let $\mu^{ (k)}$ be the normalized counting measure
on $k^{-1}\Z^d$ so that $\mu^{ (k)} ([0, 1)^d)=1$.
Given $T>0$,  $g\in C_c^2 (\R^d)$,
$\beta\in(0,\infty]$ and $h\in \mathscr{S}_{\beta, g,T}
$, let
$u_k^\w\in C^1([0,T];L^2(k^{-1}\Z^d;\mu^{(k)}))$
be a solution to the following equation associated with
$\LL_t^{k,\w}$:
\begin{equation}\label{e6-4}
\begin{cases}
  \frac{\partial}{\partial t} u_k^\w(t,x)=\LL_t^{k,\w} u_k^\w(t,\cdot)(x)
  + h (t,x),\quad & (t,x)\in [0,T]\times k^{-1}\Z^d,\\
  u_k^\w(0,x)=g(x).
\end{cases}
\end{equation}
We
extend  the definition of $u^{\w}_k$ to $u^{\w}_k: [0,T]\times \R^d \to \R$ by setting $u^{\w}_k(t,x)=u^{\w}_k(t,z)$ if $x\in \prod_{1\le i\le d}(z_i,z_i+k^{-1}]$ for
the unique $z:=(z_1,z_2,\cdots, z_d)\in k^{-1}\Z^d$.

\medskip

The following is the main  result of this paper. It  gives the
quantitative homogenization
 result for
  the time-dependent operators $\{\LL_t^\w\}_{t\ge0}$.
Throughout the paper, for $a, b\in \R$, $a\vee b:= \max \{a, b\}$
and $a\wedge b =\min\{a, b\}$.

\begin{theorem}\label{T:1.1}
Assume that {\bf Assumption (H)} holds and that
$d>\alpha$.
 For  $T>0$,   $g\in C_c^2 (\R^d)$,  $\beta\in(0,\infty]$ and $h\in \mathscr{S}_{\beta, g,T}
 $, let
  $\bar u$ and $u^\w_k$ be the solutions to \eqref{e6-3} and \eqref{e6-4}, respectively.
Then, for
a.e. $\w\in \Omega$ and
any
$\gamma>0$, there
are
a constant $C_0>0$ $($which depends on
$T$,
$f$, $g$ and $\gamma)$ and a random variable
$k_0(\w)\ge1$ $($which depends on
$T$ and $\gamma$ but is independent of $f$ and $g)$
such that for all  $k\ge k_0(\w)$,
\begin{equation}\label{t6-1-1aa}
\begin{split}
& \sup_{t\in [0,T]}\|u^{\w}_k(t,\cdot)-\bar u(t,\cdot)\|_{L^2(\R^d;dx)}\\
 &\le C_0\sqrt{\pi\left(k^{\alpha}T\right)}+C_0
 \begin{cases}
\min\left\{k^{-\frac{\alpha(d+2\beta)}{2(2d+\alpha+2\beta)}}(\log^{\frac{\alpha(1+\gamma)}{(4(d-\alpha))\wedge 2d}}k),
\, k^{- (\frac{1-\alpha}{1+\gamma}\wedge \frac{d}{2(1+\gamma)})}\log^{1/2}k\right\},
&\quad \alpha\in (0,1),\\
k^{-\frac{d+2\beta}{2(2d+1+2\beta)}}(\log^{ 1+\frac{1+\gamma}{(4(d-1))\wedge 2d} } k)
,&\quad \alpha=1,\\
k^{-\frac{(2-\alpha)(d+2\beta)}{2(2d+\alpha+2\beta)}}(\log^{\frac{\alpha(1+\gamma)}{(4(d-\alpha))\wedge 2d}} k),
&\quad \alpha\in (1,2),
\end{cases}
\end{split}
\end{equation}
where $\displaystyle \pi(t):=\frac{1}{t}  \int_0^t \left|\sK (s)-\sK \right|^2\,ds$.
\end{theorem}

We point out that   Theorem \ref{T:1.1} above in particular is applicable  to the  time-independent conductance case, and gives the quantitative homogenization result for the semigroups. It complements
the quantitative homogenization results recently obtained in \cite{CCKW3} for the time-independent conductance model in terms of resolvents.

\begin{remark}\label{remark1.2} We  further
make
three
comments on  Theorem \ref{T:1.1}.
\begin{itemize}
\item [(1)] Different from the stochastic homogenization
for nearest neighbor random walks in time-dependent random environments \cite{ACJS,ACS,ABM},
we do not assume in this paper
 any ergodic condition on the random coefficients $\{w_{t,x,y}: t\in \R_+, (x,y)\in E\}$ in
  the time variable $t$.
In fact, according to {\bf Assumption (H)}(iv), for every fixed $(x,y)\in E$,
the collection of random variables $\{w_{t_i,x,y}\}_{i\ge 1}$ may not be independent with each other
for mutually distinct $\{t_i\}_{i\ge 1}\subset \R_+$. See Example \ref{ex1-3} below.

\item [(2)] It is seen from \eqref{t6-1-1aa} that the convergence rates of
the stochastic homogenization for the time-dependent operators $\{\LL_t^\w\}_{t\ge0}$ come from two terms.
The first one is $\pi\left(k^\alpha T\right)$, which
describes the convergence for the expectation of random coefficients $\sK (t)=\Ee\left[w_{t,x,y}\right]$
to $\sK$ as $t\to \infty$, and the
second  term comes from the averaging
 of space variables for random coefficients $\{w_{t,x,y}:t\in \R_+, (x,y)\in E\}$, which is
 of the polynomial decay.
The leading order of the second term is
$\frac{(\alpha\wedge(2-\alpha))(d+2\beta)}{2(2d+\alpha+2\beta)}$ for all $\alpha\in (0,2)$, which is increasing with respect to the parameter $\beta$. Therefore, the faster decay of the test function $h(t,x)$ given in \eqref{e6-3} (that is, the larger $\beta$ in \eqref{e6-3}), the faster rate
in the quantitative stochastic homogenization for time-dependent random
conductance models with stable-like jumps. In particular,
when $\beta=\infty$ and $\sK (t)\equiv \sK $ for all $t\in \R_+$ (which implies $\pi(t)\equiv 0$), the estimate \eqref{t6-1-1aa} is
 reduced to
\begin{equation}\label{e:1.11}\begin{split}
& \sup_{t\in [0,T]}\|u^{\w}_k(t,\cdot)-\bar u(t,\cdot)\|_{L^2(\R^d;dx)}
 \\
&\le C_0
\begin{cases}
\min\left\{k^{-{\alpha/2}}(\log^{\frac{\alpha(1+\gamma)}{(4(d-\alpha))\wedge 2d}}k),
\, k^{- (\frac{1-\alpha}{1+\gamma}\wedge \frac{d}{2(1+\gamma)})} \log^{1/2}k \right\},
&\quad \alpha\in (0,1),\\
k^{-1/2}(\log^{ 1+\frac{1+\gamma}{(4(d-1))\wedge 2d} } k),
&\quad \alpha=1,\\
k^{-(2-\alpha)/2}(\log^{\frac{\alpha(1+\gamma)}{(4(d-\alpha))\wedge 2d}} k),
&\quad \alpha\in (1,2).
 \end{cases}
\end{split}
\end{equation}
This assertion corresponds to the main result in \cite[Theorem 1]{CCKW3} for
time-independent
random conductance models with stable-like jumps.

\item [(3)] Note that in \cite{F,KPP,KPP1,PZ,PZ1} different
kinds
of ergodicity properties for the time variable are assumed, which in turn yield different limit equations. We believe that under
these conditions, the convergence speed would be different from that in Theorem \ref{T:1.1}.
\end{itemize}
\end{remark}

We give some examples that satisfy {\bf Assumption (H)}.

\begin{example}\label{ex1-3}
\begin{enumerate}
\item
Let $\{Z_{(k)}^{(i,(x,y))}: i=1,2, (x,y)\in E, k\in {\mathbb N}_+\}$ be
independent and uniformly bounded
non-negative
random variables with
$\Ee [ Z_{(k)}^{(i,(x,y))}]=1$ for all $i=1,2,$ $(x,y)\in E$ and $k\in {\mathbb N}_+$.
Fix  $(x,y)\in E$. For any $n\ge0$, on the time interval $[n, n+1]$, we define $t\mapsto w_{t,x,y}$
to be
 a linear function on the subintervals
$[n, n+ 1/2]$
and $[n+ 1/2, n+1]$ with  $w_{n,x,y} = Z_{(n+1)}^{(1,(x,y))}$, $w_{n+ 1/2, x,y} = Z_{(n+1)}^{(2,(x,y))}$ and
$w_{n+1,x,y}= Z_{(n+2)}^{(1,(x,y))}$; that is,
\begin{align*}
w_{t,x,y}=\begin{cases}
(2n+1-2t)Z_{(n+1)}^{(1,(x,y))} + (2t-2n)Z_{(n+1)}^{(2,(x,y))} ,\quad & n\le t\le n+1/2,\\
(2n+2-2t)Z_{(n+1)}^{(2,(x,y))}+(2t-(2n+1))Z_{(n+2)}^{(1,(x,y))},\quad & n+1/2\le t< n+1.
\end{cases}
\end{align*}
Then, {\bf Assumption (H)}(i), (ii) and (iv) clearly hold with $\sK(t)\equiv 1$,
 {\bf Assumption (H)}(iii) is because
$
\Ee\left[|w_{t,x,y}-w_{s,x,y}|^2\right]\le c_1|t-s|^2,
$
which can be obtained thanks to the
uniform boundedness
of $\{Z_{(k)}^{(i,(x,y))}: i=1,2, (x,y)\in E, k\in {\mathbb N}_+\}$. Hence, Theorem \ref{T:1.1} holds for the time-dependent process given by
the (time-dependent) random conductance $\{w_{t,x,y}|x-y|^{-d-\alpha}:t\in \R_+, (x,y)\in E\}$. This is a non-trivial degenerate time-dependent
example, when $\Pp (Z_{(k)}^{(i,(x,y))}=0)>0$.
Note that this example has in addition
a finite range dependence property in the sense that for any $t, s\in [0, \infty)$ with $|t-s|\geq 1$,
the sequences $\{w_{s,x,y}: (x, y)\in E\}$ and $\{w_{t,x,y}: (x, y)\in E\}$ are independent.

\medskip

\item In the setting of (1), let
\begin{align*}
w_{t,x,y}=\begin{cases}
\cos^2 ((t-n)\pi) Z_{(n+1)}^{(1,(x,y))} +\sin^2 ((t-n)\pi)  Z_{(n+1)}^{(2,(x,y))} ,\quad & n\le t\le n+1/2,\\
\sin^2 ( (n+1-t) \pi) Z_{(n+1)}^{(2,(x,y))}+\cos^2 ( (n+1-t) \pi)  Z_{(n+2)}^{(1,(x,y))},\quad & n+1/2\le t< n+1.
\end{cases}
\end{align*}
By the same reasoning,  $\{w_{t,x,y}: t\in \R_+, (x,y)\in E\}$  satisfies {\bf Assumption (H)}
 with $\sK(t)\equiv 1$. This example also has an additional finite range dependence property.

\medskip

\item    Let $\{Z^{(i,(x,y))}: i=1,2, (x,y)\in E\}$ be
 independent and uniformly bounded  non-negative
random variables with  $\Ee [ Z^{(i,(x,y))}]=1$ for all $i=1,2$ and $(x,y)\in E$.
Let $f_1$ and $f_2:[0,\infty) \to \R_+$ be global Lipschitz functions satisfying that there are positive constants $K_1$, $K_2$ and $\sK $ so that
$
K_1\le \min\{f_1(t),f_2(t)\}\le \max\{f_1(t),f_2(t)\}\le K_2$ for all $t\in \R_+$
and that
$\displaystyle \lim\limits_{t \to \infty}\frac{1}{t}  \int_0^t
\left|f_1(s)+f_2(s)-\sK \right|^2\,ds=0.
$
Define
$$
w_{t, x, y} = f_1(t) Z^{ (1,(x,y))} + f_2(t)   Z^{(2,(x,y))},\quad t\in \R_+,(x,y)\in E.
$$
Clearly, $\{w_{t,x,y}: t\in \R_+, (x,y)\in E\}$  satisfies {\bf Assumption (H)}
with $\sK(t)=f_1 (t)+f_2(t)$,
This example does not have  a finite range dependence property in general.
\end{enumerate}
\end{example}

There are quite some activities recently in the study of
stochastic processes in space-time random environments.
Andres, Chiarini, Deuschel and Slowik \cite{ACJS}, Andres, Chiarini and Slowik
\cite{ACS} established quenched invariance principle and quenched local limit theorem for (symmetric) nearest
neighbor random walks in space-time ergodic environments under moment conditions on the  time-dependent conductances respectively.
Later, Biskup and Rodriguez \cite{BR} proved quenched invariance principle for this model under
some improved
moment conditions.
Meanwhile,
Kleptsyna, Piatnitski and Popier \cite{KPP, KPP1} investigated the limit for higher order expansion of the solution
to second order non-autonomous parabolic differential equations with divergence form, where the space variable is periodic
and the time variable is ergodic. Piatnitski and Zhizhina \cite{PZ, PZ1} have studied qualitative homogenization
for a class of non-symmetric non-local operators endowed with $L^2$-integrable jumping kernel, periodic
space variables and ergodic time variables. Fehrman \cite{F} studied the stochastic homogenization with space-time ergodic divergence-free drift.
For these non-symmetric models, an extra term will appear in the limit, partly due to different ergodic properties of the time variable. Deuschel and Guo \cite{DG1},
Deuschel, Guo and  Ram\'irez \cite{DGR} have proved quenched invariance principle and quenched local limit theorem
for nearest neighbor random walk in time-dependent balanced random environments, respectively,
where the homogenized process is a Brownian motion.

The convergence rates of scaled equations
are the central issues in the theory of quantitative
homogenization, and have been studied extensively for second-order parabolic operators in divergence form,
with
rapidly oscillating, time-independent or time-dependent, periodic coefficients; see \cite{KLS, GS} and the references therein.
The study of stochastic homogenization
advances significantly,
thanks to
the environment as seen from the particle technique
introduced by
Kipnis and Varadhan \cite{KV}, Kozlov \cite{Ko} and Papanicolaou and Varadhan \cite{PV}
in the static random environment setting.
Yurinski \cite{Yu} firstly established a result
on the convergence rate for symmetric
diffusions in random environments.
In recent years, there have been a great amount of
important progress
in the quantitative stochastic homogenization for linear elliptic or parabolic PDEs in divergence form with random coefficients or discrete elliptic equations associated with nearest-neighbour random walks in random media.
We refer the reader to Armstrong, Kussi and Mourrat \cite{AKM,AKM1,AKM2}, Armstrong and Smart \cite{AS}, Gloria and Otto \cite{Go1,Go2,Go3}
for systematic
account of
 recent development on the theory of quantitative stochastic homogenization (in static environments).
In particular, Armstrong, Bordas and Mourrat \cite{ABM} studied the convergence rate for symmetric parabolic differential equations in
space-time environments satisfying the finite range dependent property.
See \cite{AK,AM,AW,CMOW,DG,FGW,GGM,GM,MOW} and reference therein for various applications of the theory of quantitative stochastic homogenization.

The fractional Laplacian operator or the $\alpha$-stable random walk, as a typical non-local operator or a random walk with long range jumps,
 has attracted a lot of interests.
The corresponding qualitative homogenization problem has been
studied
extensively;
see \cite{CCKW2, CCKW3, CKW20,  CKK, CKW21, FH}. Among them,
the result on quantitative stochastic homogenization was relatively limited.
The convergence rate for scaled elliptic equations associated with $\alpha$-stable-like random conductance models with (time-independent) coefficients was obtained in \cite{CCKW3}; see Remark \ref{remark1.2}(1) above.  However, by our knowledge, no quantitative result was known for parabolic equations with (time-dependent) random
coefficients. The contribution of our main Theorem \ref{T:1.1} is to address this problem.

\begin{remark}
We would like to outline the new ideas and the novelties of our proof for
 Theorem \ref{T:1.1}, though it is partly motivated by that of \cite[Theorem 1]{CCKW3}. The main difference and difficulty come from the following
three facts.
\begin{itemize}
\item [(1)]
As explained in \cite{CCKW3}, due to the heavy tail
phenomenon
of the conductances, i.e., for $t\in \R_+$,
$$ \Ee\left[\sum_{y\in \Z^d}|x-y|^2\frac{w_{t,x,y}}{|x-y|^{d+\alpha}}\right]=\infty,$$
it seems that there
is no global corrector whose derivatives have good tail  estimates at infinity. So we
construct a regional corrector instead.
The equation for regional corrector in the  time-dependent ergodic environment is a parabolic equation,
rather than an elliptic equation in the static random
environment case, so the arguments in \cite{CCKW3} for the time-independent framework are not directly applicable.
Compared with \cite{CCKW3}, some extra efforts
are required to establish the space-time large scale estimate \eqref{p6-1-1} for the regional corrector,
where
the $L^2$-norm in space variable
and
the uniform bound on the time variable
are chosen, and for
the energy term we take an integral with respect to the time variable.
In particular, due to the time-dependence of random
coefficients  $\{w_{t,x,y}: t\in \R_+,(x,y)\in E\}$,
we develop some new method to analyze
the large scale estimates  for the regional corrector
 in both space
and time variables
(see Section \ref{section2-}), as well as to estimate the average of a class of  scaled time-dependent operators
(see Remark \ref{r3-4} below).  These estimates are crucial
to the proof of Theorem \ref{T:1.1}, see Section \ref{section4} for details.

\item [(2)] Uniformly elliptic (non-degenerate) conditions on the random conductances were adopted in \cite{ABM} to study
the convergence rates for symmetric parabolic differential equations in space-time environments, while in this paper we
can still obtain the result on quantitative stochastic homogenization for stable-like random walks in
space-time environments that are endowed with degenerate conductances. The main tools are the large scale
uniform (with respect to the time variable) Poincar\'e inequality \eqref{l2-1-1} and the multi-scale Poincar\'e inequality \eqref{l2-2-1}.
The key ingredients to establish these two inequalities are the uniform time-dependent non-degenerate estimates for the volume
of vertices whose associated conductances have satisfactory lower bounds; see Lemma \ref{l2-0} below. Unlike
the case of static environments in \cite{CCKW3}, we use the Kolmogorov continuous condition (i.e., {\bf Assumption (H)}(iii)) to verify that
in the quenched sense, the set of vertices associated with non-degenerate conductances will not change very fast when the difference of time variables
is small.

\item [(3)]
Besides the time-dependent framework, the
quantitative result
in Theorem \ref{T:1.1} is more general than that in  \cite{CCKW3}, as
here we do not require the solution to the equation \eqref{e6-3}
has
compact support.
(That is, only the case that $\beta=\infty$ is treated in \cite{CCKW3}.)
Due to the property of regional corrector, here we adopt
cut-off arguments and some new ideas are
introduced
 in the proof of Theorem \ref{T:1.1}. In particular, Theorem \ref{T:1.1} reveals that
 the decay behavior of solution
 $\bar u(t,x)$ for the limit equation \eqref{e6-3} may have an essential impact on the convergence rate for the associated scaled equations.
\end{itemize}
\end{remark}

\ \

The rest of this paper is organized as follows. In Section \ref{section2-}, we
derive
 $L^2$-estimates and energy estimates for the solutions to regional parabolic equations associated with the operator $\LL_t^{k,\w}$. In order to obtain them we establish the Poincar\'e inequality and multi-scale Poincar\'e inequality for the time-dependent operator $\LL_t^{k,\w}$ in here.
In Section \ref{section3} we present uniform $L^2$-estimates for the differences among scaled operators $\sL^{k,\w}_t$, as well as their variants, and the limit operator $\bar \sL$. The last section is devoted to the proof of Theorem \ref{T:1.1}.

\section{Parabolic equations associated with the operator $\LL_t^{k,\w}$} \label{section2-}
Throughout this paper, for each $x\in \R^d$ and $R>0$, set $B_R(x):=x+(-R,R]^d$, and
with a little abuse of notation but it should be clear from the context,
we also use $B_R(x)$ to denote
$B_R(x) \cap \Z^d$ or $B_R(x) \cap  (k^{-1}\Z^d)$
for $k\ge 1$. For simplicity, we write $B_R=B_R(0)$.
Recall that $\mu$ denotes the counting measure on $\Z^d$.
For any $k\in \N$, let $\mu^{ (k)}$ be the normalized counting measure on
$k^{-1}\Z^d$ so that $\mu^{ (k)}((0,1]^d)=1$.
For $f: k^{-1}\Z^d \to \R$, set \begin{equation}\label{e:cou-k}\mu^{(k)}(f):=\int_{k^{-1}\Z^d} f(x)\,\mu^{(k)}(dx)=k^{-d}\sum_{x\in k^{-1}\Z^d} f(x).\end{equation}
For any subset $U\subset \Z^d$ and any
function $f: U \to \R^n$ with $n\ge 1$, define
 the normalized integral of $f$ on $U$ as
 \begin{equation}\label{e1-1a}
 \oint_{U}f\,d\mu=\oint f(t,x)\mu(dx)
:=\frac{1}{\mu(U)}\sum_{x\in U} f(x).
\end{equation}
For any $t\in \R_+$, $U\subset \Z^d$ and $f:\R_+\times U\to \R^n$ with  $n\ge 1$, define
$$
\LL^\w_{t,U}f(t,x)
:=\sum_{y\in U:y\neq x}\left(f(t,y)-f(t,x)\right)\frac{w_{t,x,y}(\w)}{|x-y|^{d+\alpha}} ,
\quad {  x\in U },
$$
and
 \begin{align*}
 \E_{t,U}^\w(f(t,\cdot),f(t,\cdot)) &:=\sum_{i=1}^n
\E^\w_{t,U}(f^{(i)}(t,\cdot),f^{(i)}(t,\cdot)):=\frac{1}{2}\sum_{i=1}^n\sum_{x,y\in U: x\neq y}(f^{(i)}(t,x)-f^{(i)}(t,y))^2\frac{w_{t,x,y}(\w)}{|x-y|^{d+\alpha}}.
\end{align*}
It is not difficult to verify that for any $t\in \R_+$, every finite subset $U\subset \Z^d$ and $f:\R_+\times U \to \R^n$,
\begin{equation}\label{e1-3a}
-\int_{U}\left\langle \LL_{t,U}^\w f(t,x), f(t,x)\right\rangle\,\mu(dx)=\mathscr{E}_{t,U}^{\w}(f(t,\cdot),f(t,\cdot)).
\end{equation}
In the rest of this paper, for simplicity of notations sometime we will omit the parameter
$\w$  when there is no danger of confusion.

\medskip

In this section, we consider the time-dependent regional corrector, which is a
solution to a parabolic equation associated with the regional
version $\LL_{t,U}^{k,\w}$ of the operator $\LL_t^{k,\w}$, and  establish
its $H^{\alpha/2}$-bounds by applying Poincar\'e-type inequalities
that are studied in the following subsection.

\subsection{Poincar\'e inequality and multi-scale Poincar\'e inequality}

In this subsection, we consider {\bf Assumption (H')} that is slightly weaker than {\bf Assumption (H)}.
We say  {\bf Assumption (H')}  holds if  the
time-dependent non-negative random variables $\{w_{t,x,y}: t\in \R_+, (x,y)\in E\}$ satisfy
(i), (iii) and (iv) of {\bf Assumption (H)} and have {\bf (H)}(ii) being replaced by the following weaker moment condition

\medskip

 \noindent {\bf (H')}(ii): {\it
there exist constants  $p>1$ and $\tilde  C_1\geq 1 $ so that
\begin{equation}\label{e:2.4}
\Ee  [ w_{t, x, y}^p] \leq \tilde C_1 \quad \hbox{for all } t\in \R_+ \hbox{ and } x\not=y \in \Z^d .
\end{equation}}

As in \cite[Section 2]{CCKW3}, for every $x_1,x_2,y\in \Z^d$, $t\ge0$, $r\ge 1$ and $\gamma>0$, we define
$$B_{r,\gamma}^{t, x_1,x_2}(y)=\left\{z\in B_r(y)\setminus\{x_1,x_2\}: w_{t,x_1,z}\ge \gamma, w_{t,x_2,z}\ge \gamma\right\},$$
which reflects good vertices whose associated conductance have a uniform  lower bound in the ball $B_r(y)$. We need
the following
result on
large scale lower bounds for the volume of $B_{r,\gamma}^{t, x_1,x_2}(y)$
under {\bf Assumption (H')}.

\begin{lemma}\label{l2-0}
Suppose that
{\bf Assumption (H')} holds.
 Then there
exist constants $\lambda_0,\delta_0\in (0,1)$  such that, for a.e.\ $\w\in \Omega$ and
every $\theta>1/d$, there is a random variable $m_0(\w)\ge1$ so that for all $m>m_0(\w)$, $x_1,x_2,y\in B_{2^m}$ with $x_1\neq x_2$, $0\le t \le 2^{m\alpha}T$ and $[m^\theta]\le r \le 2^m$,
\begin{equation}\label{l2-0-1}
\mu(B_{r,\delta_0}^{t,x_1,x_2}(y))\ge \lambda_0\mu(B_r(y)).
\end{equation}
\end{lemma}

\begin{proof}
 Take
$\delta_0= 1/(4K_1)$,
where $K_1$ is the positive constant given in \eqref{a1-1-0}.
 By {\bf Assumption (H)}(i),
 {\bf (H')}(ii) and the H\"older's inequality,
   it holds that for all $x\not= y\in \Z^d$ and $t\in \R_+$,
\begin{align*}
K_1 &\le \Ee w_{t,x,y}=  \Ee [w_{t,x,y}; w_{t, x, y} \leq
2
\delta_0] +  \Ee [w_{t,x,y}; w_{t, x, y} >
2\delta_0] \\
&\leq
2\delta_0
+ (  \Ee [w_{t,x,y}^p])^{1/p}  \, \Pp( w_{t, x, y} >
( 2\delta_0)^{1/q})
 \leq 2\delta_0 + \tilde C_1^{1/p}  \, \Pp( w_{t, x, y} > 2\delta_0)^{1/q},
 \end{align*}
 where $q= p/(p-1)$.
 This implies that for all $x\not= y\in \Z^d$ and $t\in \R_+$,
$$
\Pp(w_{t,x,y}>2\delta_0)\ge
(K_1/2)^q
\tilde C_1^{-1/(p-1)}  =:p_0.
  $$
Fix $x_1\neq x_2\in \Z^d$ and $t\in \R_+$, and set $\eta_z^{t,x_1,x_2}(\w):=
 \I_{\{w_{t,x_1,z}>2\delta_0,w_{t,x_2,z}>2\delta_0\}}(\w)$
for any $z\in \Z^d \setminus\{x_1, x_2\}$.
By {\bf Assumption (H)}(iv), $\{\eta_z^{t,x_1,x_2}:z\in \Z^d\setminus\{ x_1,  x_2\}\}$ is a sequence of
independent
$\{0,1\}$-valued random variables with $\Pp(\eta_z^{t,x_1,x_2}=1)
\ge p_0^2$. Hence, taking $\lambda_0\in (0,p_0^2/2)$ and using Hoeffding's inequality (e.g.\ see \cite[Theorem 2.16, p.\ 21]{BDR}),
we obtain that for all $x_1,x_2,y\in \Z^d$ with $x_1\neq x_2$, $t\in \R_+$ and $r\ge5$,
\begin{equation}\label{l2-0-2}
\begin{split}
&\Pp\left(\mu\left(B_{r,2\delta_0}^{t, x_1,x_2}(y)\right)\le \lambda_0 \mu(B_r(y))\right)=\Pp\bigg(\sum_{z\in  B_r(y)
\setminus\{ x_1,  x_2\} } \I_{\{w_{t,x_1,z}>2\delta_0,w_{t,x_2,z}>2\delta_0\}}
\le   \lambda_0  \mu(B_r(y))\bigg)\\
&\le
\Pp\bigg(\frac{\sum_{z\in  B_r(y)
\setminus\{ x_1,  x_2\} }
\left(\eta_z^{t,x_1,x_2}-\Ee[\eta_z^{t,x_1,x_2}]\right)}{\mu(B_r(y))}
\le 2 \lambda_0 -p_0^2\bigg)
\le c_1e^{-c_2r^{d}}.
\end{split}
\end{equation} By choosing $c_1$ large enough, we can see that the estimate \eqref{l2-0-2} holds for all $1\le r\le 4$.

Let $k\ge1$ be an integer. For every $m\ge 1$, let
\begin{align*}
t_i^{m}:=i2^{-mk} \quad \hbox{for } 0\le i \le [2^{m(k+\alpha)}T]+1.
\end{align*}
By \eqref{l2-0-2}, we obtain that for every $\theta>1/d$ and $  \lambda_0 \in (0,p_0^2/2)$,
\begin{equation}\label{l2-0-3}
\begin{split}
\sum_{m=1}^\infty\sum_{x_1,x_2,y\in B_{2^m};x_1\neq x_2}\sum_{i=1}^{[2^{m(k+\alpha)}T]+1}\sum_{r=[m^\theta]}^{2^m}
\Pp\left(\mu\left(B_{r,
2\delta_0
}^{t_i^{m}, x_1,x_2}(y)\right)\le \lambda_0 \mu(B_r(y))\right)<\infty.
\end{split}
\end{equation}
On the other hand, choosing $k$ large enough so that $k>\frac{2d+\alpha+p}{\eta}$ (here $\eta>0$ is the positive constant given in \eqref{a1-1-1}),
and using \eqref{a1-1-1} and the Markov inequality,  we arrive at
\begin{align*}
&\sum_{m=1}^\infty\Pp\left(\bigcup_{x,y\in B_{2^{m+1}}}\bigcup_{i=0}^{[2^{m(k+\alpha)}T]+1}\left\{\w\in \Omega: \left|w_{t_i^{m},x,y}(\w)-w_{t_{i+1}^{m},x,y}(\w)\right|>2^{- m}\right\}\right)\\
&\le \sum_{m=1}^\infty\sum_{x,y\in B_{2^{m+1}}}\sum_{i=0}^{[2^{m(k+\alpha)}T]+1}2^{m p}\Ee\left[\left|w_{t_i^{m},x,y}-w_{t_{i+1}^{m},x,y}\right|^p\right]
\le c_3\sum_{m=1}^\infty 2^{-m\left(k(1+\eta)-(k+\alpha+p+2d)\right)}
<\infty.
\end{align*}
This, along with \eqref{l2-0-3} and the Borel-Cantelli lemma, yields that there are an integer $m_1(\w)\ge0$ and a $\Pp$-null set $\Lambda_1$ so that
for every  $\w\notin \Lambda_1$, $m\ge m_1(\w)$, $x_1,x_2,y_1\in B_{2^{m}}$ with $x_1\neq x_2$, $x,y\in B_{2^{m+1}}$, $[m^\theta]\le r \le 2^m$ and $0\le i \le [2^{m(1+\alpha)}T]$,
\begin{equation}\label{l2-0-5}
\mu(B_{r, 2\delta_0 }^{t_i^{m}, x_1,x_2}(y_1))> \lambda_0  \mu(B_r(y_1)),\quad  |w_{t_i^{m},x,y}(\w)-w_{t_{i+1}^{m},x,y}(\w) |\le 2^{-m}.
\end{equation}

Let $\mathscr{C}_m:=\{t_i^m\}_{i=0}^{[2^{m(k+\alpha)}T]+1}$ for all $m\ge1$. Then, for every $m\ge1$ and
$0\le t\le 2^{m\alpha}T$, there exists  a non-random sequence $\{t_n\}_{n\ge0}$ such that the following properties hold:
\begin{itemize}
\item [(i)] $t_0,t_1\in \mathscr{C}_m$ satisfy that $t_1\le t \le t_0$ and $t_0-t_1=2^{-mk}$;
\item [(ii)] $t_n\in \mathscr{C}_{m+n-1}$ and $t_n\le t \le t_n+2^{-(m+n-1)k}$ for every $n\ge 2$;
\item [(iii)] $t_n,t_{n-1}\in \mathscr{C}_{m+n-1}$ (since $\mathscr{C}_m \subset \mathscr{C}_{m+1}$ for all $m\ge1)$, and $t_n-t_{n-1}\le 2^{-(m+n-2)k}$ for every $n\ge 2$.
\end{itemize}
According to all the properties above, for every $m\ge1$, $0\le t\le 2^{m\alpha}T$ and $n\ge 2$, we can find $s_0<s_1\cdots<s_N$ so that
$N\le 2^k$, $s_0=t_{n-1}$, $s_N=t_n$, $s_i\in \mathscr{C}_{m+n-1}$ and $s_{i+1}-s_i=2^{-(m+n-1)k}$ for each $0\le i \le N-1$.
Therefore, for every $\w\notin \Lambda_1$, $m\ge m_1(\w)$ and $x,y\in B_{2^{m+1}}$,
\begin{align}\label{l2-0-6}
 \left|w_{t_n,x,y}(\w)-w_{t_{n-1},x,y}(\w)\right|&\le \sum_{i=0}^{N-1}
 \left|w_{s_{i+1},x,y}(\w)-w_{s_{i},x,y}(\w)\right|\le N\cdot 2^{-(m+n-1)}\le 2^{-(m+n-1-k)},
\end{align}
where the second inequality follows from \eqref{l2-0-5}.

Furthermore, according to \eqref{a1-1-1} and the Kolmogorov continuity theorem (e.g. see \cite[Theorem 39.3]{Bau}), for every $x,y\in \Z^d$, there exists a $\Pp$-null set
$\Lambda_{x,y}\subset \Omega$ such that
\begin{align*}
s\mapsto w_{s,x,y}(\w)\ \text{is\ continuous},\quad \w\notin \Lambda_{x,y}.
\end{align*}
Let $\Lambda_2:=\bigcup_{x,y\in \Z^d}\Lambda_{x,y}$, which is a $\Pp$-null set and satisfies that
\begin{equation}\label{l2-0-4}
s\mapsto w_{s,x,y}(\w)\ \text{is\ continuous},\quad x,y\in \Z^d,\ \w\notin \Lambda_2.
\end{equation}
Therefore we deduce by  \eqref{l2-0-4} and \eqref{l2-0-6} that for every $\w\notin \Lambda:=\Lambda_1 \bigcup \Lambda_2$, $m\ge m_1(\w)$, $0\le t \le 2^{m\alpha}T$ and
$ x,y\in B_{2^{m+1}}$,
\begin{align*}
\left|w_{t,x,y}(\w)-w_{t_0,x,y}(\w)\right|&=\lim_{n \to \infty}\left|w_{t_n,x,y}(\w)-w_{t_0,x,y}(\w)\right|\\
&\le \sum_{n=1}^\infty\left|w_{t_n,x,y}(\w)-w_{t_{n-1},x,y}(\w)\right|\\
&\le 2^{-m}+\sum_{n=2}^\infty 2^k \cdot 2^{-(m+n-1)}\le c_4 2^k 2^{-m}.
\end{align*} This together with \eqref{l2-0-5} gives us that for every $0\le t \le 2^{m\alpha}T$, $x_1,x_2,y\in B_{2^m}$ with $x_1\neq x_2$, $z\in B_{r, 2\delta_0 }^{t_0, x_1,x_2}(y)$ and $[m^\theta]\le r \le 2^m$,
\begin{align*}
w_{t,x_i,z}(\w)\ge w_{t_0,x_i,z}(\w)-c_42^k2^{-m}\ge
2\delta_0 -c_42^k2^{-m},\quad i=1,2.
\end{align*}
Hence, for every $\w\notin \Lambda$, there is $m_2(\w)\ge1$ so that
for every $m\ge m_2(\w)$ and $0\le t \le 2^{m\alpha}T$, there exists a non-random
 $t_0\in \mathscr{C}_m$
such that
\begin{align*}
w_{t,x_i,z}(\w)\ge
\delta_0,
\ |t-t_0|\le 2^{-km},\ \,x_1,x_2,y\in B_{2^m}\ \text{with}\ x_1\neq x_2, \,z\in B_{r,
2\delta_0
}^{t_0, x_1,x_2}(y),\ [m^\theta]\le r \le 2^m,\ i=1,2,
\end{align*}
which in particular means that $B_{r,
2\delta_0
}^{t_0, x_1,x_2}(y)\subset B_{r,
\delta_0
}^{t,x_1,x_2}(y)$.

According to this and \eqref{l2-0-5} again, we know that for every $\w\notin \Lambda$, there exists $m_2(\w)\ge1$ so that for
for every $m\ge m_2(\w)$,
\begin{align*}
\mu(B_{r,
\delta_0
}^{t, x_1,x_2}(y))>  \lambda_0  \mu(B_r(y)),\quad 0\le t\le 2^{m\alpha}T,\ x_1,x_2,y\in B_{2^m}\ \text{with}\ x_1\neq x_2,\ [m^\theta]\le r \le 2^m.
\end{align*}
The proof is complete.
\end{proof}

With the aid of \eqref{l2-1-1}, we can follow the arguments in \cite[Propositions 2.2 and 2.3]{CCKW3} to prove
the following time-dependent Poincar\'e inequality and multi-scale Poincar\'e inequality whose associated
conductances may be degenerate. For
reader's convenience, we will give the details of the  proofs in the Appendix of this paper.

\begin{proposition}\label{l2-1}{\bf(Poincar\'e inequality)} Under
{\bf Assumption (H')},
 there exists a constant $C_1>0$ such that for
 a.e.\ $\omega\in \Omega$ and each $\theta>1/d$, there is $m_0(\w)\ge1$ so that
 for all $m\ge m_0(\w)$, $0\le t\le 2^{m\alpha}T$, $y \in \Z^d$, $[m^\theta]\le r\le 2^m$ and $f:\Z^d \to \R$,
\begin{equation}\label{l2-1-1}
\oint_{B_r(y)}f^2(x)\,\mu(dx)-\bigg(\oint_{B_r(y)}f(x)\,\mu(dx)\bigg)^2\le C_1r^{\alpha-d}\mathscr{E}_{t, B_r(y)}^{\w}(f,f).
\end{equation}
\end{proposition}

When $n<m$, define
\begin{equation}\label{defzmn}
\Z_{m,n}^d:=\left\{z=(z_1,\cdots,z_d)\in B_{2^m}: z_i=k_i2^n\hbox{ for some odd}\ k_i\in
\Z
\hbox{ for all } 1\le i \le d\right\};
\end{equation}
when $m=n$, define $\Z_{m,m}^d=\{0\}$.

\begin{proposition}\label{l2-2}{\bf(Multi-scale Poincar\'e inequality)}
Under
{\bf Assumption (H')},
there exists a constant $C_2>0$ such that for
 a.e.  $\omega\in \Omega$ and each $\theta>1/d$, there is $m_0(\w)\ge1$ so that
 for all $m\ge m_0(\w)$, $0\le t\le 2^{m\alpha}T$, $y \in \Z^d$, $n\in \N \cap [\theta \log_2 (m \log 2),  m]$ and $f,g:\Z^d \to \R$,
\begin{equation}\label{l2-2-1}
\begin{split}
\sum_{x\in B_{2^m}}f(x)\bigg(g(x)- \oint_{B_{2^m}}g\,d\mu
\bigg)\le&
\sum_{z\in \Z_{m,n}^d}\sum_{x\in B_{2^n}(z)}f(x)\bigg(g(x)- \oint_{B_{2^n}(z)}g\,d\mu
\bigg)\\
&+C_2\mathscr{E}_{t,B_{2^m}}^\w(g,g)^{1/2}\sum_{k=n}^{m-1} 2^{k(d+\alpha)/2}\bigg(\sum_{y\in \Z_{m,k}^d}\bigg( \oint_{B_{2^k}(y)}f\,d\mu
\bigg)^2\bigg)^{1/2}.
\end{split}
\end{equation}
\end{proposition}

As mentioned before, under {\bf Assumption (H)} the time-dependent random coefficients $\{w_{t,x,y}: t\in \R_+,(x,y)\in E\}$ may be degenerate.
Hence, similar to \cite[Proposition 2.3]{CCKW3},  the multi-scale Poincar\'e inequality \eqref{l2-2-1}
holds for the time-dependent Dirichlet form $(\E^{\w}_t, \F^{\w})$ with possibly degenerate random conductances.

\begin{remark} \rm
\begin{itemize}
\item[{\rm(i)}]
Throughout the paper, the Kolmogorov continuous condition {\bf Assumption (H)}(iii) is only used in the proof of Lemma \ref{l2-0},
which allows us to pass the estimate \eqref{l2-0-2} from each fixed $t>0$ to all $t$ in a bounded interval.
Lemma \ref{l2-0} itself
   is only used to establish Propositions \ref{l2-1} and \ref{l2-2}.
  If the associated time-dependent random coefficients $\{w_{t,x,y}: t\in \R_+,(x,y)\in E\}$
are bounded from below by a positive constant, then it is relatively easy to get a uniform-in-time Poinacr\'e inequality \eqref{l2-1-1} and so the multi-scale Poincar\'e inequality \eqref{l2-2-1}. In this case,  {\bf Assumption (H)}(iii) is not needed.

\item[{\rm(ii)}]
The results of this subsection hold under {\bf Assumption (H)} but with its
{\bf (H)}(ii)
being replaced by {\bf (H)}(ii') of \eqref{e:2.4}.
In the sequel, we will need {\bf Assumption (H)}(ii).
See Remark \ref{r2-5} for comments how {\bf Assumption (H)}(ii) is used
in the results in the next subsection.
 \end{itemize}
\end{remark}

\subsection{Large scale estimates for correctors}

By applying the multi-scale Poincar\'e inequality \eqref{l2-2-1}, we can establish large scale estimates for correctors.
We first consider the case that
$\alpha\in (1,2)$.
For any $m\ge1$, consider the following system of parabolic equations on $\R_+\times B_{2^m}$:
\begin{equation}\label{e6-5}
\begin{cases}
  \frac{\partial}{\partial t}\phi_m(t,x)=\LL^\w_{t,B_{2^m}} \phi_m(t,\cdot)(x)+V(t,x)-\displaystyle \oint_{B_{2^m}} V(t,\cdot)\,d\mu
,&\quad (t,x)\in \R_+\times B_{2^m},\\
 \phi_m(0,x)=0,&\quad
\end{cases}
\end{equation}
where
\begin{equation}\label{e6-5a}
V(t,x):=\sum_{z\in \Z^d}\frac{z}{|z|^{d+\alpha}}w_{t,x,x+z}
\quad  \hbox{and} \quad
\oint_{B_{2^m}}V(t,\cdot)\,d\mu
:=\frac{1}{\mu(B_{2^m})}\sum_{z\in B_{2^m}}
V(t,z).
\end{equation}
Here and in what follow of this section,
$B_{2^m}$ is understood to be $B_{2^m} \cap \Z^d$.
Then, one can see
that the solution to \eqref{e6-5} is given by
$$
\phi_m(t,x)=\int_0^tP_{s,t,B_{2^m}}\left(V(s,x)-\displaystyle \oint_{B_{2^m}} V(s,\cdot)\,d\mu
\right)\,ds,
$$
where $\{P_{s,t,B_{2^m}}\}_{s\le t}$ is the Markov transition semigroup associated with the  time-dependent generator $\{\LL^\w_{t,B_{2^m}}\}_{t>0}$
on $B_{2^m}\cap \Z^d$;
e.g. see  \cite[Chapter 6, Sections 4 and 5]{Fri}.

\begin{proposition}\label{p6-1} Suppose that $\alpha\in(1,2)$ and $d>\alpha$.
Assume that {\bf Assumption (H)} holds. Let $\phi_m:\R_+\times \Z^d \to \R^d$ be the solution to \eqref{e6-5}.
Then, for
a.e. $\w\in \Omega$,
every $\gamma>0$ and $T>0$, there exist a constant $C_3>0$ $($which is independent of $T)$ and a random variable $m_1(\w)\ge1$ $($which depends on $\gamma$ and $T)$ such that for every $m\ge m_1(\w)$,
\begin{equation}\label{p6-1-1}
\sup_{t\in [0,2^{m\alpha}T]}\int_{B_{2^m}}|\phi_m(t,x)|^2\mu(dx)+
\int_0^{2^{m\alpha}T} \E^\w_{t,B_{2^m}}\left(\phi_m(t,\cdot),\phi_m(t,\cdot)\right) dt\le
C_3m^{\frac{(1+\gamma)\alpha}{(2(d-\alpha))\wedge d}}2^{m(\alpha+d)}T.
\end{equation}
\end{proposition}
\begin{proof} Write $\phi_m(t,x)= (\phi_m^{(1)}(t,x),\cdots,\phi_m^{(d)}(t,x) )$ and $V(t,x)= (V^{(1)}(t,x),\cdots,V^{(d)}(t,x) )$.
By \eqref{e6-5},
it holds that
$$
\begin{cases}
 \displaystyle\frac{d}{dt}\int_{B_{2^m}}\phi_m(t,x)\,\mu(dx)=
\int_{B_{2^m}}\LL^\w_{t,B_{2^m}}\phi_m(t,\cdot)(x)\,\mu(dx)+\displaystyle\int_{B_{2^m}} \left(V(t,x)-
 \oint_{B_{2^m}}V(t,\cdot)\,d\mu
\right )\,\mu(dx)=0,\quad t\ge0,\\
 \displaystyle\int_{B_{2^m}}\phi_m(0,x)\,\mu(dx)=0.
\end{cases}
$$
 So
 $\displaystyle\int_{B_{2^m}}\phi_m(t,x)\,\mu(dx)=0$ for all
$t\geq 0$. By \eqref{e1-3a},
we have
for all $t\ge0$,
\begin{equation}\label{p6-1-3}
\begin{split}
   \frac{d}{d t}\left(\int_{B_{2^m}}|\phi_m(t,x)|^2\mu(dx)\right)
 =& \, 2\sum_{i=1}^d\int_{B_{2^m}}\langle \LL^\w_{t,B_{2^m}} \phi_m^{(i)}(t,\cdot)(x),\phi_m^{(i)}(t,x)\rangle\mu(dx)\\
 &+
2\int_{B_{2^m}}\left\langle  V(t,x)-
\oint_{B_{2^m}}V(t,\cdot)\,d\mu,
\phi_m(t,x)\right\rangle\mu(dx)\\
 =&\,   -2 \E^\w_{t,B_{2^m}}\left(\phi_m(t,\cdot),\phi_m(t,\cdot)\right) \\
 &+
2\int_{B_{2^m}}\left\langle V(t,x)- \oint_{B_{2^m}}V(t,\cdot)\,d\mu,
\phi_m(t,x)\right\rangle\mu(dx).
\end{split}
\end{equation}

According to \eqref{defzmn}, $B_{2^m}= \cup_{z\in \Z_{m,n}^d}
B_{2^n}(z)$
 and  $| \Z_{m,n}^d|=2^{d(m-n)}$.
By applying  the multi-scale Poincar\'e inequality \eqref{l2-2-1}
with
$f=V^{(i)}(t,\cdot)-\displaystyle
\oint_{B_{2^m}}V^{(i)}(t,\cdot)\,d\mu
$ and $g=\phi_m^{(i)}(t,\cdot)$, and using the fact that $\displaystyle\int_{B_{2^m}}\phi_{m}^{(i)}(t,x)\mu(dx)=0$
for all $t\geq 0$ and   $1\le i \le d$, we have
 for a.e. $\w\in \Omega$, every
$m\ge m_0(\w)$, $0\le t\le 2^{m\alpha}T$ and $\left[\frac{(1+\gamma)\log_2 m}{d}\right]<n \le m$,
\begin{align*}
& \int_{B_{2^m}}\left\langle V(t,x)- \oint_{B_{2^m}}V^{(i)}(t,\cdot)\,d\mu,
\phi_m(t,x)\right\rangle\mu(dx)\\
&\le
\sum_{z\in \Z_{m,n}^d}\int_{B_{2^n}(z)}
\left\langle  V(t,x)-
\oint_{B_{2^m}}V^{(i)}(t,\cdot)\,d\mu,
\phi_m(t,x)- \oint_{B_{2^n}(z)}\phi_m(t,\cdot)\,d\mu
\right\rangle\,\mu(dx)\\
&\quad+c_0 \sum_{i=1}^d\mathscr{E}^\w_{t,B_{2^m}}(\phi^{(i)}_m(t,\cdot),\phi^{(i)}_m(t,\cdot)) ^{1/2}\sum_{k=n}^{m-1} 2^{k(d+\alpha)/2}
\left(\sum_{y\in \Z_{m,k}^d}\left| \oint_{B_{2^k}(y)}V^{(i)}(t,\cdot)\,d\mu -\oint_{B_{2^m}}V^{(i)}(t,\cdot)\,d\mu
\right|^2\right)^{1/2}\\
&\le
\sum_{z\in \Z_{m,n}^d}\int_{B_{2^n}(z)}
\left\langle  V(t,x)- \oint_{B_{2^m}}V^{(i)}(t,\cdot)\,d\mu,
\phi_m(t,x)-\oint_{B_{2^n}(z)}\phi_m(t,\cdot)\,d\mu \right\rangle\,\mu(dx)\\
&\quad+c_1 \mathscr{E}^\w_{t,B_{2^m}}(\phi_m(t,\cdot),\phi_m(t,\cdot)) ^{1/2}\sum_{k=n}^{m-1} 2^{k(d+\alpha)/2}
\left(\sum_{y\in \Z_{m,k}^d}\left|\oint_{B_{2^k}(y)}V(t,\cdot)\,d\mu
-\oint_{B_{2^m}}V(t,\cdot)\,d\mu\right|^2\right)^{1/2}\\
&=  :I_{m,n,1}(t)+I_{m,n,2}(t).
\end{align*}

Let $C_1>0$ be the constant in the
Poincar\'e inequality \eqref{l2-1-1}. Then,
by
Young's inequality, for every
$\left[\frac{(1+\gamma)\log_2 m}{d}\right] < n \le m$
 and
 $0\le t\le 2^{m\alpha}T$,
\begin{equation}\label{p6-1-4a}
\begin{split}
I_{m,n,1}(t)\le & 2C_1 2^{n\alpha}\sum_{z\in \Z_{m,n}^d}\int_{B_{2^n}(z)}\left|V(t,x)-\oint_{B_{2^m}}V(t,\cdot)\,d\mu\right|^2\,\mu(dx)\\
&+\frac{2^{-n\alpha}}{8C_1}\sum_{z\in \Z_{m,n}^d}\int_{B_{2^n}(z)}\left|\phi_m(t,x)- \oint_{B_{2^n(z)}}\phi_m(t,\cdot)\,d\mu
\right|^2\,\mu(dx)\\
\le &c_2 2^{md+n\alpha}+ \frac{1}{8}\sum_{i=1}^d\sum_{z\in \Z_{m,n}^d}\int_{B_{2^n}(z)}\int_{B_{2^n}(z)}(\phi_m^{(i)}(t,x)-\phi_m^{(i)}(t,y))^2\frac{w_{t,x,y}}{|x-y|^{d+\alpha}}\,\mu(dx)\,\mu(dy)\\
\le& c_2 2^{md+n\alpha}+\frac{1}{4} \E^\w_{t,B_{2^m}}\left(\phi_m(t,\cdot),\phi_m(t,\cdot)\right),\end{split}
\end{equation} where in the second inequality we used \eqref{l2-1-1} and the fact that $\sup_{t\in \R_+,x\in \Z^d}|V(t,x)|\le c_3$ due to {\bf Assumption (H)}(ii) and $\alpha\in (1,2)$, and the last inequality follows from the property that $$\sum_{z\in \Z_{m,n}^d} \E^\w_{t,B_{2^n}(z)}(\phi^{(i)}_m(t,\cdot),\phi^{(i)}_m(t,\cdot))
\le\E^\w_{t,B_{2^m}}(\phi^{(i)}_m(t,\cdot),\phi^{(i)}_m(t,\cdot))$$ thanks to $\sum_{z\in \Z_{m,n}^d}B_{2^n}(z)=B_{2^m}.$

On the other hand, by the Cauchy-Schwarz inequality,  for every $T>0$,
\begin{equation}\label{p6-1-4}
\begin{split}
&\int_0^{2^{m\alpha}T}
I_{m,n,2}(t)\,dt\\
&=c_1\sum_{k=n}^{m-1} 2^{k(d+\alpha)/2}\cdot  \int_0^{2^{m\alpha}T}  \mathscr{E}^\w_{t,B_{2^m}}(\phi_m(t,\cdot),\phi_m(t,\cdot)) ^{1/2}
\left(\sum_{y\in \Z_{m,k}^d}\left|\oint_{B_{2^k}(y)}V(t,\cdot)\,d\mu
-\oint_{B_{2^m}}V(t,\cdot)\,d\mu\right|^2\right)^{1/2}\,dt \\
&\le c_1\sum_{k=n}^m 2^{k(d+\alpha)/2}
\left(\int_0^{2^{m\alpha}T}  \mathscr{E}^\w_{t,B_{2^m}}(\phi_m(t,\cdot),\phi_m(t,\cdot)) \,dt\right)^{1/2}\\
&\quad\times
\left(\int_0^{2^{m\alpha}T}\sum_{y\in \Z_{m,k}^d}
\left|\oint_{B_{2^k}(y)}V(t,\cdot)\,d\mu
-\oint_{B_{2^m}}V(t,\cdot)\,d\mu\right|^2 \,dt\right)^{1/2}\\
&\le \sqrt{2}c_1\sum_{k=n}^m 2^{k(d+\alpha)/2}
\left(\int_0^{2^{m\alpha} T}  \mathscr{E}^\w_{t,B_{2^m}}(\phi_m(t,\cdot),\phi_m(t,\cdot)) \,dt\right)^{1/2}\\
&\quad \times
\left(\int_0^{2^{m\alpha} T}\sum_{y\in \Z_{m,k}^d}
\left(
\left|\oint_{B_{2^k}(y)}V(t,\cdot)\,d\mu\right|^2+\left|\oint_{B_{2^m}}V(t,\cdot)\,d\mu\right|^2\right) \,dt\right)^{1/2}.
\end{split}
\end{equation}

For every
$1\le k \le m$,
\begin{align*}
&\int_0^{2^{m\alpha}T}\sum_{y\in \Z^d_{m,k}}\left|\oint_{B_{2^k}(y)}V(t,\cdot)\,d\mu\right|^2\,dt\\
&=2^{-2kd}\sum_{i=1}^d\sum_{y\in \Z^d_{m,k}}\int_0^{2^{m\alpha}T}\sum_{z_1,z_2\in B_{2^k}(y)}V^{(i)}(t,z_1)V^{(i)}(t,z_2)\,dt\\
&=2^{-2 kd}\sum_{i=1}^d\sum_{y\in \Z^d_{m,k}}\int_0^{2^{m\alpha}T}\sum_{z_1,z_2\in B_{2^k}(y)}
\sum_{z_3,z_4\in \Z^d}\frac{(z_1-z_3)^{(i)}}{|z_1-z_3|^{d+\alpha}}
\frac{(z_2-z_4)^{(i)}}{|z_2-z_4|^{d+\alpha}}\xi_{t,z_1,z_3}\xi_{t,z_2,z_4}\,dt,
\end{align*}
where
$\xi_{t,x,y}:=w_{t,x,y}-\sK (t)$
for all $t\in \R_+$ and $x\neq y\in \Z^d$, and  in the last equality we used the fact
that
$$
V(t,x)=\sum_{z\in \Z^d}\frac{z}{|z|^{d+\alpha}}w_{t,x,x+z}=\sum_{z\in \Z^d}\frac{z}{|z|^{d+\alpha}}
\left(w_{t,x,x+z}-\Ee\left[w_{t,x,x+z}\right]\right)=\sum_{z\in \Z^d}\frac{z}{|z|^{d+\alpha}}\left(w_{t,x,x+z}-\sK (t)\right).
$$
Therefore,
\begin{equation}  \label{e:2.19}
\begin{split}&\Ee\left[\left(\sum_{y\in \Z^d_{m,k}}\int_0^{2^{m\alpha}T}\left|\oint_{B_{2^k}(y)}V(t,\cdot)\,d\mu\right|^2dt\right)^2\right]
\\
&= 2^{-4 kd}\sum_{i,i'=1}^d\sum_{y,y'\in \Z^d_{m,k}}\int_0^{2^{m\alpha}T}\int_0^{2^{\alpha m}T}
\sum_{z_1,z_2\in B_{2^k}(y)}\sum_{z_1',z_2'\in B_{2^k}(y')}
\sum_{z_3,z_4,z_3',z_4'\in \Z^d}
\frac{(z_1-z_3)^{(i)}(z_1'-z_3')^{(i')}}{(|z_1-z_3||z_1'-z_3'|)^{d+\alpha}}\\
& \qquad\qquad\qquad\qquad\qquad\qquad\qquad\qquad \times
\frac{(z_2-z_4)^{(i)}(z_2'-z_4')^{(i')}}{(|z_2-z_4||z_2'-z_4'|)^{d+\alpha}}
\Ee\left[\xi_{t,z_1,z_3}\xi_{t,z_2,z_4}\xi_{t',z_1',z_3'}\xi_{t',z_2',z_4'}\right]\,dt\,dt'.
\end{split}\end{equation}

Note  that $\Ee[\xi_{t,x,y}]=0$ for all $t\in \R_+$ and $(x,y)\in E$. According to {\bf Assumption (H)}(iv), it is easy to see that
$$
\Ee\left[\xi_{t,z_1,z_3}\xi_{t,z_2,z_4}\xi_{t',z_1',z_3'}\xi_{t',z_2',z_4'}\right]\neq 0
$$
only if in one of three cases:
\begin{itemize}
\item [(a)]   $(z_1,z_3)=(z_2,z_4)$, $(z_1',z_3')=(z_2',z_4')$;

\medskip

\item [(b)]  $(z_1,z_3)=(z_1',z_3')$, $(z_2,z_4)=(z_2',z_4')$;

\medskip

\item [(c)] $(z_1,z_3)=(z_2',z_4')$, $(z_2,z_4)=(z_1',z_3')$.
\end{itemize}
Note that in cases (b) and (c), $y$ has to be the same as $y'$ in the sum on the right hand side of \eqref{e:2.19}.
Hence,
\begin{align*}
& \Ee\left[\left(\sum_{y\in \Z^d_{m,k}}\int_0^{2^{m\alpha} T}\left|\oint_{B_{2^k}(y)}V(t,\cdot)\,d\mu\right|^2dt\right)^2\right]\le c_3
 \sum_{l=1}^2 J_{l,k,m},
\end{align*}
where
\begin{align*}
J_{1,k,m}:=& 2^{-4kd}\sum_{i,i'=1}^d\sum_{y,y'\in \Z^d_{m,k}}\int_0^{2^{m\alpha} T}
\int_0^{2^{m\alpha} T}\sum_{(z_1,z_3)\in E: z_1\in B_{2^k}(y)}
\sum_{(z_1',z_3')\in E: z_1'\in B_{2^k}(y')}\frac{|(z_1-z_3)^{(i)}|^2}{|z_1-z_3|^{2(d+\alpha)}}\\
&\qquad\qquad\qquad\qquad\qquad\qquad\qquad\qquad\qquad\qquad\times  \frac{|(z_1'-z_3')^{(i')}|^2}{|z_1'-z_3'|^{2(d+\alpha)}}
\Ee\left[|\xi_{t,z_1,z_3}|^2  \, |\xi_{t',z_1',z_3'}|^2\right]
\,dt\,dt'
\end{align*}
is the upper bound for case (a),  and
\begin{align*}
J_{2,k,m}:=& 2^{-4 kd }\sum_{i,i'=1}^d\sum_{y\in \Z^d_{m,k}}\int_0^{2^{m\alpha} T}
\int_0^{2^{m\alpha} T}\sum_{(z_1,z_3)\in E: z_1\in B_{2^k}(y)}
\sum_{(z_2,z_4)\in E: z_2\in B_{2^k}(y)}\\
&\qquad\times \frac{|(z_1-z_3)^{(i)}||(z_1-z_3)^{(i')}|}{ |z_1-z_3| ^{2(d+\alpha)}} \frac{|(z_2-z_4)^{(i)}||(z_2-z_4)^{(i')}|}{ |z_2-z_4| ^{2(d+\alpha)}}
\Ee\left[|\xi_{t,z_1,z_3}||\xi_{t,z_2,z_4}||\xi_{t',z_1,z_3}||\xi_{t',z_2,z_4}|\right]\,dt\,dt'
\end{align*}
is the combined bound for cases (b) and (c).
Since $|\xi_{t,x,y}|\le c_4$ for all $t\in \R_+$ and  $x,y\in \Z^d$ as well as
$d+2\alpha>2$ (this is always true because $\alpha\in (1,2)$),
we have for every
$1\le k \le m$,
\begin{align*}
J_{1,k,m}&\le c_52^{-4kd}2^{2m\alpha}T^2
\sum_{y,y'\in \Z^d_{m,k}}\sum_{(z_1,z_3)\in E: z_1\in B_{2^k}(y)}
\sum_{(z_1',z_3')\in E: z_1'\in B_{2^k}(y')}\frac{1}{(|z_1-z_3||z_1'-z_3'|)^{2(d+\alpha-1)}}\\
&\leq c_52^{-4kd}2^{2m\alpha}T^2 \,  \Big(  \sum_{z_1\in B_{2^m}}  \sum_{z_3 \in \Z^d: (z_1, z_3)\in E} \frac{1}{ |z_1-z_3|^{2(d+\alpha-1)}} \Big)^2
 \\
&\le   c_6  2^{-4kd+2md+2 m\alpha}T^2
\end{align*} and
\begin{align*}
J_{2,k,m}
&\le c_5 2^{-4kd} {2^{2m\alpha}T^2 } \sum_{y\in \Z^d_{m,k}}
\sum_{(z_1,z_3)\in E: z_1\in B_{2^k}(y)}
\sum_{(z_2,z_4)\in E: z_2\in B_{2^k}(y)}\frac{1}{(|z_1-z_3||z_2-z_4|)^{2(d+\alpha-1)}}\\
&\le c_6 2^{-4kd}   \, 2^{2 m\alpha} T^2 \, 2^{(m-k)d} 2^{2kd}
\le c_6 2^{-3kd+md+2m\alpha}T^2
\le c_62^{-4kd+2md+2 m\alpha}T^2.
\end{align*}
Putting all the estimates together yields that
for all $\left[\frac{(1+\gamma)\log_2 m}{d}\right]<k\le m$,
\begin{align*}
&
\Ee\left[\left(\sum_{y\in \Z_{m,k}^d}\int_0^{2^{m\alpha }T}\left|\oint_{B_{2^k}(y)}V(t,\cdot)\,d\mu\right|^2dt\right)^2\right]\le
c_72^{-2kd+2((m-k)d+m\alpha)}T^2.
\end{align*}

By the Markov inequality, we can derive that for every
$\theta_0\in ({\alpha}/{d},1)$
(note that we used $d>\alpha$ here)
and $\gamma>0$,
\begin{align*}
&\Pp\left(\bigcup_{m\ge 1}\bigcup_{[\frac{(1+\gamma)\log_2 m}{2(1-\theta_0)d}]\le k \le m}\left\{\left|
\sum_{y\in \Z_{m,k}^d}\int_0^{2^{m\alpha }T}\left|\oint_{B_{2^k}(y)}V(t,\cdot)\,d\mu\right|^2dt\right|>2^{m\alpha+(m-k)d-\theta_0 k d}T\right\}\right)\\
&\le \sum_{m=1}^\infty\sum_{k=[\frac{(1+\gamma)\log_2 m}{2(1-\theta_0)d}]}^m2^{2\theta_0 k d-2m\alpha-2(m-k)d}T^{-2}
\Ee\left[\left(\sum_{y\in \Z_{m,k}^d}\int_0^{2^{m\alpha }T}\left|\oint_{B_{2^k}(y)}V(t,\cdot)\,d\mu\right|^2dt\right)^2\right]\\
&\le c_7\sum_{m=1}^\infty \sum_{k=[\frac{(1+\gamma)\log_2 m}{2(1-\theta_0)d}]}^m 2^{-2k(1-\theta_0)d}\le c_8\sum_{m=1}^\infty m^{-(1+\gamma)}<\infty.
\end{align*}
This along with Borel-Cantelli's lemma yields that
for a.e. $\w\in \Omega$,
there is $m^*(\w)\ge 1$ such that for all $m\ge m^*(\w)$ and $[\frac{(1+\gamma)\log_2 m}{2(1-\theta_0)d}]\le
k\le m$,
\begin{equation}\label{p6-1-5--}
\sum_{y\in \Z_{m,k}^d}\int_0^{2^{m\alpha }T}\left|\oint_{B_{2^k}(y)}V(t,\cdot)\,d\mu\right|^2dt\le 2^{m\alpha+(m-k)d-\theta_0 k d}T.
\end{equation}

Following the same argument as above, for a.e. $\w\in \Omega$ we also can find  $m^{**}(\w)\ge 1$ so that for all
$m\ge m^{**}(\w)$ and
$[\frac{(1+\gamma)\log_2 m}{2(1-\theta_0)d}]\le
k\le m$,
\begin{equation}\label{p6-1-5}
\int_0^{2^{m\alpha}T}\left|\oint_{B_{2^m}}V(t,\cdot)\,d\mu\right|^2dt
\le 2^{m\alpha-\theta_0 m d}T.
\end{equation}
Putting \eqref{p6-1-5--} and \eqref{p6-1-5} into \eqref{p6-1-4},
we find that
for all
$n\ge \left[\frac{(1+\gamma)\log_2 m}{2(1-\theta_0)d}\right]$ and
$m\ge m_0(\w):=\max\{m^*(\w),m^{**}(\w)\},$
\begin{align*}
\int_0^{2^{m\alpha}T}I_{m,n,2}(t)\,dt
&\le c_9T^{1/2}\left(\sum_{k=[\frac{(1+\gamma)\log_2 m}{2(1-\theta_0)d}]}^m
2^{\frac{k(d+\alpha)+m\alpha+(m-k)d-\theta_0 k d}{2}}\right)\left(\int_0^{2^{m\alpha}T} \E^\w_{t,B_{2^m}}
\left(\phi_m(t,\cdot), \phi_m(t,\cdot)\right) dt\right)^{1/2}\\
&\le c_{10}T^{1/2}2^{\frac{md+m\alpha}{2}}\left(\sum_{k=1}^\infty 2^{-\frac{k(\theta_0 d-\alpha)}{2}}\right)
\left(\int_0^{2^{m\alpha}T} \E^\w_{t,B_{2^m}}
\left(\phi_m(t,\cdot), \phi_m(t,\cdot)\right) dt\right)^{1/2}\\
&\le c_{11}2^{m\alpha+md}T+\frac{1}{4}\int_0^{2^{m\alpha}T} \E^\w_{t,B_{2^m}}
\left(\phi_m(t,\cdot), \phi_m(t,\cdot)\right) dt,
\end{align*}
where in the last inequality we have used the fact that $\theta_0\in ({\alpha}/{d},1)$.

Combining the estimate above with \eqref{p6-1-4a} and \eqref{p6-1-3} and taking
$n=\max\left(\left[\frac{(1+\gamma)\log_2 m}{2(1-\theta_0)d}\right],\left[\frac{(1+\gamma)\log_2 m}{d}\right]\right)$,
we finally obtain that
\begin{align*}
\quad & \sup_{t\in [0,T]}\int_{B_{2^m}}|\phi_m(2^{m\alpha}t,x)|^2\mu(dx)\le -
\int_0^{2^{m\alpha}T} \E^\w_{t,B_{2^m}}\left(\phi_m(t,\cdot), \phi_m(t,\cdot)\right) dt+c_{12}m^{\frac{(1+\gamma)\alpha}{(2(1-\theta_0)d)\wedge d}}2^{m\alpha+md}T,
\end{align*}
where we used the fact that $\phi_m(0,x)=0$.
Taking
$\theta_0\in ({\alpha}/{d},1)$ close to ${\alpha}/{d}$ and changing $\gamma$ small enough if necessary,
we  get the desired conclusion \eqref{p6-1-1}.
\end{proof}

Next, we turn to the case that $\alpha\in (0,1]$. In this case, for any $m\ge1$  we consider the
following system of parabolic equations on $\R_+\times B_{2^m}$:
\begin{equation}\label{e6-7}
\begin{cases}
  \frac{\partial}{\partial t}\widetilde\phi_m(t,x)=\sL^\w_{t,B_{2^m}} \widetilde \phi_{m}(t,\cdot)(x)+V_m(t,x)
-
  \displaystyle \oint_{B_{2^m}} V_m(t,\cdot)\,d\mu,\quad  & (t,x)\in \R_+\times B_{2^m},\\
  \widetilde \phi_{m} (0,x)=0,
\end{cases}
\end{equation}
where
\begin{equation}\label{e6-7a}
V_m(t,x):=\sum_{z\in \Z^d: |z|\le 2^ m} \frac{z}{|z|^{d+\alpha}}w_{t,x,x+z}(\w).
\end{equation}

\begin{proposition}\label{p6-4}
Suppose that $\alpha\in (0,1]$ and  $d>\alpha$. Assume that  {\bf Assumption (H)} holds.  Let
$\widetilde \phi_m:\R_+\times
B_{2^m}\to \R^d$ be the solution to \eqref{e6-7}. Then, for
a.e. $\w\in \Omega$,
any $ \gamma>0$ and $T>0$, there exist a constant $C_4>0$ $($which is independent of $T)$ and a random variable $m_2(\omega)\ge1$ $($which
depends on $\gamma$ and $T$$)$ such that for every $m\ge m_2(\w)$,
\begin{equation}\label{p6-4-1}
\begin{split}
&\sup_{t\in [0,2^{m\alpha}T]}\int_{B_{2^m}}|\widetilde \phi_m(t,x)|^2\,\mu(dx)+ \int_0^{2^{m\alpha}T}
 \mathscr{E}^\w_{t,B_{2^m}}(\widetilde \phi_m(t,\cdot),\widetilde \phi_m(t,\cdot)) \,dt\\
 &\le C_4
\begin{cases}
m^{2+\frac{1+\gamma}{(2(d-1))\wedge d}}2^{m(d+1)}T,
&\quad \alpha=1,\\
m^{\frac{\alpha(1+\gamma)}{(2(d-\alpha))\wedge d}}2^{m(d+2(1-\alpha)+\alpha)}T,
 &\quad \alpha\in (0,1).\end{cases}\end{split}
\end{equation}
\end{proposition}
\begin{proof}
According to {\bf Assumption (H)}{(ii)}, we can easily see that
$\sup_{t\in \R_+,x\in \Z^d}|V_m(t,x)|\le c_1 m $ for all $m\ge1$ when $\alpha=1$, and $\sup_{t\in \R_+,x\in \Z^d}|V_m(t,x)|\le c_1 2^{m(1-\alpha)} $ for all $m\ge1$ when $\alpha\in (0,1)$.
On the other hand, note that \eqref{p6-1-3} still holds with $\phi_m(t,x)$ and $V(t,x)$ replaced by $\widetilde \phi_m(t,x)$ and $V_m(t,x)$,  respectively. Similarly, we can define $I_{m,n,1}(t)$ and $I_{m,n,2}(t)$ as the same way as that in the proof of Proposition \ref{p6-1} with
$\phi_m(t,x)$ and $V(t,x)$ replaced by $\widetilde \phi_m(t,x)$ and $V_m(t,x)$, respectively.

Suppose that $\alpha\in (0,1)$. Using the fact $\sup_{t\in \R_+,x\in \Z^d}|V_m(t,x)|\le c_1 2^{m(1-\alpha)}$
and following the argument for \eqref{p6-1-4a}, we can derive that
\begin{align}\label{p6-4-2}
I_{m,n,1}(t)&\le c_22^{md+2m(1-\alpha)+n\alpha}+\frac{1}{4}\E^\w_{t,B_{2^m}}\left(\widetilde\phi_m(t,\cdot),\widetilde\phi_m(t,\cdot)\right).
\end{align}

Furthermore, according to the argument for \eqref{e:2.19}, we obtain
\begin{align*}
& \Ee\left[\left(\sum_{y\in \Z^d_{m,k}}\int_0^{2^{m\alpha} T}\left|\oint_{B_{2^k}(y)}V_m(t,\cdot)\,d\mu\right|^2dt\right)^2\right]\le c_3
 \sum_{l=1}^2 J_{l,k,m},
\end{align*}
where
\begin{align*}
J_{1,k,m}:=& 2^{-4kd}\sum_{i,i'=1}^d\sum_{y,y'\in \Z^d_{m,k}}\int_0^{2^{m\alpha} T}
\int_0^{2^{m\alpha} T}\sum_{(z_1,z_3)\in E: z_1\in B_{2^k}(y),|z_1-z_3|\le 2^m}
\sum_{(z_1',z_3')\in E: z_1'\in B_{2^k}(y'),|z_1'-z_3'|\le 2^m}\frac{|(z_1-z_3)^{(i)}|^2}{|z_1-z_3|^{2(d+\alpha)}}\\
&\qquad\qquad\qquad\qquad\qquad\qquad\qquad\qquad\qquad\qquad\times  \frac{|(z_1'-z_3')^{(i')}|^2}{|z_1'-z_3'|^{2(d+\alpha)}}
\Ee\left[|\xi_{t,z_1,z_3}|^2  \, |\xi_{t',z_1',z_3'}|^2\right]
\,dt\,dt',
\end{align*}
\begin{align*}
J_{2,k,m}:=& 2^{-4 kd }\sum_{i,i'=1}^d\sum_{y\in \Z^d_{m,k}}\int_0^{2^{m\alpha} T}
\int_0^{2^{m\alpha} T}\sum_{(z_1,z_3)\in E: z_1\in B_{2^k}(y),|z_1-z_3|\le 2^m}
\sum_{(z_2,z_4)\in E: z_2\in B_{2^k}(y),|z_2-z_4|\le 2^m}\\
&\qquad\times \frac{|(z_1-z_3)^{(i)}||(z_1-z_3)^{(i')}|}{ |z_1-z_3| ^{2(d+\alpha)}} \frac{|(z_2-z_4)^{(i)}||(z_2-z_4)^{(i')}|}{ |z_2-z_4| ^{2(d+\alpha)}}
\Ee\left[|\xi_{t,z_1,z_3}||\xi_{t,z_2,z_4}||\xi_{t',z_1,z_3}||\xi_{t',z_2,z_4}|\right]\,dt\,dt'
\end{align*}
and $\xi_{t,x,y}:=w_{t,x,y}-
\sK (t)$.
By some direct computation, we get that for every $1\le k \le m$,
\begin{align*}
J_{1,k,m}&\le
 c_4 2^{-4kd}2^{2m\alpha}T^2 \,  \Big(  \sum_{z_1\in B_{2^m}}  \sum_{z_3 \in \Z^d: (z_1, z_3)\in E, |z_1-z_3|\le 2^m} \frac{1}{ |z_1-z_3|^{2(d+\alpha-1)}} \Big)^2\\
&\le   c_5  2^{-4kd+2md+2 m\alpha}\left(\max\{2^{m(4-4\alpha-2d)},1\}+(\log m)^2\I_{\{d=2-2\alpha\}}\right)T^2
\end{align*} and
\begin{align*}
J_{2,k,m}
&\le c_4 2^{-4kd} {2^{2m\alpha}T^2 } \sum_{y\in \Z^d_{m,k}}
\sum_{(z_1,z_3)\in E: z_1\in B_{2^k}(y),|z_1-z_3|\le 2^m}
\sum_{(z_2,z_4)\in E: z_2\in B_{2^k}(y),|z_2-z_4|\le 2^m}\frac{1}{(|z_1-z_3||z_2-z_4|)^{2(d+\alpha-1)}}\\
&\le c_52^{-3kd+md+2m\alpha}\left(\max\{2^{m(4-4\alpha-2d)},1\}+(\log m)^2\I_{\{d=2-2\alpha\}}\right)T^2.
\end{align*}
Hence, applying both estimates above and the Borel-Cantelli lemma as that in the proof for \eqref{p6-1-5--}, we can see that
for a.e. $\w\in \Omega$,
there is a random variable $m^*(\w)\ge 1$ so that for all $m\ge m^*(\w)$ and $[\frac{(1+\gamma)\log_2 m}{2(1-\theta_0)d}]\le
k\le m$ with $\theta_0\in (\alpha/d,1)$,
\begin{equation*}
\sum_{y\in \Z_{m,k}^d}\int_0^{2^{m\alpha }T}\left|\oint_{B_{2^k}(y)}V(t,\cdot)\,d\mu\right|^2dt\le 2^{m\alpha+(m-k)d-\theta_0 k d}
\left(\max\{2^{m(2-2\alpha-d)},1\}+(\log m)I_{\{d=2-2\alpha\}}\right)T.
\end{equation*}
Similarly, for a.e. $\w\in \Omega$ we also can find  $m^{**}(\w)\ge 1$ so that for all
$m\ge m^{**}(\w)$ and
$[\frac{(1+\gamma)\log_2 m}{2(1-\theta_0)d}]\le
k\le m$,
$$
\int_0^{2^{m\alpha}T}\left|\oint_{B_{2^m}}V(t,\cdot)\,d\mu\right|^2dt
\le  2^{m\alpha-\theta_0 k d}
\left(\max\{2^{m(2-2\alpha-d)},1\}+(\log m)\I_{\{d=2-2\alpha\}}\right)T.
$$
Therefore, putting these estimates  into \eqref{p6-1-4} (which still holds true), we find that
for all
$n\ge \left[\frac{(1+\gamma)\log_2 m}{2(1-\theta_0)d}\right]$ and
$m\ge m_0(\w):=\max\{m^*(\w),m^{**}(\w)\},$
\begin{align*}
\int_0^{2^{m\alpha}T}I_{m,n,2}(t)\,dt
&\le c_6\left(\max\{2^{m(1-\alpha-d/2)},1\}+(\log m)^{1/2}\I_{\{d=2-2\alpha\}}\right)T^{1/2}\left(\sum_{k=[\frac{(1+\gamma)\log_2 m}{2(1-\theta_0)d}]}^m
2^{\frac{k(d+\alpha)+m\alpha+(m-k)d-\theta_0 k d}{2}}\right)\\
&\quad\times\left(\int_0^{2^{m\alpha}T} \E^\w_{t,B_{2^m}}
\left(\widetilde\phi_m(t,\cdot), \widetilde\phi_m(t,\cdot)\right) dt\right)^{1/2}\\
&\le c_{7}T^{1/2}2^{(md+m\alpha)/{2}}\left(\max\{2^{m(1-\alpha-d/2)},1\}+(\log m)^{1/2}\I_{\{d=2-2\alpha\}}\right)\left(\sum_{k=1}^\infty 2^{-\frac{k(\theta_0 d-\alpha)}{2}}\right)\\
&\quad\times \left(\int_0^{2^{m\alpha}T} \E^\w_{t,B_{2^m}}
\left(\widetilde\phi_m(t,\cdot), \widetilde\phi_m(t,\cdot)\right) dt\right)^{1/2}\\
&\le c_{8}2^{m\alpha+md}\left(\max\{2^{m(2-2\alpha-d)},1\}+(\log m)\I_{\{d=2-2\alpha\}}\right)T+\frac{1}{4}\int_0^{2^{m\alpha}T} \E^\w_{t,B_{2^m}}
\left(\widetilde\phi_m(t,\cdot), \widetilde\phi_m(t,\cdot)\right) dt.
\end{align*}

Combining this with \eqref{p6-1-3} and \eqref{p6-4-2} and taking
$n=\max\left(\left[\frac{(1+\gamma)\log_2 m}{2(1-\theta_0)d}\right],\left[\frac{(1+\gamma)\log_2 m}{d}\right]\right)$
yield that
\begin{align*}
\quad & \sup_{t\in [0,T]}\int_{B_{2^m}}|\widetilde\phi_m(2^{m\alpha}t,x)|^2\mu(dx)\le -
\int_0^{2^{m\alpha}T} \E^\w_{t,B_{2^m}}\left(\widetilde\phi_m(t,\cdot), \widetilde\phi_m(t,\cdot)\right) dt+
c_{9}m^{\frac{(1+\gamma)\alpha}{(2(1-\theta_0)d)\wedge d}}2^{m\alpha+2m(1-\alpha)+md}T,
\end{align*}
where we have used  the fact that $2(1-\alpha-d)<2(1-\alpha)$.
So we prove \eqref{p6-4-1} for the case $\alpha\in (0,1)$.

For the case $\alpha=1$, we can prove the desired conclusion by almost the same way, so we omit
the details.
\end{proof}

\begin{remark}\label{r2-5} \rm
 The proofs of two propositions above are based on the boundedness assumption of $\{w_{t,x,y}: t\in \R_+, (x,y)\in E\}$. In particular, we use this for the estimate of the first term in the right hand side of \eqref{p6-1-4a}. Similarly, this condition
is used for
 the estimate \eqref{p6-4-2}.
Therefore, while it is possible to weaken the boundedness assumption on $\{w_{t,x,y}: t\in \R_+, (x,y)\in E\}$
 to the moment condition
 such like
 \eqref{e:2.4},
  the large scale estimates for the associated correctors would be larger, and so the convergence rates in Theorem \ref{T:1.1}
  would be worse.
\end{remark}

 \section{Quantitative differences between the scaled operators $\{\sL_t^{k,\w}\}_{k\ge1}$ and the limit $\bar \sL$}\label{section3}

In this section, we aim to establish quantitative differences
  between   the scaled operators $\{\sL_t^{k,\w}\}_{k\ge1}$ and the limit operator $\bar \sL$. Recall that for any $f\in C_b^2(\R^d)$,
\begin{align*}
\bar \sL f(x)=
\begin{cases}
\displaystyle\int_{\R^d} \left(f(x+z)-f(x)-\langle\nabla f(x), z\rangle \I_{\{|z|\le 1\}}\right)\frac{\sK }{|z|^{d+\alpha}}\,dz,\quad &\alpha\in (0,1],\\
\displaystyle\int_{\R^d} \left(f(x+z)-f(x)-\langle\nabla f(x), z\rangle \right)\frac{\sK }{|z|^{d+\alpha}}\,dz,\quad &\alpha\in (1,2),
\end{cases}
\end{align*}
where $\sK =\lim_{t \to \infty}\frac{1}{t}\int_0^t \sK (s)\,ds$ is defined by \eqref{a1-1-0a} and $\sK (t)=\Ee w_{t,x,y}$ for all $t>0$ and $(x,y)\in E$.
It is also natural to consider its corresponding discrete approximation; that is, for any $f\in {\mathcal B}_b(k^{-1}\Z^d)$ (the space of bounded measurable functions on $k^{-1}\Z^d$),  define
\begin{equation}\label{e3-1a}
\begin{split}
\bar \sL^{k}_tf(x)= &k^{-d}\sum_{z\in k^{-1}\Z^d}\left(f(x+z)-f(x)\right)\frac{\sK (k^\alpha t)}{|z|^{d+\alpha}}\\
  =&\begin{cases} k^{-d}\displaystyle\sum_{z\in k^{-1}\Z^d}\left(f(x+z)-f(x)-\langle\nabla f(x), z\rangle \I_{\{|z|\le 1\}}\right)\frac{\sK (k^\alpha t)}{|z|^{d+\alpha}},\quad &\alpha\in (0,1],\\
  k^{-d}\displaystyle\sum_{z\in k^{-1}\Z^d}\left(f(x+z)-f(x)-\langle\nabla f(x), z\rangle  \right)\frac{\sK (k^\alpha t)}{|z|^{d+\alpha}},\quad &\alpha\in (1,2).
\end{cases}
\end{split}
\end{equation}

 \begin{proposition}\label{p6-3--}
Suppose that $f\in C_b^2(\R^d)$ satisfies
\begin{equation}\label{p6-5-1}
|\nabla^i f(x)|\le
C_0 ((1+|x|)^{-d-\beta}+\I_{\{|x|\le R_0\}}(x) ),\quad x\in \R^d,\, i=0,1,2
\end{equation}
for some positive constants $C_0$, $R_0$ and $\beta\in (0,\infty]$.
Then there exists a positive constant $C_1:=C_1(C_0,R_0,\beta)$ such that  for all $k\ge 1$,
\begin{equation}\label{p6-3---}
\begin{split}
&\int_0^T\int_{k^{-1}\Z^d} |  \bar \sL_t^{k}f(x)-  \bar \sL f(x)|^2\,\mu^{ (k)}(dx)dt
\le C_1T\pi\left(k^\alpha T\right)+C_1T
\begin{cases}k^{-2},\quad  &\alpha\in (0,1),\\
 k^{-2}(\log k)^2,\quad &\alpha=1,\\
 k^{-2(2-\alpha)},\quad &\alpha\in (1,2).
\end{cases}
\end{split}
\end{equation}
\end{proposition}

\begin{proof}
Define
\begin{align*}
\bar \sL^{k,0}f(x)= &k^{-d}\sum_{z\in k^{-1}\Z^d}\left(f(x+z)-f(x)\right)\frac{\sK }{|z|^{d+\alpha}}\\
  =&\begin{cases} k^{-d}\displaystyle\sum_{z\in k^{-1}\Z^d}\left(f(x+z)-f(x)-\langle\nabla f(x), z\rangle \I_{\{|z|\le 1\}}\right)\frac{\sK }{|z|^{d+\alpha}},\quad &\alpha\in (0,1],\\
  k^{-d}\displaystyle\sum_{z\in k^{-1}\Z^d}\left(f(x+z)-f(x)-\langle\nabla f(x), z\rangle  \right)\frac{\sK }{|z|^{d+\alpha}},\quad &\alpha\in (1,2).
\end{cases}
\end{align*}
Using the mean value theorem, we can obtain that for all $t\in \R_+$,
\begin{align*}
&\int_{k^{-1}\Z^d} |  \bar \sL_t^{k}f(x)-  \bar \sL^{k,0} f(x)|^2\,\mu^{ (k)}(dx)\\
&\le c_1\left((1+|x|)^{-d-\beta}+\I_{\{|x|\le R_0\}}(x)\right)
\left(k^{-d}\sum_{z\in k^{-1}\Z^d}\left(1\wedge |z|^2\right)\frac{1}{|z|^{d+\alpha}}+
k^{-d}\sum_{z\in k^{-1}\Z^d:|z|>1}(1+|x+z|)^{-d-\beta}|z|^{-d-\alpha}\right)\\
&\quad\times \left|\sK (k^\alpha t)-\sK \right|^2\\
&\le c_2\left((1+|x|)^{-d-\beta}+(1+|x|)^{-d-\alpha}\right) \left|\sK (k^\alpha t)-\sK \right|^2,
\end{align*}
where in the last inequality we used
\begin{align*}
&k^{-d}\sum_{z\in k^{-1}\Z^d:|z|>1}(1+|x+z|)^{-d-\beta}|z|^{-d-\alpha}\\
&\le k^{-d}
\left(\sum_{z\in k^{-1}\Z^d:|z|>\frac{1+|x|}{2}}+\sum_{z\in k^{-1}\Z^d:|z|\le \frac{1+|x|}{2},|z|>1}\right)
(1+|x+z|)^{-d-\beta}|z|^{-d-\alpha}\\
&\le c_3\left((1+|x|)^{-d-\beta}+(1+|x|)^{-d-\alpha}\right).
\end{align*}
Hence,
\begin{align*}
\int_0^T\int_{k^{-1}\Z^d} |  \bar \sL_t^{k}f(x)-  \bar \sL^{k,0} f(x)|^2\,\mu^{ (k)}(dx)
dt\le c_4T\pi(k^\alpha T).
\end{align*}
Thus, in order to obtain the desired assertion we only need to verify that
\begin{equation}\label{p6-5-1b}
\int_{k^{-1}\Z^d} |  \bar \sL^{k,0}f(x)-  \bar \sL f(x)|^2\,\mu^{ (k)}(dx)
\le c_5
\begin{cases}k^{-2},\quad  &\alpha\in (0,1),\\
 k^{-2}(\log k)^2,\quad &\alpha=1,\\
 k^{-2(2-\alpha)},\quad &\alpha\in (1,2).
\end{cases}
\end{equation}

When $\beta=\infty$, $f\in C_c^2(\R^d)$ and the estimate
\eqref{p6-5-1b}
has been established in
\cite[Proposition 4.1]{CCKW3}. So it remains to prove this for the case that $\beta\in (0,\infty)$.

{\bf Case 1}: $\alpha\in (1,2)$.
Set
\begin{align*}
H(x;z):=\frac{\delta f(x;z)}{|z|^{d+\alpha}},\quad x,z\in \R^d,
\end{align*}
where
$
\delta f(x;z):=f(x+z)-f(x)-\langle \nabla f(x), z \rangle.
$
For every $k\ge 4\sqrt{d}$, $x,z\in k^{-1}\Z^d$ and $y\in \prod_{1\le i \le d}(z_i,z_i+k^{-1}]$ (here we use the notation $z=(z_1,\cdots,z_d)\in k^{-1}\Z^d$), by
the mean value theorem, we have
\begin{align*}
&\left|H(x;z)-H(x;y)\right|\\
&\le \frac{1}{2}\sup_{x_1\in \R^d:|x_1-x|\le \frac{2\sqrt{d}}{k}}\left|\nabla^2 f(x_1)\right|\cdot\left(|z|^{2-d-\alpha}+|y|^{2-d-\alpha}\right)\I_{\{|z|\le \frac{2\sqrt{d}}{k}\}}\\
&\quad+\Bigg(\left|\delta f(x;y)\right|\cdot\left||y|^{-d-\alpha}-|z|^{-d-\alpha}\right|+
\left|\delta f(x;z)-\delta f(x;y)\right|\cdot |z|^{-d-\alpha}\Bigg)\I_{\{|z|>\frac{2\sqrt{d}}{k}\}}\\
&\le c_6(1+|x|)^{-d-\beta}\left(|z|^{2-d-\alpha}+|y|^{2-d-\alpha}\right)\I_{\{|z|\le \frac{2\sqrt{d}}{k}\}}
+c_6\Bigg[\sup_{x_1\in \R^d:|x_1-x|\le \frac{1+|x|}{2}}\left|\nabla^2 f(x_1)\right|\cdot \frac{|z-y|}{|y|^{d+\alpha-1}}\\
&\quad\quad
+\left(\int_0^1\int_0^1\left|\nabla^2 f\left(x+t(y+s(z-y))\right)\right|dsdt\right)\cdot \frac{|z-y|}{|y|^{d+\alpha-1}}\Bigg]\I_{\{\frac{2\sqrt{d}}{k}\le |z|\le \frac{1+|x|}{2}\}}\\
&\,\, +c_6\Bigg[\Big(|f(x+y)|+|f(x)|+|\nabla f(x)||y|\Big)\cdot\frac{|z-y|}{|y|^{d+\alpha+1}}+
\left(\int_0^1\left|\nabla f\left(x+y+s(z-y)\right)-\nabla f(x)\right|ds\right)
\cdot \frac{|z-y|}{|y|^{d+\alpha}}\Bigg]\I_{\{|z|>\frac{1+|x|}{2}\}}\\
&\le c_7(1+|x|)^{-d-\beta}\Bigg(\left(|z|^{2-d-\alpha}+|y|^{2-d-\alpha}\right)\I_{\{|z|\le \frac{2\sqrt{d}}{k}\}}+k^{-1}|y|^{-d-\alpha+1}\I_{\{|z|>\frac{2\sqrt{d}}{k}\}}\Bigg)\\
&\quad +c_7k^{-1}(1+|x+y|)^{-d-\beta}|y|^{-d-\alpha}\I_{\{|z|>\frac{1+|x|}{2}\}}.
\end{align*}
Hence, for every $x\in k^{-1}\Z^d$ and $k\ge 4\sqrt{d}$,
\begin{align*}
&|\bar \sL^{k,0} f(x)-\bar \sL f(x)|\\
&=\left|\sum_{z\in k^{-1}\Z^d}\int_{\prod_{1\le i \le d}(z_i,z_i+k^{-1}]}\left(H(x;z)-H(x;y)\right)dy\right|\\
&\le c_8(1+|x|)^{-d-\beta}\Bigg(\int_{\{|y|\le \frac{3\sqrt{d}}{k}\}}|y|^{2-d-\alpha}dy+k^{-d}\sum_{z\in k^{-1}\Z^d: |z|\le \frac{3\sqrt{d}}{k}}|z|^{2-d-\alpha}
+k^{-1}\int_{\{|y|>\frac{\sqrt{d}}{k}\}}|y|^{-d-\alpha+1}dy\Bigg)\\
&\quad +c_8k^{-1}\int_{\{|y|>\frac{1+|x|}{3}\}}(1+|x+y|)^{-d-\beta}|y|^{-d-\alpha}dy\\
&\le c_9\left(k^{-(2-\alpha)}(1+|x|)^{-d-\beta}+k^{-1}(1+|x|)^{-d-\alpha}\right).
\end{align*}
This implies that \eqref{p6-3---} holds for all $k\ge 4\sqrt{d}$.

On the other hand, by some direct computation, it is easy to verify that for every $f\in C_b^2(\R^d)$ satisfying \eqref{p6-5-1},
\begin{equation}\label{e:remark1}
\int_{k^{-1}\Z^d}  | \bar \sL^{k,0} f(x)-\bar \sL f(x) |^2\,
\mu^{(k)}(dx)\le 2\int_{k^{-1}\Z^d} |\bar \sL^{k,0}f(x)|^2\,
\mu^{(k)}(dx)+ 2\int_{k^{-1}\Z^d} |\bar \sL^{k,0} f(x)|^2\,
\mu^{(k)}(dx)\le c_{10}.
\end{equation}
Therefore, the desired assertion \eqref{p6-3---} for all $k\ge 1$ follows immediately.

{\bf Case 2:} $\alpha\in (0,1]$. In this case, we define
\begin{align*}
\hat H(x;z):
=\frac{\hat \delta f(x;z)}{|z|^{d+\alpha}},\quad x,z\in \R^d,
\end{align*}
where $\hat \delta f(x;z):=f(x+z)-f(x)-
\langle \nabla f(x), z\I_{\{|z|\le 1\}}\rangle$.
By the mean value theorem and the same arguments as in {\bf Case 1},
we find that for every $k\ge 4\sqrt{d}$, $x,z\in k^{-1}\Z^d$ and $y\in  \prod_{1\le i\le d}(z_i,z_i+k^{-1}]$,
\begin{align*}
&|\hat H(x;z)-\hat H(x;y)|\\
&\le \frac{1}{2}\sup_{x_1\in \R^d:|x_1-x|\le \frac{2\sqrt{d}}{k}}\left|\nabla^2 f(x_1)\right|\cdot\left(|z|^{2-d-\alpha}+|y|^{2-d-\alpha}\right)\I_{\{|z|\le \frac{2\sqrt{d}}{k}\}}\\
&\quad+\Bigg(\left|\hat \delta_k f(x;y)\right|\cdot\left||y|^{-d-\alpha}-|z|^{-d-\alpha}\right|+
\left|\hat\delta_k f(x;z)-\hat\delta_k f(x;y)\right|\cdot |z|^{-d-\alpha}\Bigg)\I_{\{|z|>\frac{2\sqrt{d}}{k}\}}\\
&\quad+\Big(\frac{|\langle \nabla f(x), z\rangle|}{|z|^{d+\alpha}}\I_{\{1-1/k\le |z|\le1\}}+
\frac{|\langle \nabla f(x), y\rangle|}{|y|^{d+\alpha}}\I_{\{1-(1+\sqrt{d})/k\le |y|\le 1+(1+\sqrt{d})/k\}}\Big)\\
&\le c_{11}(1+|x|)^{-d-\beta}\Bigg(\left(|z|^{2-\alpha-d}+|y|^{2-\alpha-d}\right)\I_{\{|z|\le 2\sqrt{d}/k\}}+
k^{-1}|y|^{-d-\alpha+1}\I_{\{2\sqrt{d}/k<|z|\le 1- 1/k\}}\\
&\quad\quad+ \left(k^{-1}|y|^{-d-1-\alpha}+k^{-1}|y|^{-d-\alpha}\right)\I_{\{|z|> 1- 1/k\}}+c_{11}\left(\I_{\{1-1/k\le |z|\le1\}}+\I_{\{1-(1+\sqrt{d})/k\le |y|\le 1+(1+\sqrt{d})/k\}}\right)\Bigg)\\
&\quad +c_{11}k^{-1}(1+|x+y|)^{-d-\beta}(1+|y|)^{-d-\alpha}\I_{\{|z|> {(1+|x|)}/{2}\}},
\end{align*}
where $\hat \delta_k f(x;z):=f(x+z)-f(x)-
\left\langle \nabla f(x), z\I_{\{|z|\le 1-{1}/{k}\}}\right\rangle$.

When $\alpha=1$, we derive that for every $x\in k^{-1}\Z^d$ and $k\ge 4\sqrt{d}$,
\begin{align*}
|  \bar \sL^{k,0} f(x)
-\bar \sL f(x)|
&=\left|\sum_{z\in k^{-1}\Z^d}\int_{\prod_{1\le i \le d}(z_i,z_i+k^{-1}]}\left(\hat H(x;z)-\hat H(x;y)\right)dy\right|\\
&\le c_{12}(1+|x|)^{-d-\beta}\Bigg(\int_{\{|y|\le 3\sqrt{d}/k\}}|y|^{1-d}\,dy+
k^{-d}\sum_{z\in k^{-1}\Z^d:|z|\le 2\sqrt{d}/k}|z|^{1-d}\\
&\quad\quad+
 k^{-1}\int_{\{\sqrt{d}k^{-1}<|y|\le 2\}}
|y|^{-d}\,dy
+  k^{-1}\int_{\{|y|>1-2\sqrt{d}/k\}}|y|^{-d-1}\, dy+k^{-1}\Bigg)\\
&\quad+c_{12}k^{-1}\int_{\{|y|>\frac{1+|x|}{3}\}}(1+|x+y|)^{-d-\beta}(1+|y|)^{-d-1}dy\\
&\le  c_{13}\left((k^{-1}\log k) (1+|x|)^{-d-\beta}+k^{-1}(1+|x|)^{-d-1}\right),
\end{align*}
which implies \eqref{p6-5-1b} for every $k\ge 4\sqrt{d}$.

Following similar arguments as above, we can show that \eqref{p6-5-1b} for all $k\ge 4\sqrt{d}$ when $\alpha\in (0,1)$.

Combining both estimates above with \eqref{e:remark1},  we prove \eqref{p6-5-1b} when $\alpha\in (0,1]$. The proof is complete.
\end{proof}

Next, we will establish the following statement.

\begin{proposition}\label{p6-3} Suppose that {\bf Assumption (H)} holds and that
$\alpha\in (0,1)$.
For
a.e. $\w\in \Omega$,
any $\theta\in (0,1)$,
$C_1, R_0>0$ and $ \beta\in (0,\infty]$,
 there exist a constant
$C_2:=C_2(\theta, C_1, R_0, \beta)>0$
$($independent of $T)$ and a random variable $k_1(\w)\ge1$ such that for all $k\ge k_1(\w)$ and
all  $f\in C\left([0,T];C_b^2(\R^d)\right)$ satisfying that
\begin{equation}\label{p6-3-1a}
|\nabla^i f(t,x)|\le
C_1\left((1+|x|)^{-d-\beta}+\I_{\{|x|\le R_0\}}(x)\right),
\quad t\in[0,T],\quad x\in \R^d,\, i=0,1,2,
\end{equation}
it holds that
\begin{equation}\label{p6-3-1}
\int_0^T \int_{k^{-1}\Z^d} |\sL^{k,\w}_t f(t,\cdot)(x)- \bar \sL_t^{k}f(t,\cdot)(x)
|^2\,\mu^{(k)}(dx)\,dt\le C_2 \max\{k^{- 2\theta(1-\alpha)},k^{-\theta d}\}T.
\end{equation}
\end{proposition}

On the other hand, due to the corrector approach we adopt later we shall consider the following variant of the scaled operators  $\{\sL_t^{k,\w}\}_{k\ge1}$.
For every $f\in {\mathcal B}_b(k^{-1}\Z^d)$, define
\begin{equation}\label{e6-6}
 \widehat \sL^{k,\w}_t f(x):=\begin{cases}
\displaystyle  k^{-d}\sum_{z\in k^{-1}\Z^d}\left(f(x+z)-f(x)-\langle \nabla f(x), z \rangle\I_{\{|z|\le 1\}}\right)\frac{w_{k^\alpha t,kx,k(x+z)}}{|z|^{d+\alpha}},&\quad \alpha\in (0,1],\\
\displaystyle k^{-d}\sum_{z\in k^{-1}\Z^d}\left(f(x+z)-f(x)-\langle \nabla f(x), z\rangle\right)\frac{w_{k^\alpha t,kx,k(x+z)}}{|z|^{d+\alpha}},&\quad \alpha\in (1,2).\end{cases}\end{equation}

\begin{proposition}\label{p6-2}  Suppose that {\bf Assumption (H)} holds.
For
a.e. $\w\in \Omega$,
every $\theta\in (0,1)$,
  $C_1,  R_0>0$ and $\beta\in (0,\infty]$,
there exist a constant
 $C_3:= C_3(\theta, C_1, R_0, \beta)>0$
$($independent of $T)$ and a random variable $k_2(\w)\ge 1$ such that
for every $k\ge k_2(\w)$ and every $f\in C\left([0,T];C_b^2(\R^d)\right)$ satisfying
\eqref{p6-3-1a}, it holds that
\begin{equation}\label{p6-2-1}
\begin{split}
\int_0^{T}\int_{k^{-1}\Z^d} |\widehat \LL_t^{k,\w} f(t,\cdot)(x)- \bar \LL_t^k f(t,\cdot)(x)
|^2\,\mu^{ (k)}(dx)\,dt
\le C_3\max\{
k^{-2\theta(2-\alpha)},k^{-\theta d}\}T.
\end{split}
\end{equation}
\end{proposition}

It is clearly seen from \eqref{p6-3-1} and \eqref{p6-2-1} that
the non-divergence form operator version
of the original symmetric
scaled operators  $\{\widehat \LL_t^{k,\w}\}_{k\ge1}$ approximate the discrete operator $\bar \LL^k_t$ better than the scaled operators  $\{\sL_t^{k,\w}\}_{k\ge1}$. In particular, the estimate
\eqref{p6-3-1} is useful only when $\alpha\in (0,1)$. As mentioned above, the introduction of
the non-divergence form operator version
of the scaled operators  $\{\widehat \LL_t^{k,\w}\}_{k\ge1}$ arises from the corrector approach adopted below, which is essential for the case
when
$\alpha\in [1,2)$.

\smallskip

 In the following,
 we give the proofs of the
 above two propositions.
 For the notational simplicity, we will suppress
the random environment
$\w \in \Omega$ from the notations.

\begin{proof}[Proof of Proposition $\ref{p6-2}$]
We will only prove the conclusion for $\beta<\infty$. The case that $\beta=\infty$
can be verified by exactly the same way.

{\bf Case 1:} We first consider the case that $\alpha\in (1,2)$. Let $f\in C\left([0,T];C_b^2(\R^d)\right)$ satisfy
\eqref{p6-3-1a}.
For every $t\ge0 $ and $x,y\in \Z^d$,  define
$$\eta_k(t,x,y):=f(t,k^{-1}y)-f(t,k^{-1}x)-k^{-1}\left\langle\nabla f(t,k^{-1}x), y-x\right\rangle,
\quad \xi_{t,x,y}=w_{t,x,y}-\sK (t).
$$
  For all $t\geq 0$
 and $x\in \Z^d$,
$$
 \widehat \LL_t^{k}f(t,\cdot)(k^{-1}x)=k^{\alpha}
 \sum_{z\in \Z^d \setminus \{x\} }
 \frac{\eta_k(t,x,z)}{|x-z|^{d+\alpha}}w_{k^\alpha t,x,z},\qquad  \bar \LL^{k}f(t,\cdot)(k^{-1}x)=k^{\alpha}
 \sum_{z\in \Z^d \setminus \{x\} }
 \frac{\eta_k(t,x,z)}{|x-z|^{d+\alpha}}
$$ and so
\begin{equation}\label{p6-2-2}
\begin{split}
&\int_0^{T}\int_{k^{-1}\Z^d} |\widehat \LL_t^{k}f(t,x)-\bar \LL^{k}_tf(t,x) |^2\,\mu^{(k)}(dx)\,dt\\
&=k^{-(d-2\alpha)}\int_0^T\sum_{x\in \Z^d}\left(\sum_{y\in \Z^d\setminus \{x\} }\frac{\eta_k(t,x,y)}{|x-y|^{d+\alpha}}\xi_{k^\alpha t,x,y}\right)^2\,dt\\
&=k^{-(d-\alpha)}\int_0^{k^\alpha T}\sum_{x\in \Z^d}\sum_{y_1,y_2\in  \Z^d\setminus \{x\}}
\frac{\eta_k(k^{-\alpha}t,x,y_1)}{|x-y_1|^{d+\alpha}}\frac{\eta_k(k^{-\alpha}t,x,y_2)}{|x-y_2|^{d+\alpha}}\xi_{t,x,y_1}\xi_{t,x,y_2}\,dt.
\end{split}
\end{equation}
As in the proof of Proposition \ref{p6-1},
\begin{align*}
&\Ee\left[\left| \int_0^{k^\alpha T}
k^{-(d-\alpha)}\sum_{x\in \Z^d}\sum_{y_1,y_2\in  \Z^d\setminus \{x\}}
\frac{\eta_k(k^{-\alpha}t,x,y_1)}{|x-y_1|^{d+\alpha}}\frac{\eta_k(k^{-\alpha}t,x,y_2)}{|x-y_2|^{d+\alpha}}\xi_{t,x,y_1}\xi_{t,x,y_2}\,dt\right|^2\right]\\
&=k^{-2(d-\alpha)}
\sum_{x,x'\in \Z^d}
\sum_{y_1,y_2\in  \Z^d\setminus \{x\}}\sum_{y_1',y_2'\in  \Z^d\setminus \{x'\}}  \int_0^{k^\alpha T} \int_0^{k^\alpha T}
\frac{\eta_k(k^{-\alpha}t,x,y_1)}{|x-y_1|^{d+\alpha}}\frac{\eta_k(k^{-\alpha}t,x,y_2)}{|x-y_2|^{d+\alpha}}\\
&\qquad\times\frac{\eta_k(k^{-\alpha}t',x',y_1')}{|x'-y_1'|^{d+\alpha}}\frac{\eta_k(k^{-\alpha}t',x',y_2')}{|x'-y_2'|^{d+\alpha}}
\Ee\left[\xi_{t,x,y_1}\xi_{t,x,y_2}\xi_{t',x',y_1'}\xi_{t',x',y_2'}\right]\,dt\,dt'.
\end{align*}
 Applying
the same argument as that in the proof of Proposition \ref{p6-1} for the estimates of
the non-zero
  terms
of $\Ee\left[\xi_{t,x,y_1}\xi_{t,x,y_2}\xi_{t',x',y_1'}\xi_{t',x',y_2'}\right]$, we can get
 \begin{align}
&\Ee\left[\left| \int_0^{k^\alpha T}
k^{-(d-\alpha)}\sum_{x\in \Z^d}\sum_{y_1,y_2\in  \Z^d\setminus \{x\}}
\frac{\eta_k(k^{-\alpha}t,x,y_1)}{|x-y_1|^{d+\alpha}}\frac{\eta_k(k^{-\alpha}t,x,y_2)}{|x-y_2|^{d+\alpha}}\xi_{t,x,y_1}\xi_{t,x,y_2}\,dt\right|^2\right]\nonumber\\
&\le c_1k^{-2(d-\alpha)} \sum_{x,x'\in \Z^d}\sum_{y_1\in \Z^d\setminus \{x\}}\sum_{y_1'\in \Z^d\setminus \{x'\}}
  \int_0^{k^\alpha T}  \int_0^{k^\alpha T}\frac{|\eta_k(k^{-\alpha}t,x,y_1)|^2}{|x-y_1|^{2(d+\alpha)}}
\frac{|\eta_k(k^{-\alpha}t',x',y_1')|^2}{|x'-y_1'|^{2(d+\alpha)}}\,dt\,dt'\nonumber\\
&\quad+c_1k^{-2(d-\alpha)} \sum_{x \in \Z^d}\sum_{y_1,y_2\in \Z^d\setminus\{x\}}
  \int_0^{k^\alpha T}
 \int_0^{k^\alpha T}
\frac{|\eta_k(k^{-\alpha}t,x,y_1)|^2}{|x-y_1|^{2(d+\alpha)}}
\frac{|\eta_k(k^{-\alpha}t',x,y_2)|^2}{|x-y_2|^{2(d+\alpha)}}
\,dt\,dt'.\label{*3-8}
\end{align}
Note that, by \eqref{p6-3-1a} and the mean value theorem, for every
$t\ge0$ and $x,y\in \Z^d$,
\begin{align*}
|\eta_k(t,x,y)|
&\le c_2\left(1+|k^{-1}x|\right)^{-d-\beta}\left(k^{-2}|x-y|^2\I_{\{|x-y|\le k\}} +k^{-1}|x-y|\I_{\{|x-y|> k\}} \right)\\
&\quad +c_2\left(\left(1+|k^{-1}x|\right)^{-d-\beta}
+\left(1+|k^{-1}y|\right)^{-d-\beta}\right)\I_{\{|x-y|> k\}}.
\end{align*}
Thus, putting this into the estimate above, we arrive at
\begin{align*}
&\mbox{LHS of \eqref{*3-8}}\\
&\le c_2k^{-2d+4\alpha}T^2\Bigg[\sum_{x\in \Z^d}(1+|k^{-1}x|)^{-2d-2\beta}\Bigg(k^{-4}\sum_{y\in \Z^d: |x-y|\le k}|x-y|^{-2d-2\alpha+4}\\
&\qquad\qquad\qquad \qquad +\sum_{y\in \Z^d: |x-y|>k}\left(k^{-2}|x-y|^{-2d-2\alpha+2}+|x-y|^{-2d-2\alpha}\right)\Bigg)
\Bigg]^2\\
&\quad+ c_2k^{-2d+3\alpha}T\Bigg[\sum_{x\in \Z^d}(1+|k^{-1}x|)^{-4d-4\beta}\Bigg(k^{-4}\sum_{y\in \Z^d: |x-y|\le k}|x-y|^{-2d-2\alpha+4}
+k^{-2}\sum_{y\in \Z^d: |x-y|>k}|x-y|^{-2d-2\alpha+2}\\
&\qquad\qquad\qquad\qquad+\sum_{y\in \Z^d: |x-y|>k}|x-y|^{-2d-2\alpha}\Bigg)^2+\sum_{x\in \Z^d}\left(\sum_{y\in \Z^d:|y-x|>k}(1+|k^{-1}y|)^{-2d-2\beta}|x-y|^{-2d-2\alpha}\right)^2\Bigg]\\
&\le c_3\left(\max\{k^{-(8-4\alpha)},k^{-2d}\}+k^{-2d}(\log k)^2\I_{\{d=4-2\alpha\}}\right)T^2.
\end{align*}

Following the same procedure (in particular, to estimate
the non-zero term of
$\Ee\left[\prod_{i=1}^n\xi_{t,x,y_i}\prod_{j=1}^n\xi_{t',x',y_j'}\right]$)
as above, we  derive that
for every $n\ge1$,
there exists a constant $c_4:=c_4(n)>0$ such that for all $k\ge1$,
\begin{equation}\label{p6-2-3a}
\begin{split}
&\Ee\left[\left| \int_0^{k^\alpha T}
k^{-(d-\alpha)}\sum_{x\in \Z^d}\sum_{y_1,y_2\in  \Z^d\setminus \{x\}}
\frac{\eta_k(k^{-\alpha}t,x,y_1)}{|x-y_1|^{d+\alpha}}\frac{\eta_k(k^{-\alpha}t,x,y_2)}{|x-y_2|^{d+\alpha}}\xi_{t,x,y_1}\xi_{t,x,y_2}\,dt\right|^{2n}\right]\\
&\le c_4\left(\max\{k^{-n(8-4\alpha)},k^{-2nd}\}+k^{-2nd}(\log k)^{2n}\I_{\{d=4-2\alpha\}}\right)T^{2n}.
\end{split}
\end{equation}
Therefore, for every $\alpha\in (1,2)$ and $\theta\in (0,1)$, we choose $n\ge1$ so that
$n(1-\theta)\min\{(8-4\alpha),2d\}>1$ and obtain
\begin{align*}
&\sum_{k=1}^\infty\Pp\Bigg(\left| \int_0^{k^\alpha T}
k^{-(d-\alpha)}\sum_{x\in \Z^d}\sum_{y_1,y_2\in  \Z^d\setminus \{x\}}
\frac{\eta_k(k^{-\alpha}t,x,y_1)}{|x-y_1|^{d+\alpha}}\frac{\eta_k(k^{-\alpha}t,x,y_2)}{|x-y_2|^{d+\alpha}}\xi_{t,x,y_1}\xi_{t,x,y_2}\,dt\right|\\
&\qquad\qquad >\max\{k^{-2\theta(2-\alpha)},k^{-\theta d}\}T\Bigg)\\
&\le \sum_{k=1}^\infty T^{-2n}\min\{k^{4n\theta(2-\alpha)},k^{2n\theta d}\}\\
&\quad\times \Ee\left[\left| \int_0^{k^\alpha T}
k^{-(d-\alpha)}\sum_{x\in \Z^d}\sum_{y_1,y_2\in  \Z^d\setminus \{x\}}
\frac{\eta_k(k^{-\alpha}t,x,y_1)}{|x-y_1|^{d+\alpha}}\frac{\eta_k(k^{-\alpha}t,x,y_2)}{|x-y_2|^{d+\alpha}}\xi_{t,x,y_1}\xi_{t,x,y_2}\,dt\right|^{2n}\right]\\
&\le c_4\sum_{k=1}^\infty \max\{k^{-4n(1-\theta)(2-\alpha)},k^{-2n(1-\theta)d}\}(\log k)^{2n}<\infty.
\end{align*}
This along with
  the
Borel-Cantelli lemma yields that
for a.e. $\w\in \Omega$,
there exists  $k^*(\w)\ge 1$ such that for all $k\ge k^*(\w)$,
\begin{align*}
&\left| \int_0^{k^\alpha T}
k^{-(d-\alpha)}\sum_{x\in \Z^d}\sum_{y_1,y_2\in  \Z^d\setminus \{x\}}
\frac{\eta_k(k^{-\alpha}t,x,y_1)}{|x-y_1|^{d+\alpha}}\frac{\eta_k(k^{-\alpha}t,x,y_2)}{|x-y_2|^{d+\alpha}}\xi_{t,x,y_1}\xi_{t,x,y_2}\,dt\right|\\
&\le \max\{
k^{-2\theta(2-\alpha)},k^{-\theta d}\}T.\end{align*}
This, along with \eqref{p6-2-2}, yields the desired conclusion \eqref{p6-2-1} for $\alpha\in (1,2)$.

\smallskip

{\bf Case 2:} Now we prove \eqref{p6-2-1} for the case $\alpha\in (0,1]$. Set
\begin{align*}
\widetilde \eta_k(t,x,y):=f(t,k^{-1}y)-f(t,k^{-1}x)-k^{-1}
\left\langle \nabla f(t,k^{-1}x), (y-x)\I_{\{|y-x|\le k\}}\right\rangle
\quad \hbox{and} \quad \xi_{t,x,y}:=w_{t,x,y}-\sK (t).
\end{align*}
Then, by the definition of $\widehat \LL^k_t$, we have
\begin{align*}
&\int_0^{T}\int_{k^{-1}\Z^d} |\widehat \LL_t^{k}f(t,x)-\bar \LL^{k}_tf(t,x) |^2\,\mu^{(k)}(dx)\,dt\\
&=k^{-(d-\alpha)}\int_0^{k^\alpha T}\sum_{x\in \Z^d}\sum_{y_1,y_2\in  \Z^d\setminus \{x\}}
\frac{\widetilde \eta_k(k^{-\alpha}t,x,y_1)}{|x-y_1|^{d+\alpha}}\frac{\widetilde \eta_k(k^{-\alpha}t,x,y_2)}{|x-y_2|^{d+\alpha}}\xi_{t,x,y_1}\xi_{t,x,y_2}\,dt.
\end{align*}
Note that
$$
|\widetilde \eta_k(t,x,y)|
\le c_5k^{-2}\left(1+|k^{-1}x|\right)^{-d-\beta}|x-y|^2\I_{\{|x-y|\le k\}}
+c_5\left(\left(1+|k^{-1}x|\right)^{-d-\beta}
+\left(1+|k^{-1}y|\right)^{-d-\beta}\right)\I_{\{|x-y|> k\}}.
$$
Using the estimate above and following the proof of {\bf Case 1}, we can verify the desired conclusion \eqref{p6-2-1} when $\alpha\in (0,1]$.
\end{proof}

\begin{proof}[Proof of Proposition $\ref{p6-3}$]
Let
\begin{align*}
\eta_k(t,x,y):=f(t,k^{-1}y)-f(t,k^{-1}x),\quad \xi_{t,x,y}:=w_{t,x,y}-
\sK (t).
\end{align*}
Due to the definition of $\LL^k_t$,  we have
\begin{align*}
& \int_0^T \int_{k^{-1}\Z^d} |\sL^{k}_t f(t,\cdot)(x)-\bar \sL^{k}_tf(t,\cdot)(x) |^2\,\mu^{(k)}(dx)\\
&=k^{-(d-\alpha)}\int_0^{k^\alpha T}\sum_{x\in \Z^d}\sum_{y_1,y_2\in  \Z^d\setminus \{x\}}
\frac{\eta_k(k^{-\alpha}t,x,y_1)}{|x-y_1|^{d+\alpha}}\frac{\eta_k(k^{-\alpha}t,x,y_2)}{|x-y_2|^{d+\alpha}}\xi_{t,x,y_1}\xi_{t,x,y_2}\,dt.
\end{align*}
On the other hand, it holds that
$$
|\eta_k(t,x,y)|\le c_1k^{-1}\left(1+|k^{-1}x|\right)^{-d-\beta}|x-y|\I_{\{|x-y|\le k\}}
+c_1\left(\left(1+|k^{-1}x|\right)^{-d-\beta}
+\left(1+|k^{-1}y|\right)^{-d-\beta}\right)\I_{\{|x-y|> k\}}.$$
Using this estimate and following the argument for {\bf Case 1} in the proof of Proposition \ref{p6-2}, we can prove the desired conclusion \eqref{p6-3-1}.
\end{proof}

\begin{remark}\label{r3-4}
\begin{itemize}
\item[{\rm(i)}]
Let us briefly mention the difference of the approaches between the time-independent setting and the time-dependent setting.  We take the proofs of \cite[Proposition 4.2]{CCKW3} and Proposition \ref{p6-2} for example.
(Similar remarks
hold for the proofs of  \cite[Proposition 4.3]{CCKW3} and Proposition
\ref{p6-3}.)
 In the time-dependent case, the estimate for the left hand side of \eqref{p6-2-1} is reduced to that for the right hand side of \eqref{p6-2-2}. Since in the
time-dependent  framework we are concerned
with
the space-time
 variable $(t,x,y)$ on $\R_+\times \Z^d\times \Z^d$ and the
 random coefficients $\{w_{t,x,y}\}$ are not required to be independent for all distinct pairs $(t,(x,y))
\in \R_+\times E$, we can not apply the Hoeffding inequality as in \cite[Proposition 4.2]{CCKW3}. Instead, we
estimate the moment of
 the right hand side of \eqref{p6-2-2} (this trick has
also been used in the proof of Proposition \ref{p6-1}).
  The price to pay is that the estimate \eqref{p6-2-1} is worse than that in  \cite[Proposition 4.2]{CCKW3},
   where the corresponding estimate \eqref{p6-2-1} holds with $\theta=1$.
 \item[{\rm(ii)}]
Similar to Remark \ref{r2-5}, we also use the uniform boundedness of $\{w_{t,x,y}: t\in \R_+, (x,y)\in E\}$ in the argument of
\eqref{p6-2-3a}. If we replace the boundedness condition of
 $\{w_{t,x,y}: t\in \R_+, (x,y)\in E\}$ by some
 moment condition, then the convergence rate
 \eqref{p6-3---} would be slower.\end{itemize}
\end{remark}

\section{Quantitative homogenizations} \label{section4}

In this section, we present the proof of Theorem \ref{T:1.1}.
Let  $\psi \in C_c^2(\R^d)$ be a cut-off function
taking values between $0$ and $1$ so that $\psi (x)=1$ for $|x|\le 1/2$ and $\psi =0$ for $ |x|\ge 1$.
For $R\geq 1$, define
\begin{equation}
\label{e:4.1}
\psi_R(x):= \psi(x/R) , \quad  x\in \R^d .
\end{equation}
Note that
\begin{align}
\label{e:4.2}
\nabla^i \psi_R(x)=0 \text{ for all } x\in \R^d\ {\rm with}\ |x|\le R/2\ {\rm or}\ |x|>R \quad \hbox{and} \quad
\|\nabla^i \psi_R\|_\infty\le c_0R^{-i} \ \hbox{ for  } i=1,2.
\end{align}

\smallskip

\begin{lemma}\label{l6-1}
Let
$f\in C_b^2(\R^d)$ satisfy that
\begin{equation}\label{l6-1-1a}
|\nabla^i f(x)|\le C_1
 (1+|x|)^{-d-\beta},
\quad x\in \R^d,\ i=0,1,2
\end{equation} with $C_1>0$ and
$\beta\in (0,\infty)$,
and, for every
$R\ge1$, let $\psi_R\in C_c^\infty(\R^d)$ be the cut-off function defined by
\eqref{e:4.1}.
It holds that
\begin{equation}\label{l6-1-1}
|\bar \LL f(x)-\bar \LL (f\psi_R)(x)|\le C_2\left(R^{-d-\beta}\I_{B_{2R}}(x)
+\left((1+|x|)^{-d-\beta}+R^{-\beta}(1+|x|)^{-d-\alpha}\right)\I_{B_{2R}^c}(x)\right),\ x\in \R^d,
\end{equation}
where $C_2>0$ is independent of $R$. In particular,
there is a constant $C_3>0$ such that
for all $R\ge1$,
\begin{equation}\label{l6-1-1b}\int_{\R^d}|\bar \LL f(x)-\bar \LL (f\psi_R)(x)|^2\,dx\le C_3R^{-d-2\beta}.\end{equation}  \end{lemma}
\begin{proof}
When $x\in B_{R/8}$,
\begin{align*}
|\bar \LL f(x)-\bar \LL (f\psi_R)(x)|&=\left|\int_{\{|z|>R/4\}}(f(x+z)(1-\psi_R(x+z)) \frac{\sK }{|z|^{d+\alpha}}
dz\right|\\
&\le c_1\int_{\{|z|>R/4\}}(1+|x+z|)^{-d-\beta}\frac{1}{|z|^{d+\alpha}}dz
\le c_2R^{-d-\alpha-\beta},
\end{align*}
where
we used the fact   $\psi_R(x+z)=1$ for $|z|\leq R/4$ and \eqref{l6-1-1a} in the first inequality, and
the fact   $|x+z|\ge R/8$ for every $|x|\le R/8$ and $|z|> R/4$ in the second inequality.

For $x\in B_{2R}\backslash B_{R/8}$,
\begin{align*}
|\bar \LL f(x)-\bar \LL (f\psi_R)(x)|&=\left|{\rm p.v.}\int_{\R^d}\Big(f(x+z)(1-\psi_R(x+z))-f(x)(1-\psi_R(x))\Big)
\frac{\sK }{|z|^{d+\alpha}}dz
\right|\\
&\le c_3\sup_{y\in \R^d:|y-x|\le 1}\left|\nabla^2 \left(f(1-\psi_R)\right)(y)\right|\int_{\{|z|\le 1\}}\frac{|z|^2}{|z|^{d+\alpha}}dz+
c_3|f(x)|\cdot \int_{\{|z|>1\}}\frac{1}{|z|^{d+\alpha}}dz\\
&\quad +c_3\int_{\{|z|>1,|x+z|>R/2\}}\frac{|f(x+z)|}{|z|^{d+\alpha}}dz\\
& \le c_4R^{-d-\beta},
\end{align*}
where
the
first inequality  is due to the fact that $\psi_R(x+z)=1$ if  $|x+z|\leq R/2$,
and the
second inequality follows from
  \eqref{l6-1-1a}.

When $x\notin B_{2R}$,
$\psi_R(x)=0$, and so
\begin{align*}
&|\bar \LL f(x)-\bar \LL (f\psi_R)(x)|\\
&=\left|{\rm p.v.}\int_{\R^d}\Big(f(x+z)(1-\psi_R(x+z))-f(x)\Big)\frac{\sK }{|z|^{d+\alpha}}dz\right|\\
&\le \left|\int_{\{|z|\le 1\}}\left(f(x+z)-f(x)-\langle \nabla f(x),z\rangle\right)\frac{\sK }{|z|^{d+\alpha}}dz\right|+
\left|\int_{\{|z|>1:R/2\le |x+z|\le |x|/2\}}|f(x+z)|\frac{\sK }{|z|^{d+\alpha}}dz\right|\\
&\quad+ \left|\int_{\{|z|>1: |x+z|> |x|/2\}}|f(x+z)|\frac{\sK }{|z|^{d+\alpha}}dz\right|+\left|\int_{\{|z|>1\}}|f(x)|\frac{\sK }{|z|^{d+\alpha}}dz\right|\\
&\le c_5\sup_{y\in \R^d:|y-x|\le 1}|\nabla^2 f(y)|\cdot \int_{\{|z|\le 1\}}\frac{|z|^2}{|z|^{d+\alpha}}dz+c_5(1+|x|)^{-d-\beta}
+c_5\int_{\{|z|\ge |x|/2,|x+z|\ge R/2\}}\frac{1}{|z|^{d+\alpha}(1+|x+z|)^{d+\beta}}dz\\
&\le c_6((1+|x|)^{-d-\beta}+R^{-\beta}(1+|x|)^{-d-\alpha}).
\end{align*}
Here in the second inequality we used \eqref{l6-1-1a} and the fact $|z|\ge |x|/2$ for every $x,z\in \R^d$ with
$|x+z|\le |x|/2$.

Combining
all the estimates above yields the desired conclusion \eqref{l6-1-1}. Furthermore, \eqref{l6-1-1b} is a direct consequence of \eqref{l6-1-1}.
\end{proof}

\medskip

\begin{proof}[Proof of Theorem $\ref{T:1.1}$]  Throughout the proof, we fix $\beta\in (0,\infty]$, $g\in C_c^2 (\R^d)$, $T>0$ and
$f\in \mathscr{S}_{\beta,g,T}$.
Recall that $u_k^\w\in C^1([0,T];L^2(k^{-1}\Z^d;\mu^{ (k)}))$ and
$\bar u\in C^1([0,T]; C_b^2(\R^d))$
are the solutions to parabolic equations
\eqref{e6-4} and \eqref{e6-3}, respectively.
By the definition of $\mathscr{S}_{\beta,g,T}$, there
 are constants $C_1,R_0>0$ so that
\begin{align}\label{t6-1-1a}
|\nabla^i \bar u(t,x)|+\left|\frac{\partial \bar u(t,x)}{\partial t}\right|\le
C_1\left((1+|x|)^{-d-\beta}+\I_{\{|x|\le R_0\}}(x)\right),
\quad t\in [0,T],\ x\in \R^d,\
i=0,1,2,3.
\end{align} We note that here $C_1, R_0$ may depend on $T$. We first consider that $\beta\in (0,\infty)$.

{\bf Case 1:} We consider the case that $\alpha\in (1,2)$. Let $\phi_m:\R_+\times
( B_{2^m} \cap \Z^d )
\to \R^d$ be the solution to
\eqref{e6-5}.
 We extend its definition to   $ \R_+\times \R^d \to \R^d$
by setting $\phi_m(t,x)=\phi_m(t,z)$ for
every $x\in (\prod_{i=1}^d (z_i,z_i+1])\cap B_{2^m}$ with the unique
$z=(z_1,\cdots, z_d)
\in B_{2^m}\cap \Z^d$ and setting
$\phi_m(t,x)=0$ for every $x\notin B_{2^m}$.
For every $R\ge 1$ (which may not be an integer), denote
$\phi_{[R]}(t,x)$ by $\phi_{R}(t,x)$ for simplicity of the notation.

For any $k\ge1$, let $m\ge0$ be the unique integer so that $2^m\le k<2^{m+1}$.
According to \eqref{p6-1-1}, we know that
for a.e. $\w\in \Omega$,
there exists a constant $m_1^*:=m_1^*(\w)\ge1$ such that for all
$m\ge m_1^*$ and $\theta>0$
\begin{equation}\label{t6-1-1}
\begin{split}
&\sup_{t\in \left[0,[2^{m(1+\theta)\alpha}]T\right]}\int_{B_{2^{m(1+\theta)}}}|\phi_{m(1+\theta)}(t,x)|^2\mu(dx)+\int_0^{[2^{m(1+\theta)\alpha}]T} \E_{t,B_{2^{m(1+\theta)}}}^\w
\left(\phi_{m(1+\theta)}(t,\cdot),\phi_{m(1+\theta)}(t,\cdot)\right) dt\\
&\le c_1 m^{\frac{(1+\gamma)\alpha}{(2(d-\alpha))\wedge d}}2^{m(1+\theta)(d+\alpha)}T.
\end{split}
\end{equation}

Recall the definition of $\psi_R$ for $R\ge1$ from  \eqref{e:4.1}.
Given $\theta\in (0,1)$, define
\begin{align}\label{eq:v_k}
v_k(t,x)=\bar u(t,x)\psi_{k^\theta}(x)+k^{-1}\left\langle \nabla \left(\bar u(t,\cdot)\psi_{k^\theta}(\cdot)\right)(x), \phi_{(m+2)(1+\theta)}(k^\alpha t, kx)\right\rangle,\quad t\in [0,T],\ x\in \R^d.
\end{align}
By \eqref{t6-1-1a}, \eqref{t6-1-1} and the fact ${\rm supp}[\psi_{k^\theta}(\cdot)]\subset B_{2^{(m+1)\theta}}$, we obtain that for all $2^{m+1}>k\ge 2^m\ge 2^{m_1^*}$,
\begin{equation}\label{t6-1-2}
\begin{split}
&\int_0^T \int_{\R^d}\left|v_k(t,x)-\bar u(t,x)\right|^2 dx dt\\
&\le
\int_0^T\int_{\R^d} \bar u(t,x)^2(1-\psi_{k^\theta}(x))^2\,dx\,dt+c_2k^{-2}\int_0^{T}\int_{B_{2^{(m+1)\theta}}}\left|\phi_{(m+2)(1+\theta)}(k^\alpha t,kx)\right|^2dxdt\\
&= \int_0^T\int_{\R^d} \bar u(t,x)^2(1-\psi_{k^\theta}(x))^2\,dx\,dt+c_2k^{-(d+2+\alpha)}\int_0^{k^\alpha T}\int_{B_{k2^{(m+1)\theta}}}
\left|\phi_{(m+2)(1+\theta)}(t,x)\right|^2\mu(dx)dt\\
&\le \int_0^T\int_{\R^d} \bar u(t,x)^2(1-\psi_{k^\theta}(x))^2\,dx\,dt+c_2k^{-(d+2+\alpha)}\int_0^{k^\alpha T}\int_{B_{2^{(m+1)(1+\theta)}}}
|\phi_{(m+2)(1+\theta)}(t,x)|^2\mu(dx)dt\\
&\le c_3\left(k^{-\theta(d+2\beta)}+ k^{-(2-\alpha)+(d+\alpha)\theta}(\log^{\frac{(1+\gamma)\alpha}{(2(d-\alpha))\wedge d}} k)\right) T.
\end{split}
\end{equation}

On the other hand, combining \eqref{p6-2-1} with \eqref{p6-3---}
(also by noting that \eqref{t6-1-1a} holds)
immediately yields that for
a.e. $\w\in \Omega$ and
every $\theta_0\in (0,1)$, there exists
$k_1(\w)\ge1$ such that for all $k\ge k_1(\w)$,
\begin{equation}\label{t6-1-3}
\int_0^T \int_{k^{-1}\Z^d}\left|\widehat \LL_t^k \left(\bar u(t,\cdot)\psi_{k^\theta}(\cdot)\right)(x)-\bar \LL \left(\bar u(t,\cdot)\psi_{k^\theta}(\cdot)\right)(x)\right|^2\mu^{  (k)}(dx)dt
\le c_4\left(\pi(k^\alpha T)+\max\{k^{-2\theta_0(2-\alpha)},k^{-\theta_0 d}\}\right)T.
\end{equation}

We restrict $v_k:[0,T]\times \R^d\to \R$ to $v_k:[0,T]\times k^{-1}\Z^d\to \R$ and derive that for all $t\in [0,T]$ and $x\in k^{-1}\Z^d$,
\begin{align}
&\sL^{k}_t v_k(t,\cdot)(x)  \nonumber \\
&=\sL^{k}_t\left(\bar u(t,\cdot)\psi_{k^\theta}(\cdot)\right)(x) \nonumber \\
&\quad +k^{-1}\left\langle
\int_{k^{-1}\Z^d} \Big(\nabla \left(\bar u(t,\cdot)\psi_{k^\theta}(\cdot)\right)(x+z)-\nabla \left(\bar u(t,\cdot)\psi_{k^\theta}(\cdot)\right)(x)\Big)\frac{w_{k^\alpha t,kx,k(x+z)}}{|z|^{d+\alpha}}
\,\mu^{(k)}(dz) ,    \phi_{(m+2)(1+\theta)}(k^\alpha t,kx)   \right\rangle \nonumber \\
&\quad+k^{-1}\left\langle \nabla \left(\bar u(t,\cdot)\psi_{k^\theta}(\cdot)\right)(x), \int_{k^{-1}\Z^d}\left(\phi_{(m+2)(1+\theta)}\left(k^\alpha t, k(x+z)\right)-\phi_{(m+2)(1+\theta)}\left(k^\alpha t, k x\right)\right)
\frac{w_{k^\alpha t,kx,k(x+z)}}{|z|^{d+\alpha}}\,\mu^{(k)}(dz)\right\rangle \nonumber \\
&\quad +k^{-1}\int_{k^{-1}\Z^d}\Big\langle\nabla \left(\bar u(t,\cdot)\psi_{k^\theta}(\cdot)\right)(x+z)-\nabla \left(\bar u(t,\cdot)\psi_{k^\theta}(\cdot)\right)(x) , \nonumber \\
&\qquad\qquad\qquad \qquad \phi_{(m+2)(1+\theta)}(k^\alpha t, k(x+z))-\phi_{(m+2)(1+\theta)}(k^\alpha t, kx)\Big\rangle \frac{w_{k^\alpha t, kx,k(x+z)}}{|z|^{d+\alpha}}\,\mu^{(k)}(dz) \nonumber \\
&=:\sum_{i=1}^4 I_i^{(k)}(t,x).   \label{e:4.11}
\end{align}

According to
\eqref{e6-6}
and \eqref{e6-5a},
\begin{align*}
I_1^{(k)}(t,x)=&\widehat \LL_t^k \left(\bar u(t,\cdot)\psi_{k^\theta}(\cdot)\right)(x)+\left\langle \nabla \left(\bar u(t,\cdot)\psi_{k^\theta}(\cdot)\right)(x), \int_{k^{-1}\Z^d} z \frac{w_{k^\alpha t,kx,k(x+z)}}{|z|^{d+\alpha}}\,\mu^{(k)}(dz)\right\rangle\\
=&\widehat \LL_t^k \left(\bar u(t,\cdot)\psi_{k^\theta}(\cdot)\right)(x)+k^{\alpha-1}\left\langle \nabla \left(\bar u(t,\cdot)\psi_{k^\theta}(\cdot)\right)(x), \int_{\Z^d} z \frac{w_{k^\alpha t,kx,kx+z}}{|z|^{d+\alpha}}\,\mu(dz)\right\rangle\\
=&\widehat \LL_t^k \left(\bar u(t,\cdot)\psi_{k^\theta}(\cdot)\right)(x)+k^{\alpha-1}\left\langle \nabla \left(\bar u(t,\cdot)\psi_{k^\theta}(\cdot)\right)(x), V(k^\alpha t,kx)\right\rangle,
\end{align*}
where $\widehat\LL_t^k$ and $V(t,x)$ are defined by \eqref{e6-6} and \eqref{e6-5a}, respectively.
Thus,
\begin{equation}\label{e:addo1}
I_1^{(k)}(t,x)
=\bar \LL \left(\bar u(t,\cdot)\psi_{k^\theta}(\cdot)\right)(x)+k^{\alpha-1}\left\langle \nabla \left(\bar u(t,\cdot)\psi_{k^\theta}(\cdot)\right)(x), V(k^\alpha t,kx) \right\rangle+K_1^{ (k)}(t,x),
\end{equation}
where  $$K_1^{ (k)}(t,x) := (\widehat \LL_t^k -\bar \sL)\left(\bar u(t,\cdot)\psi_{k^\theta}(\cdot)\right)(x).$$
By \eqref{t6-1-3},
for each
a.e. $\w\in \Omega$ and $\theta_0\in (0,1)$,
we can find $m_2^*:=m_2^*(\w)>0$ such that for all $k\ge 2^{m_2^*}$,
\begin{equation}  \label{e:4.13}
\int_0^T \int_{k^{-1}\Z^d}|K_1^{  (k)}(t,x)|^2\mu^{ (k)}(dx)dt
\le c_5\left(\pi(k^\alpha T)+\max\{k^{-2\theta_0(2-\alpha)},k^{-\theta_0 d}\}\right)T.
\end{equation}

For $I_2^{ (k)}(t,x)$, we write
\begin{align*}
&I_2^{ (k)}(t,x)\\
&=k^{-1}\Bigg\langle
\phi_{(m+2)(1+\theta)}(k^\alpha t, k x),
 \\
 &\quad\quad \int_{k^{-1}\Z^d} \Bigg(\nabla \left(\bar u(t,\cdot)\psi_{k^\theta}(\cdot)\right)(x+z)-\nabla \left(\bar u(t,\cdot)\psi_{k^\theta}(\cdot)\right)(x)
 -\langle \nabla^2 \left(\bar u(t,\cdot)\psi_{k^\theta}(\cdot)\right)(x),z\I_{\{|z|\le 1\}}\rangle\Bigg)\frac{w_{k^\alpha t,kx,k(x+z)}}{|z|^{d+\alpha}}\,\mu^{(k)}(dz)\\
&\quad\quad  +\int_{k^{-1}\Z^d}\langle
 \nabla^2 \left(\bar u(t,\cdot)\psi_{k^\theta}(\cdot)\right)(x),
z\I_{\{|z|\le 1\}}\rangle\frac{w_{k^\alpha t, kx,k(x+z)}}{|z|^{d+\alpha}}\,
\mu^{(k)}(dz)\Bigg\rangle\\
&=:I_{21}^{ (k)}(t,x)+I_{22}^{ (k)}(t,x).
\end{align*}
Using \eqref{t6-1-1} and \eqref{t6-1-1a} as well as the definition of $\psi_{k^\theta}$,
for a.e. $\w\in \Omega$, we can find
$m_3^*:=m_3^*(\w)\ge1$ such that for all $k\ge 2^{m_3^*}$,
\begin{align*}
\int_0^T\int_{k^{-1}\Z^d}|I_{21}^{(k)}(t,x)|^2\mu^{(k)}(dx)dt
&\le c_6k^{-2-d-\alpha}\int_0^{2^{(m+2)\alpha}T}
\int_{B_{2^{(m+2)(1+\theta)}}}
\left|\phi_{(m+2)(1+\theta)}(t,x)\right|^2\mu(dx)dt\\
&\le c_7(k^{-(2-\alpha)+(d+\alpha)\theta}\log^{\frac{(1+\gamma)\alpha}{(2(d-\alpha))\wedge d}}k) T.
\end{align*}

For $I^{(k)}_{22}(t, x)$, we consider its dual.
For any $g\in C_b\left([0,T];L^2(k^{-1}\Z^d;\mu^{(k)})\right)$ and $t\in [0,T]$,
\begin{align*}
&\int_{k^{-1}\Z^d}I_{22}^{(k)}(t,x)g(t,x)\,\mu^{(k)}(dx)\\
&=k^{-1}\int_{k^{-1}\Z^d}\int_{\{|z|\le 1\}}
\langle G(t,x),\phi_{(m+2)(1+\theta)}(k^\alpha t, kx)\otimes z\rangle
\frac{w_{k^\alpha t, k x, k(x+z)}}{|z|^{d+\alpha}}\,\mu^{(k)}(dz)\,\mu^{(k)}(dx)\\
&=k^{-1}\Bigg(\frac{1}{2}\int_{k^{-1}\Z^d}\int_{\{|z|\le 1\}}
\langle G(t,x),\phi_{(m+2)(1+\theta)}(k^\alpha t, kx)\otimes z\rangle
\frac{w_{k^\alpha t, k x, k(x+z)}}{|z|^{d+\alpha}}\mu^{(k)}(dz)\,\mu^{(k)}(dx)\\
&\qquad\qquad -\frac{1}{2}\int_{k^{-1}\Z^d}\int_{\{|z|\le 1\}}
\langle G(t, x+z),\phi_{(m+2)(1+\theta)}(k^\alpha t, k(x+z))\otimes z\rangle
\frac{w_{k^\alpha t, k x, k(x+z)}}{|z|^{d+\alpha}}\,\mu^{(k)}(dz)\,\mu^{(k)}(dx)\Bigg)\\
&=(2k)^{-1}\int_{k^{-1}\Z^d}\int_{\{|z|\le 1\}}
\left\langle G(t, x),\phi_{(m+2)(1+\theta)}(k^\alpha t, kx)\otimes z\right\rangle- \left\langle G(t, x+z),\phi_{(m+2)(1+\theta)}(k^\alpha t, k(x+z))\otimes z\right\rangle\\
&\qquad\qquad\times
\frac{w_{k^\alpha t, k x, k(x+z)}}{|z|^{d+\alpha}}\,\mu^{(k)}(dz)\,\mu^{(k)}(dx).
\end{align*}
Here, $G(t,x):=g(t,x)\nabla^2\left(\bar u(t,\cdot)\psi_{k^\theta}(\cdot)\right)(x)$, and in the second equality we
 used the change of variables $x=\tilde x+\tilde z$ and $z=-\tilde z$.
Hence,
\begin{align*}
&\left|\int_{k^{-1}\Z^d}I_{22}^{(k)}(t,x)g(t,x)\,\mu^{(k)}(dx)\right|\\
&\le (2k)^{-1}\Bigg(\int_{B_{k^\theta+1}}\int_{\{|z|\le 1\}}|\phi_{(m+2)(1+\theta)}(k^\alpha t, k x)|
|G(t,x+z)-G(t,x)||z|\frac{w_{k^\alpha t, k x,k(x+z)}}{|z|^{d+\alpha}}\,\mu^{(k)}(dz)\,\mu^{(k)}(dx)\\
&\qquad  +\int_{B_{k^\theta}}\int_{\{|z|\le 1\}}|G(t,x)|
|\phi_{(m+2)(1+\theta)}(k^\alpha t, k (x+z))-\phi_{(m+2)(1+\theta)}(k^\alpha t, k x)||z|\frac{w_{k^\alpha t, k x,k(x+z)}}{|z|^{d+\alpha}}\,\mu^{(k)}(dz)\,\mu^{(k)}(dx)\Bigg)\\
&=:I_{221}^{(k)}(t)+I_{222}^{(k)}(t),
\end{align*}
where in the inequality we used again the change of variables $x=\tilde x+\tilde z$ and $z=-\tilde z$,
and also  the fact that $\text{supp} [G] \subset B_{k^\theta}$.
On the other hand,
by
the Cauchy-Schwartz inequality, we have for all $k\ge 2^{m_1^*}$,
\begin{align*}
I_{221}^{(k)}(t) &\le c_8k^{-1}
\left(\int_{B_{k^\theta+1}}|\phi_{(m+2)(1+\theta)}(k^\alpha t, k x)|^2\left(\int_{\{|z|\le 1\}}\frac{|z|^2}{|z|^{d+\alpha}}\,\mu^{(k)}(dz)\right)
\,\mu^{(k)}(dx)\right)^{1/2}\\
&\qquad\qquad\times\left(\int_{k^{-1}\Z^d}\int_{\{|z|\le 1\}}\frac{|G(t,x+z)-G(t,x)|^2w_{k^\alpha t, k x,k(x+z)}}{|z|^{d+\alpha}}
\,\mu^{(k)}(dz)\,\mu^{(k)}(dx)\right)^{1/2}\\
&\le c_{9}k^{-1}\left(\int_{B_{k^\theta+1}}|\phi_{(m+2)(1+\theta)}(k^\alpha t, k x)|^2\,\mu^{(k)}(dx)\right)^{1/2}  \mathscr{E}^{(k),\w}(G(t,\cdot),G(t,\cdot)) ^{1/2}\\
&\le c_{10}k^{-1}\left(\int_{B_{2k^\theta}}|\phi_{(m+2)(1+\theta)}(k^\alpha t, k x)|^2\,\mu^{(k)}(dx)\right)^{1/2}
\left( \mathscr{E}^{(k),\w}(g(t,\cdot),g(t,\cdot))^{1/2}+\|g(t,\cdot)\|_{L^2(k^{-1}\Z^d;\mu^{(k)})}
\right),
\end{align*}
where
\begin{align*}
   \mathscr{E}^{(k),\w}(G(t,\cdot),G(t,\cdot))
 &
 =\frac{1}{2}
 \sum_{i,j=1}^d
 \int_{k^{-1}\Z^d}\int_{k^{-1}\Z^d}
\frac{(G^{(ij)}(t,x+z)-G^{(ij)}(t, x))^2w_{k^\alpha t, k x,k(x+z)}(\w)}{|z|^{d+\alpha}}\,\mu^{(k)}(dz)\,\mu^{(k)}(dx)
\end{align*} and
$$\mathscr{E}^{(k),\w}(g(t,\cdot),g(t,\cdot)) =\frac{1}{2}\int_{k^{-1}\Z^d}\int_{k^{-1}\Z^d}\frac{\left|g(t,x+z)-g(t,x)\right|^2w_{k^\alpha t,kx,k(x+z)}}{|z|^{d+\alpha}}
\mu^{ (k)}(dz)\mu^{ (k)}(dx)$$ with $$G^{(ij)}(t,x)=g(t,x)\frac{\partial^2 \left(\bar u(t,\cdot) \psi_{k^\theta}(\cdot)\right)(x)}
{\partial x_i \partial x_j},\quad 1\le i,j\le d,$$
and in the last inequality we used \eqref{t6-1-1a}.

By the Cauchy-Schwartz inequality and \eqref{t6-1-1a} again,
\begin{align*}
I_{222}^{(k)}(t)
&\le c_{11}k^{-1}\left(\int_{B_{k^\theta}}G^2(t,x)\left(\int_{\{|z|\le 1\}}\frac{|z|^2}{|z|^{d+\alpha}}\,\mu^{(k)}(dz)\right)\,
\mu^{(k)}(dx)\right)^{1/2}\\
&\quad\times\left(\int_{B_{k^\theta}}\int_{\{|z|\le 1\}}
\left|\phi_{(m+2)(1+\theta)}(k^\alpha t, k(x+z))-\phi_{(m+2)(1+\theta)}(k^\alpha t, k x)\right|^2
\frac{w_{k^\alpha t, k x,k(x+z)}}{|z|^{d+\alpha}}\,\mu^{(k)}(dz)\,\mu^{(k)}(dx)\right)^{1/2}\\
&\le c_{12}k^{-1} \left(\int_{B_{k^\theta}}g^2(t,x)\,
\mu^{(k)}(dx)\right)^{1/2}\\
&\quad\times \left(k^{-(d-\alpha)}
\int_{B_{k^{1+\theta}}}
\int_{B_{2k^{1+\theta}}}
\left|\phi_{(m+2)(1+\theta)}(k^\alpha t, x)-\phi_{(m+2)(1+\theta)}(k^\alpha t, y)\right|^2
\frac{w_{k^\alpha t, x,y}}{|x-y|^{d+\alpha}}
\,\mu (dy)\,\mu(dx)
\right)^{1/2}.
\end{align*}

Putting all the estimates together yields that for every $g\in C_b\left([0,T];L^2(k^{-1}\Z^d;\mu^{ (k)})\right)$ and $t\in [0,T]$,
\begin{align*}
& \Big|\int_{k^{-1}\Z^d}I_{22}^{ (k)}(t,x)g(t,x)\mu^{ (k)}(dx)\Big|\\
&\le c_{13}k^{-1}\left(
\int_{B_{2k^\theta}}
\left|\phi_{(m+2)(1+\theta)}(k^\alpha t,kx)\right|^2\mu^{  (k)}(dx)\right)^{1/2}\left(\E_t^{(k)}\left(g(t,\cdot),g(t,\cdot)\right)^{1/2}+\|g(t,\cdot)\|_{L^2(k^{-1}\Z^d;\mu^{ (k)})}
\right)\\
&\quad +c_{13}k^{-\frac{d+2-\alpha}{2}}\E_{k^\alpha t,B_{2^{(m+2)(1+\theta)}}}\left(\phi_{(m+2)(1+\theta)}(k^\alpha t,\cdot),\phi_{(m+2)(1+\theta)}(k^\alpha t,\cdot)\right)^{1/2}\|g(t,\cdot)\|_{L^2(k^{-1}\Z^d;\mu^{(k)})}.
\end{align*}
This, along with \eqref{t6-1-1} and the Young inequality, gives us
that for every
$k\ge 2^{m_3^*}$,
\begin{align*}
&\int_0^T \int_{k^{-1}\Z^d}I_{22}^{  (k)}(t,x)g(t,x)\mu^{(k)}(dx)d t\\
&\le {  \frac{1}{4} } \int_0^T \E_t^{(k)}\left(g(t,\cdot),g(t,\cdot)\right)dt
+ {  \frac{1}{4} }  \int_0^T \|g(t,\cdot)\|_{L^2(k^{-1}\Z^d;\mu^{  (k)})}^2dt\\
&\quad +c_{14}k^{-(d+2+\alpha)}\int_0^{2^{(m+2)\alpha}T}\int_{B_{2^{(m+2)(1+\theta)}}}\left|\phi_{(m+2)(1+\theta)}(t,x)\right|^2
\mu(dx)
dt\\
&\quad
+c_{14}k^{-(d+2)}\int_{0}^{2^{(m+2)\alpha}T}
\E_{t,B_{2^{(m+2)(1+\theta)}}}\left(\phi_{(m+2)(1+\theta)}(t,\cdot),\phi_{(m+2)(1+\theta)}(t,\cdot)\right)
dt\\
&\le {  \frac{1}{4} }  \int_0^T \E_t^{(k)}\left(g(t,\cdot),g(t,\cdot)\right)dt
+{ \frac{1}{4} }  \int_0^T \|g(t,\cdot)\|_{L^2(k^{-1}\Z^d;\mu^{(k)})}^2dt+c_{15}(k^{-(2-\alpha)+(d+\alpha)\theta}
\log^{\frac{(1+\gamma)\alpha}{(2(d-\alpha))\wedge d}} k )T.
\end{align*}

Furthermore, by the change of variable, we get
\begin{align*}
I_3^{ (k)}(t,x)
&=k^{\alpha-1}\left\langle \nabla \left(\bar u(t,\cdot)\psi_{k^\theta}(\cdot)\right)(x),
\int_{\Z^d}\left(\phi_{(m+2)(1+\theta)}(k^\alpha t,y)-\phi_{(m+2)(1+\theta)}(k^\alpha t, kx)\right)\frac{w_{k^\alpha t,kx,y}}{|y-kx|^{d+\alpha}}\mu(dy)\right\rangle\\
&=k^{\alpha-1}\Bigg\langle \nabla \left(\bar u(t,\cdot)\psi_{k^\theta}(\cdot)\right)(x), \LL_{t,B_{2^{(m+2)(1+\theta)}}}\phi_{(m+2)(1+\theta)}(k^\alpha t,\cdot)(x)\\
&\quad +
\int_{B_{2^{(m+2)(1+\theta)}}^c}\left(\phi_{(m+2)(1+\theta)}(k^\alpha t,y)-\phi_{(m+2)(1+\theta)}(k^\alpha t, kx)\right)\frac{w_{k^\alpha t,kx,y}}{|y-kx|^{d+\alpha}}\mu(dy)\Bigg\rangle\\
&=k^{\alpha-1}\Bigg\langle \nabla \left(\bar u(t,\cdot)\psi_{k^\theta}(\cdot)\right)(x), \frac{\partial}{\partial t}\phi_{(m+2)(1+\theta)}(\cdot,kx)(k^\alpha t)-V(k^\alpha t,kx) +\oint_{B_{2^{(m+2)(1+\theta)}}}V(k^\alpha t,\cdot)\,d\mu\\
&\quad+ \int_{B_{2^{(m+2)(1+\theta)}}^c}\left(\phi_{(m+2)(1+\theta)}(k^\alpha t,y)-\phi_{(m+2)(1+\theta)}(k^\alpha t, kx)\right)\frac{w_{k^\alpha t,kx,y}}{|y-kx|^{d+\alpha}}\mu(dy)\Bigg\rangle\\
&=:k^{\alpha-1}\left\langle \nabla \left(\bar u(t,\cdot)\psi_{k^\theta}(\cdot)\right)(x), \frac{\partial}{\partial t}\phi_{(m+2)(1+\theta)}(\cdot,kx)(k^\alpha t)-V(k^\alpha t,kx)\right\rangle
+I_{31}^{(k)}(t,x)+I_{32}^{(k)}(t,x).
\end{align*}
Here the third inequality follows from the corrector equation \eqref{e6-5}.
According to \eqref{p6-1-5}  and the fact that $2^m\le k<2^{m+1}$, we know that, for a.e. $\w\in \Omega$ and every
$\theta_0\in (\alpha/d,1)$,
there is a random variable $m_4^*:=m_4^*(\w)>1$ such that
for all $k\ge 2^{m_4^*}$,
\begin{align*}
\int_0^T\int_{k^{-1}\Z^d}|I_{31}^{(k)}(t,x)|^2\,\mu^{(k)}(dx)\,dt &\le k^{2(\alpha-1)-\alpha}\|\nabla (\bar u(t,\cdot) \psi_{k^\theta}(\cdot))\|_\infty^2 \int_0^{k^\alpha T}
\left| \oint_{B_{2^{(m+2)(1+\theta)}}}V( t,\cdot)\,d\mu
\right|^2\,dt\\
 &\le c_{16}k^{2\alpha-2
 -\theta_0(1+\theta) d} T.
\end{align*}
Note that $\phi_{(m+2)(1+\theta)}(y)=0$ for all $y\in B_{2^{(m+2)(1+\theta)}}^c$ and
$\text{supp} [ \psi_{k^\theta} ] \subset B_{k^\theta}$.
Hence,
for all $k\ge 2^{m_1^*}$,
\begin{align*}
&\int_0^T\int_{k^{-1}\Z^d}|I_{32}^{(k)}(t,x)|^2\,\mu^{(k)}(dx)\,d t\\
&\le
k^{2(\alpha-1)}\|\nabla (\bar u(t,\cdot)\psi_{k^\theta}(\cdot))\|_\infty^2\int_0^T
\int_{B_{k^\theta}}|\phi_{(m+2)(1+\theta)}(k^\alpha t, k x)|^2 \bigg(\sum_{y\in k^{-1}\Z^d:|y-k x|\ge k^{1+\theta}}\frac{w_{k^\alpha t, kx,y}}{|y-k x|^{d+\alpha}}\bigg)^2\,
\mu^{(k)}(dx)\,dt\\
&\le c_{17}k^{-2(1+\theta\alpha)}\int_0^T\int_{B_{k^\theta}}|\phi_{(m+2)(1+\theta)}(k^\alpha t, k x)|^2\,\mu^{(k)}(dx)\,dt\le c_{18}
k^{-(2-\alpha)+(d+\alpha)\theta}(\log^{\frac{(1+\gamma)\alpha}{(2(d-\alpha))\wedge d}}k) T.
\end{align*}
where the last inequality follows from  \eqref{t6-1-1}.

Summarizing all above estimates for $I_3^{(k)}(t,x)$, we
conclude that
\begin{equation}\label{t3-1-6}
\begin{split}
I_3^{(k)}(t,x)&= k^{\alpha-1}\left\langle \nabla \left(\bar u(t,\cdot)\psi_{k^\theta}(\cdot)\right)(x), \frac{\partial}{\partial t}\phi_{(m+2)(1+\theta)}(\cdot,kx)(k^\alpha t)-V(k^\alpha t,kx)\right\rangle
+K_3^{(k)}(t, x),
\end{split}
\end{equation}
where for all $k\ge 2^{\max\{m_1^*, m_3^*\}}$,
\begin{equation} \label{e:4.15}
\int_0^T\int_{k^{-1}\Z^d}|K_3^{(k)}(t,x)|^2\,\mu^{(k)}(dx)\,d t\le c_{19}\max\{k^{-(2-\alpha)+(d+\alpha)\theta}
,k^{2\alpha-2+\theta\alpha-\theta_0(1+\theta) d}\}(\log^{\frac{\alpha(1+\gamma)}{(2(d-\alpha))\wedge d}}k )T.
\end{equation}
In particular, according to \eqref{e:addo1} and \eqref{t3-1-6}, one can see that the term with the coefficient
$k^{\alpha-1} V(k^\alpha, kx)$ in $I_1^{ (k)}(t,x)$ vanishes (by the associated terms in $I_3^{(k)}(t,x)$),
which partly explains why we define the function $v_k(t,x)$
 as in  \eqref{eq:v_k}.

For the estimate for $I_4^{(k)}(t,x)$, we consider  the cases $x\in B_{{3k^\theta}/{2}}$ and $x\notin B_{{3k^\theta}/{2}}$ separately.
When $x\in B_{{3k^\theta}/{2}}^c$, we get from the fact $\nabla \left(\bar u(t, \cdot)\psi_{k^\theta}(\cdot)\right)(y)\neq 0$ only if $y\in B_{k^\theta}$ that
for every $k\ge 2^{m_1^*}$,
\begin{align*}
& |I_4^{(k)}(t,x)|\\
&=k^{-1}\left|\int_{B_{k^\theta}}\frac{
\left\langle \nabla \left(\bar u(t,\cdot)\psi_{k^\theta}(\cdot)\right)(y),
\phi_{(m+2)(1+\theta)}\left(k^\alpha t, k y\right)-\phi_{(m+2)(1+\theta)}\left(k^\alpha t, k x\right)\right\rangle
w_{k^\alpha t, k x,k y}}{|y-x|^{d+\alpha}}\,\mu^{(k)}(dy)\right|\\
&\le c_{20}\|\nabla (\bar u(t,\cdot)\psi_{k^\theta}(\cdot))\|_\infty k^{-1}(1+|x|)^{-d-\alpha}\cdot
\left(\int_{ B_{k^\theta}}\left|\phi_{(m+2)(1+\theta)}\left(k^\alpha t, k y\right)\right|\mu^{(k)}(dy)+
k^{\theta d}|\phi_{(m+2)(1+\theta)}(k^\alpha t, k x)|\right)\\
&\le c_{21}k^{-1}(1+|x|)^{-d-\alpha}
\left(k^{\theta d/2}\left(\int_{  B_{k^\theta}}\left|\phi_{(m+2)(1+\theta)}\left(k^\alpha t, k y\right)\right|^2\mu^{(k)}(dy)\right)^{1/2}+k^{\theta d}|\phi_{(m+2)(1+\theta)}(k^\alpha t, k x)|\right),
\end{align*}
where the first inequality
follows from the fact that $|y-x|\ge c_{22}(1+|x|)$ for all $y\in B_{k^\theta}$ and $x\in B_{{3k^\theta}/{2}}^c$.

When $x\in B_{{3k^\theta}/{2}}$, it holds that
\begin{align*}
&|I_4^{(k)}(t,x)|\\
&=k^{-1}\Bigg|
\int_{k^{-1}B_{{2^{(m+2)(1+\theta)}}}}
\left\langle \nabla \left(\bar u(t,\cdot)\psi_{k^\theta}(\cdot)\right)(y)-\nabla \left(\bar u(t,\cdot)\psi_{k^\theta}(\cdot)\right)(x), \phi_{(m+2)(1+\theta)}\left(k^\alpha t, k y\right)-\phi_{(m+2)(1+\theta)}\left(k^\alpha t, k x\right)\right\rangle\\
&\qquad\times \frac{w_{k^\alpha t, k x,k y}}{|y-x|^{d+\alpha}}\,\mu^{(k)}(dy)+
\langle \nabla \bar u(t, x)\psi_{k^\theta}(x), \phi_{(m+2)(1+\theta)}(k^\alpha t, k x)\rangle
\int_{B_{k^{-1}{2^{(m+2)(1+\theta)}}}^c}
\frac{w_{k^\alpha t, k x,k y}}{|y-x|^{d+\alpha}}\,\mu^{(k)}(dy)\Bigg|\\
&\le c_{23}k^{-1}\left(\int_{B_{k^{-1}2^{(m+2)(1+\theta)}}}
\frac{\left|\nabla \left(\bar u(t,\cdot)\psi_{k^\theta}(\cdot)\right)(y)-\nabla \left(\bar u(t,\cdot)\psi_{k^\theta}(\cdot)\right)(x)\right|^2}{|y-x|^{d+\alpha}}\,\mu^{(k)}(dy)\right)^{1/2}\\
&\qquad\qquad\times
\left(\int_{B_{k^{-1}2^{(m+2)(1+\theta)}}}
\frac{\left|\phi_{(m+2)(1+\theta)}\left(k^\alpha t, k y\right)-\phi_{(m+2)(1+\theta)}\left(k^\alpha t, k x\right)\right|^2w_{k^\alpha t, k x,k y}}{|y-x|^{d+\alpha}}\,\mu^{(k)}(dy)
\right)^{1/2}\\
&\quad +c_{23}k^{-1}|\phi_{(m+2)(1+\theta)}(k^\alpha t, k x)|\cdot\left(\int_{\{y\in k^{-1}\Z^d:|y-x|\ge k^\theta/2\}}\frac{1}{|y-x|^{d+\alpha}}\,\mu^{(k)}(dy)\right)\\
&\le c_{24}k^{-1}\Bigg(\left(\int_{B_{k^{-1}2^{(m+2)(1+\theta)}}}
\frac{\left|\phi_{(m+2)(1+\theta)}\left(k^\alpha t, k y\right)-\phi_{(m+2)(1+\theta)}\left(k^\alpha t, k x\right)\right|^2w_{k^\alpha t, k x,k y}}{|y-x|^{d+\alpha}}\,\mu^{(k)}(dy)
\right)^{1/2}\\
&\quad +k^{-\theta\alpha}|\phi_{(m+2)(1+\theta)}(k^\alpha t, k x)|\Bigg),
\end{align*}
where in the equality above
we used the fact $\nabla \left(\bar u(t,\cdot)\psi_{k^\theta}(\cdot)\right)(y)=\phi_{(m+2)(1+\theta)}(k^\alpha t, k y)=0$ for all $y\in B_{k^{-1}{2^{(m+2)(1+\theta)}}}^c$, and
in the first inequality we used the Cauchy-Schwarz inequality and the fact $|y-x|\ge k^\theta/2$ for every $y\in B_{k^{-1}{2^{(m+2)(1+\theta)}}}^c$ and $x\in B_{{3k^\theta}/{2}}$.   Hence, we obtain that
\begin{align*}
|I_4^{  (k)}(t,x)|
&\le c_{25}
\Bigg(
k^{-1}\Big(\int_{k^{-1}B_{2^{(m+2)(1+\theta)}}}
\frac{|\phi_{(m+2)(1+\theta)}(k^\alpha t,ky)-\phi_{(m+2)(1+\theta)}(k^\alpha t,kx)|^2 w_{k^\alpha t,kx,ky}}{|y-x|^{d+\alpha}}\mu^{ (k)}(dy)\Big)^{1/2}\I_{B_{{3k^\theta}/{2}}}(x)\\
&\qquad\quad \,\, +k^{-1}(1+|x|)^{-d-\alpha}k^{\theta d/2}\left(\int_{B_{k^\theta}}
\left|\phi_{(m+2)(1+\theta)}(k^\alpha t,ky)\right|^2\mu^{ (k)}(dy)\right)^{1/2}\I_{B_{{3k^\theta}/{2}}^c}(x)\\
&\qquad\quad\,\,
+k^{-1-\theta\alpha}\left|\phi_{(m+2)(1+\theta)}(k^\alpha t, kx)\right|\Bigg).
\end{align*}
Thus, we get that for all $k\ge 2^{m_3^*}$
\begin{align*}
 \int_0^T \int_{k^{-1}\Z^d}|I_4^{ (k)}(t,x)|^2\mu^{ (k)}(dx)dt&\le c_{26}\Bigg[
k^{-(d+2+\alpha)}\left( k^{-2\theta\alpha} + k^{\theta d}
\int_{B_{{3k^\theta}/{2}}^c}(1+|x|)^{-2d-2\alpha}\mu^{  (k)}(dx)\right)\\
&\qquad\quad\quad\times\left(\int_0^{2^{(m+2)\alpha} T}\int_{B_{2^{(m+2)(1+\theta)}}}\left|\phi_{(m+2)(1+\theta)}(t,x)\right|^2\mu(dx)dt\right)\\
&\qquad\quad +k^{-(d+2)}\int_0^{2^{(m+2)\alpha} T} \E_{t,B_{2^{(m+2)(1+\theta)}}}\left(\phi_{(m+2)(1+\theta)}(t,\cdot),\phi_{(m+2)(1+\theta)}(t,\cdot)\right) dt\Bigg]\\
&\le c_{27}k^{-(2-\alpha)+(d+\alpha)\theta}(\log^{\frac{(1+\gamma)\alpha}{(2(d-\alpha))\wedge d}}k)T,
\end{align*} where the last inequality is a consequence of \eqref{t6-1-1}.

Note that by \eqref{e:4.11}, \eqref{e:addo1} and \eqref{t3-1-6},
 \begin{align*}
 \LL_t^k v_k(t,\cdot)(x)
  & =  ( I^{  (k)}_1 (t,x) +I^{  (k)}_3 (t,x) )+I^{ (k)}_2(t,x)+I^{ (k)}_4 (t,x)\\
&=
\bar \LL \left(\bar u(t,\cdot)\psi_{k^\theta}(\cdot)\right)(x)
+k^{\alpha-1}\left\langle \nabla \left(\bar u(t,\cdot)\psi_{k^\theta}(\cdot)\right)(x), \frac{\partial}{\partial t}\phi_{(m+2)(1+\theta)}(\cdot,kx)(k^\alpha t)\right\rangle\\
&\quad
+ [( K_1^{(k)}(t,x)+K^{(k)}_3(t,x) ) + I^{ (k)}_{21}(t,x) +I^{ (k)}_4(t,x)] +I^{ (k)}_{22}(t,x) \\
&=: \bar \LL \left(\bar u(t,\cdot)\psi_{k^\theta}(\cdot)\right)(x)
+k^{\alpha-1}\left\langle \nabla \left(\bar u(t,\cdot)\psi_{k^\theta}(\cdot)\right)(x),
\frac{\partial}{\partial t}\phi_{(m+2)(1+\theta)}(\cdot,kx)(k^\alpha t)\right\rangle
+ J^{ (k)}(t,x)+I^{ (k)}_{22}(t,x).
 \end{align*}
  Combining the  above estimates for $K_1^{(k)}(t,x)$ and $K^{(k)}_3(t,x)$ in \eqref{e:4.13} and \eqref{e:4.15}
 with the estimates for  $I_{21}^{ (k)}(t,x)$ and $I_4^{(k)}(t,x)$,  we find that for
a.e. $\w\in \Omega$,
there exists $k_0^*:=2^{\max\{m_1^*,m_2^*,m_3^*\}}$ such that for all $k\ge k_0^*$
and $\theta_0\in (\alpha/d,1)$,
\begin{equation}\label{t6-1-4}\begin{split}
 \int_0^T \int_{k^{-1}\Z^d}|J^{ (k)}(t,x)|^2\mu^{  (k)}(dx)dt\le & c_{28}\max\{k^{-(2-\alpha)+(d+\alpha)\theta}
,k^{2\alpha-2+\theta\alpha-\theta_0(1+\theta) d}\}
(\log^{\frac{(1+\gamma)\alpha}{(2(d-\alpha))\wedge d}}k)
T \\
&+c_{28}\left(\pi(k^\alpha T)+\max\{k^{-2\theta_0(2-\alpha)},k^{-\theta_0 d}\}\right)T,
\end{split}\end{equation}
and, for every $g\in C_b\left([0,T];L^2(k^{-1}\Z^d;\mu^{  (k)})\right)$, it holds that
\begin{equation}\label{t6-1-5}
\begin{split}
& \left|\int_0^T \int_{k^{-1}\Z^d}I_{22}^{  (k)}(t,x)g(t,x)\mu^{ (k)}(dx)dt\right|\\
&\le c_{29} k^{-(2-\alpha)+(d+\alpha)\theta}(\log^{\frac{(1+\gamma)\alpha}{(2(d-\alpha))\wedge d}}k)T+\frac{1}{2}
\int_0^T \E_t^{(k)}\left(g(t,\cdot),g(t,\cdot)\right)(x)+\frac{1}{2}\int_0^T \|g(t,\cdot)\|_{L^2(k^{-1}\Z^d;\mu^{ (k)})}^2 dt.
\end{split}
\end{equation}
 Therefore,
we deduce from the definition of $v_k$ that
for all $t\in [0,T]$ and $x\in k^{-1}\Z^d$,
\begin{equation}\label{e:dddfff}\begin{split}
&
\frac{\partial}{\partial t}\left(u_k-v_k\right)(t,x)-\LL_t^k (u_k-v_k)(t,\cdot)(x)\\
&= h(t,x)
-\left(\frac{\partial}{\partial t}\left(\bar u(t,\cdot)\psi_{k^\theta}(\cdot)\right)(x)-\bar \LL \left(\bar u(t,\cdot)\psi_{k^\theta}(\cdot)\right)(x)\right)
\\
&\quad -k^{-1}\left\langle \frac{\partial}{\partial t}\nabla \bar u(t,x)\psi_{k^\theta}(x), \phi_{(m+2)(1+\theta)}(k^\alpha t,kx)\right\rangle+
J^{ (k)}(t,x)+I_{22}^{(k)}(t,x)\\
&=\bar \LL  \bar u(x)-\bar \LL (\bar u \psi_{k^\theta})(x)+
\frac{\partial \bar u}{\partial t}(t,x)\left(1-\psi_{k^\theta}(x)\right)\\
&\quad -k^{-1}\left\langle \frac{\partial}{\partial t}\nabla \bar u(t,x)\psi_{k^\theta}(x), \phi_{(m+2)(1+\theta)}(k^\alpha t,kx)\right\rangle+J^{  (k)}(t,x)+I_{22}^{  (k)}(t,x)\\
&=: Q^{(k)}(t,x)+I_{22}^{  (k)}(t,x),
\end{split}\end{equation}
where we used the fact that $u_k$ and $\bar u$ satisfy equations \eqref{e6-4} and \eqref{e6-3}, respectively.
By \eqref{t6-1-5},
\begin{equation}\label{t6-1-6}
\begin{split}
&\int_0^T \int_{k^{-1}\Z^d}|Q^{ (k)}(t,x)|^2\mu^{ (k)}(dx)dt\\
&\le
c_{30}\Bigg(\int_0^T \int_{k^{-1}\Z^d}|J^{ (k)}(t,x)|^2\mu^{ (k)}(dx)dt+\int_0^T  \int_{k^{-1}\Z^d}
|\bar \LL  \bar u(t,\cdot)(x)-\bar \LL (\bar u(t,\cdot) \psi_{k^\theta}(\cdot))(x)|^2\mu^{ (k)}(dx)dt\\
&\qquad\qquad +\int_0^T\int_{k^{-1}\Z^d}\left|\frac{\partial \bar u}{\partial t}(t,x)\right|^2(1-\psi_{k^\theta}(x))^2\mu^{  (k)}(dx)dt\\
&\qquad\qquad +k^{-(d+2+\alpha)}\int_0^{2^{(m+2)\alpha}T}
\int_{B_{2^{(m+2)(1+\theta)}}}
|\phi_{(m+2)(1+\theta)}(t,x)|^2\mu(dx)dt\Bigg) \\
&\le c_{31}\max\{k^{-(2-\alpha)+(d+\alpha)\theta}
,k^{2\alpha-2+\theta\alpha-\theta_0(1+\theta) d},k^{-\theta(d+2\beta)},
{k^{-2\theta_0(2-\alpha)},k^{-\theta_0 d}}\}
(\log^{\frac{(1+\gamma)\alpha}{(2(d-\alpha))\wedge d}} k)T +c_{31}\pi\left(k^\alpha T\right)T,
\end{split}
\end{equation}
which the last inequality is due to
\eqref{l6-1-1b},
\eqref{t6-1-1}, \eqref{t6-1-2} and \eqref{t6-1-4}.

Multiplying $u_k-v_k$ on both sides of \eqref{e:dddfff} and taking the integration with respect to
$\mu^{ (k)}$
yield that
\begin{align*}
&\frac{1}{2}\frac{\partial}{\partial t}\left(\|u_k(t,\cdot)-v_k(t,\cdot)\|_{L^2(k^{-1}\Z^d;\mu^{ (k)})}^2\right)
+\E_t^{(k)}\left((u_k-v_k)(t,\cdot), (u_k-v_k)(t,\cdot)\right)\\
&=\int_{k^{-1}\Z^d}I_{22}^{  (k)}(t,x)(u_k(t,x)-v_k(t,x))\mu^{ (k)}(dx)+\int_{k^{-1}\Z^d}Q^{  (k)}(t,x)(u_k(t,x)-v_k(t,x))\mu^{ (k)}(dx)\\
&\le \int_{k^{-1}\Z^d}I_{22}^{ (k)}(t,x)(u_k(t,x)-v_k(t,x))\mu^{ (k)}(dx)+\int_{k^{-1}\Z^d}|Q^{ (k)}(t,x)|^2\mu^{ (k)}(dx)+
\frac{1}{4}\|u_k(t,\cdot)-v_k(t,\cdot)\|_{L^2(k^{-1}\Z^d;\mu^{ (k)})}^2.
\end{align*}
Hence, by \eqref{t6-1-5} and \eqref{t6-1-6}, for a.e. $\w\in \Omega$ and all $k\ge k_0^*(\w)$,
\begin{align*}
&\sup_{t\in [0,T]}\|u_k(t,\cdot)-v_k(t,\cdot)\|^2_{L^2(k^{-1}\Z^d;\mu^{  (k)})}
+
\int_0^T \E_t^{(k)}\left((u_k-v_k)(t,\cdot), (u_k-v_k)(t,\cdot)\right)\,dt\\
&\le \frac{1}{2}\int_0^T \E_t^{(k)}\left((u_k-v_k)(t,\cdot), (u_k-v_k)(t,\cdot)\right)\,dt+
\frac{3}{4}\int_0^T \|u_k(t,\cdot)-v_k(t,\cdot)\|^2_{L^2(k^{-1}\Z^d;\mu^{ (k)})}\,dt\\
&\quad+ c_{32}\max\{k^{-(2-\alpha)+(d+\alpha)\theta}
,k^{2\alpha-2+\theta\alpha-\theta_0(1+\theta) d},k^{-\theta(d+2\beta)},
  k^{-2\theta_0(2-\alpha)},k^{-\theta_0 d}\}
(\log^{\frac{(1+\gamma)\alpha}{(2(d-\alpha))\wedge d}} k)T
+c_{32}\pi\left(k^\alpha T\right)T,
\end{align*}
where we used the fact $u_k(0,x)=v_k(0,x)=\bar u(0,x)=g(x)$ (also thanks to $\phi_{(m+2)(1+\theta)}(0,x)=0$ for all
$x\in \Z^d$).
Choosing $\theta=\frac{2-\alpha}{2d+\alpha+2\beta}$ and $\theta_0\in ({\alpha}/{d},1)$ close to $1$ (note that $d>\alpha$), we obtain
\begin{align*}
\max\{k^{-(2-\alpha)+(d+\alpha)\theta}
,k^{2\alpha-2+\theta\alpha-\theta_0(1+\theta) d},k^{-\theta(d+2\beta)},
k^{-2\theta_0(2-\alpha)},k^{-\theta_0 d} \}
\le k^{-\frac{(2-\alpha)(d+2\beta)}{2d+\alpha+2\beta}}.
\end{align*}
This, along with the differential inequality above, we have for all $k\ge k_0^*$,
\begin{align*}
\sup_{t\in [0,T]}\|u_k(t,\cdot)-v_k(t,\cdot)\|^2_{L^2(k^{-1}\Z^d;\mu^{ (k)})}\le c_{33}\left(
k^{-\frac{(2-\alpha)(d+2\beta)}{2d+\alpha+2\beta}}(\log^{\frac{(1+\gamma)\alpha}{(2(d-\alpha))\wedge d}} k)
+\pi\left(k^\alpha T\right)\right) e^T T.
\end{align*}
Furthermore, we get
\begin{align*}
&  \|u_k(t,\cdot)-v_k(t,\cdot)\|^2_{L^2(\R^d;dx)} \\
&=\sum_{z\in k^{-1}\Z^d}\int_{\prod_{1\le i\le d}(z_i,z_i+k^{-1}]}|u_k(t,x)-v_k(t, x)|^2\,dx\\
&\le 2\sum_{z\in k^{-1}\Z^d}\int_{\prod_{1\le i\le d}(z_i,z_i+k^{-1}]}\left(|u_k(t,x)-v_k(t,z)|^2+
|v_k(t,x)-v_k(t, z)|^2\right)\,dx\\
&=2\|u_k(t,\cdot)-v_k(t,\cdot)\|^2_{L^2(k^{-1}\Z^d;\mu^{(k)})}+2\sum_{z\in k^{-1}\Z^d}\int_{\prod_{1\le i\le d}(z_i,z_i+k^{-1}]}
|v_k(t,x)-v_k(t,z)|^2\,dx\\
&\le 2\|u_k(t,\cdot)-v_k(t,\cdot)\|^2_{L^2(k^{-1}\Z^d;\mu^{(k)})}+c_{34}\bigg(k^{-2}+k^{-2}\sum_{z\in k^{-1}\Z^d}k^{-d}|\phi_{(m+2)(1+\theta)}(k^\alpha t, k z)|^2\bigg)\\
&\le 2\|u_k(t,\cdot)-v_k(t,\cdot)\|^2_{L^2(k^{-1}\Z^d;\mu^{(k)})}+c_{35}\bigg(k^{-2}+k^{-2}\oint_{B_{2^{(m+2)(1+\theta)}}}|\phi_{(m+2)(1+\theta)}(k^\alpha t, \cdot)|^2\bigg)\\
&\le c_{36}k^{-\frac{(2-\alpha)(d+2\beta)}{2d+\alpha+2\beta}}
(\log^{\frac{\alpha(1+\gamma)}{(2(d-\alpha))\wedge d}}k),
\end{align*}
where $c_{36}$ depends on $T$.
Note that
the second inequality follows from
\begin{align*}
& |v_k(t,x)-v_k(t,z)| \\
&\le c_{37}\big(|\bar u(t,x)\psi_{k^\theta}(x)-\bar u(t,z)\psi_{k^\theta}(z))|
 + k^{-1}|\nabla (\bar u(t,x)\psi_{k^\theta}(x))-\nabla (\bar u(t,z)\psi_{k^\theta}(z))|\cdot |\phi_{(m+2)(1+\theta)}(k^\alpha t, kz)|\big)\\
&\le c_{38}k^{-1}\left(
(1\vee (c_{39}|x|))^{-d-\beta}
+|\phi_{(m+2)(1+\theta)}(k^\alpha t, kz)|\right)
\qquad \hbox{for } z\in 2^{-1}\Z^d \hbox{ and } x\in \prod_{1\le i\le d}(z_i,z_i+k^{-1}].
\end{align*}
Here we used
\eqref{eq:v_k}, \eqref{t6-1-1a} and the fact
$\phi_{(m+2)(1+\theta)}(k^\alpha t, kx)=\phi_{(m+2)(1+\theta)}(k^\alpha t, kz)$
for every  $x\in \prod_{1\le i\le d}(z_i,z_i+k^{-1}]$
by the way of extending the function at the beginning of the proof;
 while
in the last inequality we used \eqref{t6-1-1} with the choice of $\theta=\frac{2-\alpha}{2d+\alpha+2\beta}$.

Combining this with \eqref{t6-1-2}, we finally complete the proof for the case $\alpha\in (1,2)$.

When $\beta=\infty$, there exists $R_0>0$ so that ${\rm supp}[\bar u(t,\cdot)]\subset B_{R_0}$ for every
$t\in [0,T]$. In this case, we can use $\bar u$ instead of $\bar u \psi_{k^\theta}$ and follow the arguments
above to prove the desired conclusion.

{\bf Case 2:} We assume that $\alpha\in (0,1]$. Let $
\widetilde\phi_m:\R_+\times\Z^d \to \R^d$ be the solution
to the equation \eqref{e6-7}. For every $2^m\le k<2^{m+1}$, we define $v_k:[0,T]\times k^{-1}\Z^d \to \R$ by
$$
v_k(t,x)=\bar u(t,x)\psi_{k^\theta}(x)+k^{-1}\left\langle \nabla (\bar u(t,x)\psi_{k^\theta}(x)),
\widetilde\phi_{(m+2)(1+\theta)}(t,kx)\right\rangle.
$$
As before, we can extend it to $v_k:[0,T]\times \R^d \to \R$ by setting $\widetilde \phi_{(m+2)(1+\theta)}(t, x)=\widetilde \phi_{(m+2)(1+\theta)}(t, z)$ if $x\in \prod_{1\le i\le d} (z_i,z_i+1]\cap B_{2^{(m+2)(1+\theta)}}$ for
the unique $z\in  B_{2^{(m+2)(1+\theta)}}\cap\Z^d$ and setting
$\widetilde \phi_{(m+2)(1+\theta)}(t,x)=0$ if $x\notin B_{2^{(m+2)(1+\theta)}}$.
Then,
using \eqref{p6-4-1}, \eqref{p6-3--} and \eqref{p6-2-1}, and following the arguments
for {\bf Case 1}, we can find $k_0^*:=k_0^*(\w)\ge1$ such that for all $k\ge k_0^*$,
\begin{align*}
&\sup_{t\in [0,T]}\|u_k(t,\cdot)-\bar u(t,\cdot)\|_{L^2(\R^d;dx)}\\
&\le c_{40}\sqrt{\pi\left(k^\alpha T\right)}+
c_{40}
\begin{cases}
k^{-\frac{d+2\beta}{2(2d+1+2\beta)}}
(\log^{1+\frac{1+\gamma}{(4(d-1))\wedge 2d}}k),
\ &\alpha=1,\\
k^{-\frac{\alpha(d+2\beta)}{2(2d+\alpha+2\beta)}}
(\log^{\frac{(1+\gamma)\alpha}{(4(d-\alpha))\wedge 2d}}k),
\ &\alpha\in (0,1).
\end{cases}
\end{align*}

Next, we are going to prove the remaining estimate in \eqref{t6-1-1aa} for $\alpha\in (0,1)$. By   \eqref{e6-3} and \eqref{e6-4},
\begin{equation}\label{t6-1-7}
\begin{split}
& \frac{\partial}{\partial t}\left(u_k-\bar u\right)(t,x) -\LL_t^k\left(u_k-\bar u\right)(t,\cdot)(x)\\
&=\left(\frac{\partial}{\partial t}u_k(t,x) -\LL_t^k u_k(t,\cdot)(x)\right)-
\left(\frac{\partial}{\partial t}\bar u(t,x) -\bar \LL \bar u(t,\cdot)(x)\right) +
\left(\LL_t^k \bar u(t,\cdot)(x)-\LL \bar u(t,\cdot)(x)\right)\\
&=\LL_t^k \bar u(t,\cdot)(x)-\LL \bar u(t,\cdot)(x):=J(t,x).
\end{split}
\end{equation}
According to \eqref{p6-3--} and \eqref{p6-3-1}, there exists $k_1^*:=k_1^*(\w)>0$ such that for all $k\ge k_1^*$,
$$
\int_0^T\int_{k^{-1}\Z^d}|J(t,x)|^2\,\mu^{  (k)}(dx)\,dt\le c_{41}
\left(\max\{k^{- \frac{2(1-\alpha)}{1+\gamma}},k^{-\frac{d}{1+\gamma}}\}
+\pi\left(k^\alpha T\right)\right)
T.
$$
Putting this into \eqref{t6-1-7} and following the the proof for {\bf Case 1} again, we can prove
for all $k\ge k_1^*$,
$$
\sup_{t\in [0,T]}\|u_k(t,\cdot)-\bar u(t,\cdot)\|_{L^2(\R^d;dx)}\le c_{42}\left(\sqrt{\pi\left(k^\alpha T\right)}+\max\{k^{- \frac{1-\alpha}{1+\gamma}},k^{-\frac{d}{2(1+\gamma)}}\}\right).
$$
This completes the proof of the theorem.
\end{proof}

 \medskip

\begin{remark}
We now explain why the convergence rate in Theorem \ref{T:1.1} contains the term $\pi (k^\alpha T)$ in the right hand side of \eqref{t6-1-1aa}. When $\sK (t)\equiv 0$, clearly $\pi (k^\alpha T)=0$.  One can do a deterministic time-change to reduce
the general case to the case that $\sK(t) \equiv 0$ as follows.

Define
\begin{equation}
a(t):=\int_0^t \sK (r)\, dr \quad \hbox{and} \quad \wt \omega_{t, x, y}:=\frac{w_{a^{-1}(t), x, y}}{\sK (a^{-1}(t))},
\end{equation}
where $a^{-1}(t)$ is the inverse function of the strictly increasing continuous function $a(t)$.
  Then
$\{\wt w_{t, x, y}: t\in [0, \infty), (x, y)\in E\}$ satisfies
{\bf Assumption (H)} with $\Ee [ \wt w_{t, x, y}]  =1$
and $\wt w_{t, x, y} \leq C/K_1$ for all $t\geq 0$ and $(x, y)\in E$.
Define
\begin{equation}\label{e:4.22}
\wt \LL_t^\w f(x):=\sum_{y\in \Z^d:y\neq x}\left(f(y)-f(x)\right)\frac{\wt w_{t,x,y}(\w)}{|x-y|^{d+\alpha}},\quad f\in L^2(\Z^d;\mu),
\end{equation}
and for each integer $k\geq 1$,
\begin{equation}\label{e:4.23}
\wt \LL_t^{k,\w} f(x):=
k^{-d}\sum_{y\in k^{-1}\Z^d:y\neq x}\left(f(y)-f(x)\right)\frac{\wt w_{k^\alpha t,kx,ky}(\w)}{|x-y|^{d+\alpha}},
\quad f\in L^2(\Z^d;\mu^{(k)})
\end{equation}

 Let $T>0$,   $g\in C_c^2 (\R^d)$,  $\beta\in(0,\infty]$ and $h\in \mathscr{S}_{\beta, g,T}$.
 Denote by  $\bar v$ the solution to  \eqref{e6-3} on $[0, \sK T] \times \R^d$ but with 1 in place of $\sK$ in the definition of $\bar \sL$
 and $h (t/\sK, x)/\sK$ in place of $h (t, x)$,
 and $v^\w_k$ the solution to \eqref{e6-4}
 but with $\wt \LL_t^{k,\w} $ in place of
 $\LL_t^{k,\w}$ and $h_k(t, x):= \frac{h( k^{-\alpha} a^{-1}(k^\alpha t), x)}{\sK (a^{-1}(k^\alpha t))}$ in place of $h (t, x)$.
Then
\begin{equation}\label{e:4.24}
u^\w_k (t, x) = v^\w_k (a_k (t), x) \quad \hbox{and} \quad \bar u (t, x)= \bar v(\sK t, x),
\end{equation}
where $a_k(t):=k^{-\alpha} a (k^\alpha t)$.

Suppose that $h=0$. Denote the right hand side of \eqref{e:1.11} by $r(k)$ for $\{ \wt w_{t, x, y}
: t\geq 0 \hbox{ and } (x, y)\in E\}$,
$ v^\w_k (t, x)$ and $\bar v(t, x)$ with $\sK T$ in place of $T$.
Then, by \eqref{e:4.24}, \eqref{e:1.11} and the Lipschitz continuity of $\bar v(t, x)$ with respect to $t$ in the Hilbert space  $L^2(\R^d; dx)$, we have
\begin{align*}
& \sup_{t\in [0, T]} \| u^\omega_k (t, \cdot) -\bar u (t, \cdot)\|_{L^2 (\R^d; dx)} \\
&\leq  \sup_{t\in [0, T]} \|  v^\w_k (a_k (t), \cdot) -  \bar v(\sK t, \cdot)\|_{L^2 (\R^d; dx)} \\
&\leq  \sup_{t\in [0, T]}  \left( \|  v^\w_k (a_k (t), \cdot) -  \bar v(a_k (t), \cdot)\|_{L^2 (\R^d; dx)}
+ \|  \bar v(a_k (t), \cdot) - v(\sK t, \cdot)\|_{L^2 (\R^d; dx)}  \right) \\
&\leq  r(k) + \sup_{t\in [0, T]}  | a_k(t) -\sK t| \\
&\leq  r(k) + \sup_{t\in [0, T]}  \left| k^{-\alpha} \int_0^{k^\alpha t} ( \sK (r) -\sK)  dr \right| \\
&=  r(k) + \sup_{t\in [0, T]}   \int_0^t  | \sK (k^\alpha s) -\sK |   ds \\
&\leq  r(k) +  \int_0^T  | \sK (k^\alpha s) -\sK |   ds \\
&\leq  r(k) +  \pi (k^\alpha T) \sqrt{T}   ,
\end{align*}
where the last inequality is due to the Cauchy-Schwarz inequality.  \qed
\end{remark}

\section{Appendix}

\subsection{Dense property of $\mathscr{S}_{\infty, g,T}$}\label{section5.2} In this part, we assert that, for any $T>0$ and $g\in C_c^2(\R^d)$, the set $\mathscr{S}_{\infty, g,T}$ is dense in $L^2([0,T]\times \R^d;dt\times dx)$. Indeed, let $\varphi\in C_b^\infty(\R)$ be such that $0\le \varphi\le 1$ on $\R$ with $\varphi(r)=1$ for all $r\le 1$ and $\varphi(r)=0$ for all $r\ge 2$. Fix $T>0$ and $g\in C_c^2(\R^d)$. Given $f\in L^2([0,T]\times \R^d;dt\times dx)$, set $$g_n(t,x)=\varphi(|x|-n) \left(\bar P_tg(x)+\int_0^t \bar P_{t-s}f(s, x)\,ds\right),$$ and
$$f_n(t,x)= \frac{\partial}{\partial t}g_n(t,x)-\bar \LL g_n(t,\cdot)(x)$$ for any $(t,x)\in [0,T]\times \R^d$. Then, following the proof of \cite[Lemma A.1]{CCKW3}, we can check that
$f_n\in \mathscr{S}_{\infty, g,T}$ and
$$\lim_{n\to\infty}\|f_n-f\|_{L^2([0,T]\times \R^d;dt\times dx)}=0.$$

\subsection{Proofs of Poincar\'e-type inequalities}\label{section2}
The proofs in this part mainly follow from these of \cite[Propositions 2.2 and 2.3]{CCKW3}, respectively.

\begin{proof}[Proof of Proposition $\ref{l2-1}$]
Note that, by \eqref{l2-0-1}, there are a constant $\delta_0\in (0,1)$, $ \lambda_0>0$ and a random variable $m_0(\w)\ge 1$ such that
for all $m\ge m_0(\w)$,
\begin{equation}\label{l2-1-2}
\mu\left(B_{r,\delta_0}^{t, x_1,x_2}(y)\right)>\lambda_0 \mu(B_r(y)),\quad 0\le t \le 2^{m\alpha}T,\ x_1,x_2,y\in B_{2^m}\ \text{with}\ x_1\neq x_2,
[m^\theta]\le r \le 2^m.
\end{equation}
Thus, for all $f:\Z^d\to \R$, $m\ge m_0(\w)$, $0\le t \le 2^{m\alpha}T$, $y\in B_{2^m}$ and
$[m^\theta]\le r \le 2^m$, according to the Cauchy-Schwartz inequality, we obtain
\begin{align*}
& \int_{B_r(y)}\Big(f(x)-
\oint_{B_r(y)}f\,d\mu
\Big)^2\,\mu(dx)\\
&=
\mu(B_r(y))^{-2}\sum_{x_1\in B_r(y)}\bigg(\sum_{x_2\in B_r(y)
\setminus \{x_1\}}
(f(x_1)-f(x_2))\bigg)^2\\
&\le \mu(B_r(y))^{-1}
\sum_{x_1,x_2\in B_r(y):x_1\neq x_2}\left(f(x_1)-f(x_2)\right)^2\\
&= \mu(B_r(y))^{-1}
\sum_{x_1,x_2\in B_r(y):x_1\neq x_2}\frac{1}{\mu\left(B_{r,\delta_0}^{t, x_1,x_2}(y)\right)}
\sum_{z\in B_{r,\delta_0}^{t, x_1,x_2}(y)}\left(f(x_1)-f(z)+f(z)-f(x_2)\right)^2\\
&\le 2 \lambda_0^{-1}\mu(B_r(y))^{-2}\sum_{x_1,x_2\in B_r(y):x_1\neq x_2} \sum_{z\in B_{r,\delta_0}^{t, x_1,x_2}(y)}
\Big(\left(f(x_1)-f(z)\right)^2+ \left(f(x_2)-f(z)\right)^2\Big),
\end{align*}
where in the second inequality we used \eqref{l2-1-2}.

By the definition of $B_{r,\delta_0}^{t, x_1,x_2}(y)$, we know that $w_{t,x_1,z}>\delta_0$ and $|x_1-z|\le 2r$ for every $z\in B_{r,\delta_0}^{t, x_1,x_2}(y)$ and $x_1,x_2\in B_r(y)$
with $x_1\neq x_2$.
Thus,
\begin{align*}
&\mu(B_r(y))^{-2}\sum_{x_1,x_2\in B_r(y):x_1\neq x_2} \sum_{z\in B_{r,\delta_0}^{t, x_1,x_2}(y)}
\left(f(x_1)-f(z)\right)^2\\
&\le \delta_0^{-1}(2r)^{d+\alpha}\mu(B_r(y))^{-2}\sum_{x_1,x_2\in B_r(y):x_1\neq x_2} \sum_{z\in B_{r,\delta_0}^{t, x_1,x_2}(y)}
\left(f(x_1)-f(z)\right)^2 \frac{w_{t,x_1,z}}{|x_1-z|^{d+\alpha}}\\
&\le c_1r^{-d+\alpha}\sum_{x_1,x_2\in B_r(y):x_1\neq x_2}\sum_{z\in B_r(y)}\left(f(x_1)-f(z)\right)^2 \frac{w_{t,x_1,z}}{|x_1-z|^{d+\alpha}}\\
&\le c_2r^\alpha\sum_{x,z\in B_r(y)}\left(f(x)-f(z)\right)^2 \frac{w_{t,x,z}}{|x-z|^{d+\alpha}}=c_2r^\alpha\mathscr{E}_{t,B_r(y)}^\w(f,f).
\end{align*}
Similarly, we have
\begin{align*}
&\quad\mu(B_r(y))^{-2}\sum_{x_1,x_2\in B_r(y);x_1\neq x_2} \sum_{z\in B_{r,\delta_0}^{t, x_1,x_2}(y)}
\left(f(x_2)-f(z)\right)^2\le c_3r^\alpha\mathscr{E}_{t,B_r(y)}^\w(f,f).
\end{align*}
Putting all the estimates above together, we can obtain the desired conclusion \eqref{l2-1-1}.
\end{proof}

\begin{proof}[Proof of Proposition $\ref{l2-2}$]
For every $1\le k <m$ and $f,g: B_{2^m}\to \R$,
\begin{equation}\label{l2-2-1a}
\begin{split}
&\sum_{z\in \Z_{m,k+1}^d}\int_{B_{2^{k+1}}(z)}f(x)\bigg(g(x)-\oint_{B_{2^{k+1}}(z)} g\,d\mu\bigg)\,\mu(dx)\\
&=\sum_{z\in \Z_{m,k+1}^d}\sum_{y\in \Z_{m,k}^d\cap B_{2^{k+1}}(z) }\int_{B_{2^{k}}(y)}f(x)\bigg(g(x)-\oint_{B_{2^{k+1}}(z)} g\,d\mu\bigg)\,\mu(dx)\\
&=\sum_{z\in \Z_{m,k+1}^d}\sum_{y\in \Z_{m,k}^d\cap B_{2^{k+1}}(z) }
\Bigg[\int_{B_{2^{k}}(y)}f(x)\left(g(x)-\oint_{B_{2^{k}}(y)} g\,d\mu\right)\,\mu(dx)\\
&\qquad \qquad\qquad\qquad\qquad\qquad+
\int_{B_{2^{k}}(y)}f(x)\left(\oint_{B_{2^{k}}(y)} g\,d\mu-\oint_{B_{2^{k+1}}(z)} g\,d\mu\right)\,\mu(dx)\Bigg]\\
&=\sum_{y\in \Z_{m,k}^d}\int_{B_{2^{k}}(y)}f(x)\left(g(x)-\oint_{B_{2^{k}}(y)} g\,d\mu\right)\,\mu(dx)\\
&\quad +\mu(B_{2^k})
\sum_{z\in \Z_{m,k+1}^d}\sum_{y\in \Z_{m,k}^d\cap B_{2^{k+1}}(z) }
\left(\oint_{B_{2^k}(y)}f\,d\mu\right)\left(\oint_{B_{2^{k}}(y)} g\,d\mu-\oint_{B_{2^{k+1}}(z)} g\,d\mu\right).
\end{split}
\end{equation}

According to the Cauchy-Schwartz inequality,
\begin{equation}\label{l2-2-1ab}
\begin{split}
&\mu(B_{2^k})\sum_{z\in \Z_{m,k+1}^d}\sum_{y\in \Z_{m,k}^d\cap B_{2^{k+1}}(z) }
\left(\oint_{B_{2^k}(y)}f\,d\mu\right)\left(\oint_{B_{2^{k}}(y)} g\,d\mu-\oint_{B_{2^{k+1}}(z)} g\,d\mu\right)\\
&\le c_12^{kd}\left(\sum_{z\in \Z_{m,k+1}^d}\sum_{y\in \Z_{m,k}^d\cap B_{2^{k+1}}(z) }
\left(\oint_{B_{2^{k}}(y)} g\,d\mu-\oint_{B_{2^{k+1}}(z)} g\,d\mu\right)^2\right)^{1/2}
\left(\sum_{y\in \Z_{m,k}^d}\left(\oint_{B_{2^k}(y)}f\,d\mu\right)^2\right)^{1/2} \\
 &\le c_12^{kd}\left(\sum_{z\in \Z_{m,k+1}^d}\sum_{y\in \Z_{m,k}^d\cap B_{2^{k+1}}(z) }
  \oint_{B_{2^{k}}(y)} \left(g(x)-\oint_{B_{2^{k+1}}(z)} g\,d\mu\right)^2 \,\mu(dx)    \right)^{1/2}
\left(\sum_{y\in \Z_{m,k}^d}\left(\oint_{B_{2^k}(y)}f\,d\mu\right)^2\right)^{1/2} .
\end{split}
\end{equation}

 By \eqref{l2-1-2}, for a.e. $\w\in \Omega$, there exist constants $\delta_0,  \lambda_0 >0$ and a random variable $m_0(\w)\ge1$ so that for
every $m\ge m_0(\w)$,
\begin{equation}\label{l2-2-2}
\mu\left(B_{2^k,\delta_0}^{t,x_1,x_2}(y)\right)> \lambda_0  \mu(B_{2^k}(y)),\quad 0\le t \le 2^{m\alpha}T,\ x_1,x_2,y\in B_{2^m}\ \text{with}\ x_1\neq x_2,
\theta \log_2 (m\log 2) \le k\le m.
\end{equation}
Hence,  using the Cauchy-Schwartz inequality again, we derive
\begin{align*}
&  \oint_{B_{2^{k}}(y)} \left(g(x)-\oint_{B_{2^{k+1}}(z)} g\,d\mu\right)^2 \,\mu(dx)\\
&=\mu(B_{2^k}(y))^{-1}\mu(B_{2^{k+1}}(z))^{-2}\sum_{x_1\in B_{2^k}(y)}
\left(\sum_{x_2\in B_{2^{k+1}}(z):x_2\neq x_1}(g(x_1)-g(x_2))\right)^2\\
&\le c_22^{-2kd}\sum_{x_1\in B_{2^k}(y), x_2\in B_{2^{k+1}}(z), x_1\neq x_2}\left(g(x_1)-g(x_2)\right)^2\\
&= c_22^{-2kd}\sum_{x_1\in B_{2^k}(y), x_2\in B_{2^{k+1}}(z), x_1\neq x_2}\mu\left(B_{2^k,\delta_0}^{t,x_1,x_2}(y)\right)^{-1}\sum_{z_1\in B_{2^k,\delta_0}^{t,x_1,x_2}(y)}\left(g(x_1)-g(z_1)+g(z_1)-g(x_2)\right)^2\\
&\le c_32^{-3kd}\sum_{x_1\in B_{2^k}(y), x_2\in B_{2^{k+1}}(z), x_1\neq x_2}
\sum_{z_1\in B_{2^k,\delta_0}^{t,x_1,x_2}(y)}\Big(\left(g(x_1)-g(z_1)\right)^2+\left(g(z_1)-g(x_2)\right)^2\Big),
\end{align*}
where the last inequality follows from \eqref{l2-2-2}. Note that
$w_{t,x_1,z_1}>\delta_0$, $w_{t,x_2,z_1}>\delta_0$, $|x_1-z_1|\le 2^{k+1}$ and $|x_2-z_1|\le 2^{k+3}$
for all $z_1\in B_{r,\delta_0}^{t,x_1,x_2}(y)$, $x_1\in B_{2^k}(y)$, $x_2\in B_{2^{k+1}}(z)$, $z\in \Z_{m,k+1}^d$ and
$y\in \Z_{m,k}^d\cap B_{2^{k+1}}(z).$
We further get
\begin{align*}
&2^{-3kd}\sum_{x_1\in B_{2^k}(y), x_2\in B_{2^{k+1}}(z), x_1\neq x_2}
\sum_{z_1\in B_{2^k,\delta_0}^{t,x_1,x_2}(y)}\left(g(x_1)-g(z_1)\right)^2\\
&\le 2^{-3kd}2^{(k+1)(d+\alpha)}\delta_0^{-1}\sum_{x_1\in B_{2^k}(y), x_2\in B_{2^{k+1}}(z), x_1\neq x_2}
\sum_{z_1\in B_{2^k,\delta_0}^{t,x_1,x_2}(y)}\left(g(x_1)-g(z_1)\right)^2 \frac{w_{t,x_1,z_1}}{|x_1-z_1|^{d+\alpha}}\\
&\le c_4 2^{-k(2d-\alpha)}\sum_{x_1\in B_{2^k}(y), x_2\in B_{2^{k+1}}(z), x_1\neq x_2}
\sum_{z_1\in B_{2^k}(y)}\left(g(x_1)-g(z_1)\right)^2 \frac{w_{t,x_1,z_1}}{|x_1-z_1|^{d+\alpha}}\\
&\le c_5 2^{-k(d-\alpha)}\sum_{x_1,z_1\in B_{2^k}(y)}\left(g(x_1)-g(z_1)\right)^2 \frac{w_{t,x_1,z_1}}{|x_1-z_1|^{d+\alpha}}.
\end{align*}
Similarly, we can obtain
\begin{align*}
&\quad 2^{-3kd}\sum_{x_1\in B_{2^k}(y), x_2\in B_{2^{k+1}}(z), x_1\neq x_2}
\sum_{z_1\in B_{2^k,\delta_0}^{t,x_1,x_2}(y)}\left(g(x_1)-g(z_1)\right)^2\\
&\le c_6 2^{-k(d-\alpha)}\sum_{x_2\in B_{2^{k+1}}(z),z_1\in B_{2^k}(y)}\left(g(x_2)-g(z_1)\right)^2 \frac{w_{t,x_2,z_1}}{|x_2-z_1|^{d+\alpha}}.
\end{align*}
Thus, for every $m\ge m_0(\w)$, $0\le t \le 2^{m\alpha}T$ and
$k \in \Z \cap [ \theta \log_2 (m\log 2) , m]$,
 \begin{align*}
\oint_{B_{2^{k}}(y)} \left(g(x)-\oint_{B_{2^{k+1}}(z)} g\,d\mu\right)^2 \,\mu(dx)\le
c_72^{-k(d-\alpha)}\int_{B_{2^k}(y)}\int_{B_{2^{k+1}}(z)}(g(x_1)-g(x_2))^2\frac{w_{t,x_1,x_2}}{|x_1-x_2|^{d+\alpha}}\,\mu(dx_1)\,\mu(dx_2),
\end{align*}
and so
\begin{align*}
 & \sum_{z\in \Z_{m,k+1}^d}\sum_{y\in \Z_{m,k}^d\cap B_{2^{k+1}}(z) }
\oint_{B_{2^{k}}(y)} \left(g(x)-\oint_{B_{2^{k+1}}(z)} g\,d\mu\right)^2 \,\mu(dx) \\
 &\le c_72^{-k(d-\alpha)}\sum_{z\in \Z_{m,k+1}^d}\sum_{y\in \Z_{m,k}^d\cap B_{2^{k+1}}(z) }
\int_{B_{2^k}(y)}\int_{B_{2^{k+1}}(z)}(g(x_1)-g(x_2))^2\frac{w_{t,x_1,x_2}}{|x_1-x_2|^{d+\alpha}}\,\mu(dx_1)\,\mu(dx_2)\\
&\le c_72^{-k(d-\alpha)}
\int_{B_{2^m}}\int_{B_{2^m}}(g(x_1)-g(x_2))^2\frac{w_{t,x_1,x_2}}{|x_1-x_2|^{d+\alpha}}\,\mu(dx_1)\,\mu(dx_2)=2c_72^{-k(d-\alpha)}
\mathscr{E}_{t,B_{2^m}}^\w(g,g).
\end{align*}

Combining this with \eqref{l2-2-1a} and \eqref{l2-2-1ab}, and taking the summation in \eqref{l2-2-1a} from $k=n$
(with $m\ge m_0(\w)$ and
$n\in \Z \cap [ \theta \log_2 (m\log 2 ) ,  m]  $)
to $k=m-1$,
we get  the desired assertion \eqref{l2-2-1}.
\end{proof}

\ \

\noindent {\bf Acknowledgements.}\,\,  The research of Xin Chen is supported by the National Natural Science Foundation of China
(No.\ 12122111).
The research of Zhen-Qing Chen is partially supported by Simons Foundation grant 962037. The research of Takashi Kumagai is supported by JSPS
KAKENHI Grant Number
JP22H00099 and 23KK0050.
The research
of Jian Wang is supported by the NSF of China the National Key R\&D Program of China (2022YFA1006003) and the National Natural Science
Foundation of China (Nos. 12225104 and 12531007).


\begin{thebibliography}{99}

 \bibitem{ABDH} S. Andres, M.T. Barlow, J.-D. Deuschel and B.M. Hambly: Invariance principle for the random conductance model, {\it Probab. Theory Related Fields}, \textbf{156} (2013), 535--580.


 \bibitem{ACJS} S. Andres, A. Chiarini, J.-D. Deuschel and M. Slowik: Quenched invariance principle for random walks with time-dependent ergodic degenerate weights, {\it Ann. Probab.}, \textbf{46}  (2018), 302--36.



 \bibitem{ACS} S. Andres, A. Chiarini and M. Slowik: Quenched local limit theorem for random walks among time-dependent ergodic degenerate weights, {\it Probab. Theory Related Fields}, \textbf{179} (2021), 1145--1181.



\bibitem{ABM}  S. Armstrong, A. Bordas and J.-C. Mourrat:
Quantitative stochastic homogenization and regularity theory of parabolic equations,
{\it Anal. PDE}, \textbf{11} (2018), 1945--2014.

\bibitem{ABK} S. Armstrong, A. Bou-Rabee and T. Kuusi:
Superdiffusive central limit theorem for a Brownian particle in a critically-correlated incompressible random drift,
arXiv:2404.01115


\bibitem{AD}  S. Armstrong and P. Dario:
Elliptic regularity and quantitative homogenization on percolation clusters,
{\it Comm. Pure Appl. Math.}, \textbf{71} (2018), 1717--1849.

\bibitem{AK} S. Armstrong and T. Kuusi:
Renormalization group and elliptic homogenization in high contrast, {\it Invent. Math.}, \textbf{242} (2025), 895--1086.

\bibitem{AKM} S. Armstrong, T. Kussi and J.-C. Mourrat:
{\it Quantitative Stochastic Homogenization and Large-scale Regularity}, Grundlehren der mathematischen Wissenschaften, vol. {\bf 352}, Springer, Cham, 2019.

\bibitem{AKM1}
S. Armstrong, T. Kuusi and J.-C. Mourrat:
The additive structure of elliptic homogenization, {\it Invent. Math.}, \textbf{208} (2017), 999--1154.

\bibitem{AKM2} S. Armstrong, T. Kuusi and J.-C. Mourrat:
Mesoscopic higher regularity and subadditivity in elliptic homogenization,
{\it Comm. Math. Phys.}, \textbf{347} (2016), 315--361.

\bibitem{AM}
S. Armstrong and J.-C. Mourrat: Lipschitz regularity for elliptic equations with random coefficients, {\it Arch. Ration. Mech. Anal.}, \textbf{219} (2016), 255--348.


\bibitem{AS} S. Armstrong and C.K. Smart:
Quantitative stochastic homogenization of convex integral functionals,
{\it Ann. Sci. \'Ec. Norm. Sup\'er.}, \textbf{48} (2016), 423--481.

\bibitem{AW} S. Armstrong and W. Wu: $C^2$ regularity of the surface tension for the  $\nabla \phi$  interface model,
{\it Comm. Pure Appl. Math.}, \textbf{75} (2022), 349--421.




\bibitem{Bar}
M.T. Barlow: Random walks on supercritical percolation clusters, {\it Ann. Probab.}, \textbf{32} (2004),
3024--3084.



\bibitem{Bau}
H. Bauer: {\it Probability Theory}, Translated by: R.B. Burckel, De Gruyter Studies in Mathematics, vol. {\bf23}, 1995.

\bibitem{BDR}
B. Bercu, B. Delyon and E. Rio: {\it Concentration Inequalities for Sums and Martingales}, Springer, Cham, 2015.


\bibitem{B} M. Biskup: Recent progress on the random conductance model, {\it Prob. Surv.}, \textbf{8} (2011), 294--373.


\bibitem{BR} M. Biskup and P.-F. Rodriguez:
Limit theory for random walks in degenerate time-dependent random environments,
{\it J. Funct. Anal.}, \textbf{274} (2018), 985--1046.


\bibitem{CMOW} G. Chatzigeorgiou, P. Morfe, F. Otto and L.-H. Wang:
The Gaussian free-field as a stream function: asymptotics of effective diffusivity in infra-red cut-off,
{\it Ann. Probab.}, \textbf{53} (2025), 1510--1536.


\bibitem{CCKW2} X. Chen, Z.-Q. Chen, T. Kumagai and J. Wang:
Homogenization of symmetric stable-like processes in stationary ergodic media,
{\it SIAM J. Math. Anal.}, \textbf{53} (2021), 2957--3001.




\bibitem{CCKW3} X. Chen, Z.-Q. Chen, T. Kumagai and J. Wang: Quantitative stochastic homogenization for random conductance models with stable-like jumps,
{\it Probab. Theory Relat. Fields}, \textbf{191} (2025), 627--669.


\bibitem{CKW20}
X. Chen, T. Kumagai and J. Wang:
Random conductance models with stable-like jumps: heat kernel estimates and Harnack inequalities,
{\it J. Funct. Anal.}, \textbf{279} (2020), paper no. 108656.

\bibitem{CKW21} X. Chen, T. Kumagai and J. Wang:
Random conductance models with stable-like jumps: Quenched invariance principle,
{\it Ann. Appl. Probab.}, \textbf{31} (2021), 1180--1231.






\bibitem{CK08} Z.-Q. Chen and T. Kumagai:
Heat kernel estimates for jump processes of mixed types on metric measure spaces, {\it Probab. Theory Relat. Fields}, \textbf{140} (2008), 277--317.

\bibitem{CKK} Z.-Q. Chen, K. Kim and T. Kumagai:
Discrete approximation of symmetric jump processes on metric measure spaces,
{\it Probab. Theory Related Fields}, \textbf{155} (2013), 703--749.

\bibitem{DG}  P. Dario and C. Gu:
Quantitative homogenization of the parabolic and elliptic Green's functions on percolation clusters,
{\it Ann. Probab.}, \textbf{49} (2021), 556--636.

\bibitem{DG1}  J.-D. Deuschel and X. Guo:
Quenched local central limit theorem for random walks in a time-dependent balanced random environment,
{\it Probab. Theory Related Fields}, \textbf{182} (2022), 111--156.

\bibitem{DGR} J.-D. Deuschel, X. Guo and A. Ram\'irez:
Quenched invariance principle for random walk in time-dependent balanced random environment,
{\it Ann. Inst. Henri Poincar\'e Probab. Stat.}, \textbf{54} (2018), 363--384.



\bibitem{F} B. Fehrman: Stochastic homogenization with space-time ergodic divergence-free drift,
{\it Ann. Probab.}, \textbf{52} (2024), 350--380.


\bibitem{FH}
F. Flegel and M. Heida: The fractional $p$-Laplacian emerging from homogenization of the random conductance model with
degenerate ergodic weights and unbounded-range jumps, {\it Calc. Var. Partial Differential Equations}, \textbf{59}, 2020, paper no. 8.


\bibitem{Fri}
A. Friedman:  {\it Stochastic Differential Equations and Applications}, {vol.\, {\bf1}}, Academic Press, Harcourt Brace Jovanovich Publishers,, New York, 1975.


\bibitem{FGW}
T. Funaki, C.-L. Gu and H. Wang:
Quantitative homogenization and hydrodynamic limit of non-gradient exclusion process, arXiv:2404.12234



\bibitem{GS}
J. Geng and Z. Shen: Convergence rates in parabolic homogenization with time-dependent periodic coefficients, {\it J. Funct. Anal.}, \textbf{272}, 2017, 2092--2113.



\bibitem{GGM} A. Giunti, C.-L. Gu and J.-C. Mourrat:
Quantitative homogenization of interacting particle systems, {\it Ann. Probab.}, \textbf{50} (2022), 1885--1946.


\bibitem{GM} A. Giunti and J.-C. Mourrat:
Quantitative homogenization of degenerate random environments,
{\it Ann. Inst. Henri Poincar\'e Probab. Stat.}, \textbf{54} (2018), 22--50.



\bibitem{GNO}  A. Gloria, S. Neukamm and F. Otto:
Quantification of ergodicity in stochastic homogenization: optimal bounds via spectral gap on Glauber dynamics,
{\it Invent. Math.}, \textbf{199} (2015), 455--515.


\bibitem{Go1} A. Gloria and F. Otto: An optimal variance estimate in stochastic homogenization of discrete elliptic equations, {\it Ann. Probab.}, \textbf{39} (2011), 779--856.

\bibitem{Go2} A. Gloria and F. Otto: An optimal error estimate in stochastic homogenization of discrete elliptic equations, {\it Ann. Appl. Probab.}, \textbf{22} (2012),  1--28.

\bibitem{Go3} A. Gloria and F. Otto: Quantitative results on the corrector equation in stochastic homogenization,
{\it J. Eur. Math. Soc.}, \textbf{19} (2017), 3489--3548.

\bibitem{GM2}  Y. Gu and J.-C. Mourrat:
Pointwise two-scale expansion for parabolic equations with random coefficients,
{\it Probab. Theory Related Fields}, \textbf{166} (2016), 585--618.

\bibitem{KLS}
C. Kenig, F. Lin and Z. Shen: Convergence rates in $L_2$ for elliptic homogenization problems, {\it Arch. Ration.
Mech. Anal.}, \textbf{203}, 2012, 1009--1036.


\bibitem{KPP}  M. Kleptsyna, A. Piatnitski and A. Popier:
Homogenization of random parabolic operators. Diffusion approximation,
{\it Stochastic Process. Appl.}, \textbf{125} (2015), 1926--1944.

\bibitem{KPP1} M. Kleptsyna, A. Piatnitski and A. Popier:
Higher order homogenization for random non-autonomous parabolic operators,
{\it Stoch. Partial Differ. Equ. Anal. Comput.}, \textbf{12} (2024), 2151--2180.

\bibitem{KV}
C. Kipnis and S.R.S. Varadhan:  Central limit theorem for additive functionals of reversible Markov processes and
application to simple exclusion, {\it Ann. Probab.}, \textbf{28} (2000), 277--302.


\bibitem{Ko}
S.M. Kozlov: The averaging of random operators, {\it Mat. Sb. (N.S.)}, \textbf{109} (1979), 188--202.

\bibitem{Ku}
T.Kumagai: {\it Random Walks on Disordered Media and Their Scaling Limits}. Lecture Notes in Mathematics/Ecole d'Ete de Probabilites de Saint-Flour, vol. {\bf 2101}, Springer, Berlin, 2014.


\bibitem{LS}
P.-L. Lions and P.E. Souganidis: Homogenization of viscous Hamilton-Jacobi equations
in stationary ergodic media, {\it Communications in Partial Differential Equations}, \textbf{30} (2005), 335--375.

\bibitem{MOW}
P. Morfe, F. Otto and C. Wagner:
The Gaussian free-field as a stream function: continuum version of the scale-by-scale homogenization result, arXiv:2404.00709



\bibitem{PV}
G.C. Papanicolaou and S.R.S. Varadhan: Boundary value problems with rapidly oscillating random coefficients. In: {\it Random Fields}, vol. I, II (Esztergom, 1979), volume 27 of Colloq. Math. Soc. J\'anos Bolyai, pp. 835--873. North-Holland, Amsterdam-New York, 1981.


\bibitem{PZ} A. Piatnitski and E. Zhizhina:
Homogenization of non-autonomous operators of convolution type in periodic media,
{\it Markov Process. Related Fields}, \textbf{29} (2023), 173--188.

\bibitem{PZ1} A. Piatnitski and E. Zhizhina:
Homogenization of non-autonomous evolution problems for convolution type operators in randomly evolving media,
{\it J. Math. Pures Appl.}, \textbf{194} (2025), Paper no. 103660, 30 pp.


\bibitem{Yu}
V.V. Yurinski: Averaging of symmetric diffusion in a random medium, {\it Sibirsk. Mat. Zh.}, \textbf{27} (1986), 167--180.


\end{thebibliography}
\end{document}